\documentclass[11pt]{extarticle}

\usepackage[margin=.7in]{geometry}
\DeclareMathAlphabet{\mathpzc}{OT1}{pzc}{m}{it}
\usepackage{amsmath,mathtools,mathrsfs,amsthm,amssymb,amsfonts,graphicx,enumerate,tikz-cd,hyperref,wasysym,float,colonequals,float,stmaryrd}
\usepackage{setspace}
\usepackage[shortlabels]{enumitem}
\usepackage[T1]{fontenc}
\usepackage{tgpagella}
\usetikzlibrary{decorations.pathmorphing}
\usepackage[parfill]{parskip}
\usepackage[toc]{appendix}
\usepackage[noabbrev,capitalize,nameinlink]{cleveref}
\usepackage{caption}
\def\p{\partial}
\def\a{\alpha}
\def\b{\beta}
\def\g{\gamma}

\def\D{\Delta}

\def\m{\mu}

\def\n{\nu}

\def\rh{\rho}
\def\s{\sigma}

\def\ph{\phi}

\def\vp{\varphi}

\def\lr{\longrightarrow}

\def\olin{\overline}

\def\es{\varnothing}
\def\sm{\setminus}

\def\sube{\subseteq}

\def\tx{\text}

\def\opn\operatorname
\def\ua{\uparrow}

\def\Ho{\mathrm{Ho}}
\def\ob{\mathrm{ob}}

\def\holim{\operatorname{holim}}
\def\dgcat{\mathrm{dgcat}}

\def\mc{\mathcal}

\def\zb{\mathbb Z}

\def\id{\mathrm{id}}

\def\bb{\begin{bmatrix}}
\def\eb{\\\end{bmatrix}}
\def\bp{\begin{proof}}
\def\ep{\end{proof}}
\def\be{\begin{equation}}
\def\ee{\end{equation}}
\def\ba{\begin{align*}}
\def\ea{\end{align*}}
\def\ben{\begin{enumerate}}
\def\een{\end{enumerate}}
\def\bv{\begin{verbatim}}
\def\ev{\end{verbatim}}
\def\bc{\begin{cases}}
\def\ec{\end{cases}}

\newcommand{\textbfit}[1]{\textbf{\textit{#1}}}

\makeatletter\newcommand*\bcd{\mathpalette\bigcdot@{.73}}\newcommand*\bigcdot@[2]{\mathbin{\vcenter{\hbox{\scalebox{#2}{$\m@th#1\bullet$}}}}}\makeatother%adjustable dot

\title{ \textbf{A Homotopical Invariant of Weinstein Surfaces}}
\author{Shanon J. Rubin\footnote{The long-form thesis version of this paper can be found through the author's website: \texttt{sites.google.com/view/shanonrubin}.}\\
\texttt{rubin@mail.tsinghua.edu.cn}\\
\emph{Yau Mathematical Sciences Center, Tsinghua University}}
\date{}

\theoremstyle{plain}
\newtheorem{theorem}{Theorem}[subsection]
\newtheorem{lemma}[theorem]{Lemma}
\newtheorem{proposition}[theorem]{Proposition}

\newtheorem{corollary}[theorem]{Corollary}
\theoremstyle{definition}
\newtheorem{definition}[theorem]{Definition}
\newtheorem{remark}[theorem]{Remark}
\newtheorem{example}[theorem]{Example}
\newtheorem{construction}[theorem]{Construction}

\begin{document}

\maketitle

\begin{abstract}
\noindent 
One generally expects that the techniques of arboreal singularities and gluing of local differential graded categories will result in a useful global invariant for all Weinstein manifolds. In this paper we construct explicit models for the homotopy limits of diagrams of microlocal sheaf categories which arise from Weinstein surfaces with arboreal skeleta. This is done by characterizing all relevant Reedy model structures on the categories of diagrams that we care about. We prove invariance using a complete set of moves for Weinstein homotopies in this setting. Finally we give combinatorial presentations of the invariant for all topological surfaces.
\end{abstract}
\frenchspacing

\tableofcontents

\section{Introduction}

\pagenumbering{arabic}

This paper covers two main topics: the development of homotopy-theoretic tools and their applications to the model category of differential-graded (dg) categories, and the further application of these to the calculation of microlocally-inspired invariants of Weinstein surfaces. In this introduction we outline our main results and the geometric motivations for our definitions.

\subsection{Overview}
One goal of symplectic topology is to understand \textit{Weinstein manifolds}, which are open symplectic manifolds with extra Morse-theoretic structure that allows for a symplectic version of handle decompositions. Weinstein manifolds are often studied via Floer theory and categorical invariants like the Fukaya category first defined in \cite{F}.

Recently, the community has also been interested in producing categories in terms of the microlocal theory of sheaves developed by Kashiwara and Schapira in \cite{KS2}. In \cite{Kon}, Kontsevich proposed that one can calculate such invariants combinatorially from the \textit{skeleton} $\mathfrak X$ of a Weinstein manifold if the singularities of $\mathfrak X$ fall into a particularly simple class. In \cite{N1}, Nadler proposed instead a larger class of combinatorial singularities of $\mathfrak X$, calling these \textit{arboreal singularities} and calculating local microlocal invariants for each singularity type.

Starkston showed in \cite{S} that (signed) arboreal singularities serve as the local models for a large family of skeleta of Weinstein manifolds, in particular for all Weinstein manifolds in dimensions 2 and 4. Further results in this direction were proven in \cite{A}.

It has been a long-term goal in our field to use the microlocal theory of sheaves to define a sheaf of categories on any Weinstein manifold $W$, the global sections of which recover the Fukaya category of $W$. See e.g. \cite{NS, N2, Sh}. The concern of this paper is to use Nadler's local description to define a category associated to a Weinstein manifold whose skeleton contains multiple arboreal singularities.

In particular, the only arboreal singularity in dimension 2 is a trivalent vertex formed by the core of a Weinsein 1-handle being incident to a copy of $D^{1}$ formed by a family of generalized Morse-Bott singularities of index 0. Nadler's theorem in this case can be summarized by the following local diagram of dg-categories, where the vertical functor corresponds to the 1-handle.  % https://q.uiver.app/#q=WzAsNSxbMSwyLCJcXG1hdGhybXtNb2R9KEFfMikiXSxbMSwxLCJcXG1hdGhybXtNb2R9KEFfMSkiXSxbMCwyLCJcXG1hdGhybXtNb2R9KEFfMSkiXSxbMSwwXSxbMiwyLCJcXG1hdGhybXtNb2R9KEFfMSkiXSxbMCwxLCJcXHRleHR0dHtDb25lfSJdLFswLDIsIlxcdGV4dHR0e2RvbWFpbn0iLDJdLFswLDQsIlxcdGV4dHR0e2NvZG9tYWlufSJdXQ==
% https://q.uiver.app/#q=WzAsNCxbMSwxLCJcXG1hdGhybXtNb2R9KEFfMikiXSxbMSwwLCJcXG1hdGhybXtNb2R9KEFfMSkiXSxbMCwxLCJcXG1hdGhybXtNb2R9KEFfMSkiXSxbMiwxLCJcXG1hdGhybXtNb2R9KEFfMSkiXSxbMCwxLCJcXHRleHR0dHtDb25lfSJdLFswLDIsIlxcdGV4dHR0e2RvbWFpbn0iLDJdLFswLDMsIlxcdGV4dHR0e2NvZG9tYWlufSJdXQ==
\be\label{localvertex}\begin{tikzcd}
	& {\mc A} \\
	{\mc A} & {\mc A^{\to}} & {\mc A}
	\arrow["{C}", from=2-2, to=1-2]
	\arrow["{\rh_{1}}"', from=2-2, to=2-1]
	\arrow["{\rh_{2}}", from=2-2, to=2-3]
\end{tikzcd}
\ee
For Nadler's setting in the microlocal theory of sheaves, $\mc A$ is the dg-category of perfect chain complexes over a field. In our context we take $\mc A$ to be any strongly pretriangulated dg-category with 2-periodic hom complexes over a commutative unital ring. The notation $\mc A^{\to}$ stands for the dg-category of $\mc A$-valued representations of the $A_{2}$ quiver, i.e. maps $A\to B$ in $Z^{0}\mc A$. The restriction functors $\rho_{1},\rho_{2},$ and $C$ pick out, respectively, the domain, codomain, and cone of such a map.

Nadler's calculation suggests that a global microlocal invariant can be constructed by gluing together the local invariants at each singularity of an arboreal skeleton. After proving that the result is invariant under a complete set of moves for arboreal skeleta, one would have a new invariant. See \cref{symp} for a more detailed sketch of this proposal.

It turns out that, even in the case $\dim W=2$, the na\"{i}ve gluing procedure is not quite invariant, and some shifts have to be added to the functors in \cref{localvertex}. \cref{symp} discusses the modifications needed, which results in our main theorem:\\

\begin{theorem}[technical statement in \cref{final}]\label{maintheorem}
Fix a strongly pre triangulated dg-category $\mc A$ with 2-periodic hom complexes. Given a Weinstein surface $W$, let $\mathfrak X^{\mathrm{arb}}$ be the skeleton of any arborealization of $W$. Let $D=D_{\mathcal A}(\mathfrak X^{\mathrm{arb}})$ be the diagram constructed from $\mathfrak X^{\mathrm{arb}}$ by the local model shown in \cref{localvertex} followed by the appropriate globally-defined shifts. Then  $\mc L(W)\coloneqq \holim D$ is a Weinstein homotopy invariant of $W$.
\end{theorem}

Readers will note the similarity between the output $\mc L_{\mc A}(W)$ for a Weinstein surface $W$ and Dyckerhoff-Kapranov's \textit{topological Fukaya category of $(S,M)$ with coefficients in $\mc A$}, defined in \cite{DK} for a stable marked oriented surface $(S,M)$. Indeed, their use of surface triangulations is dual to ours of ribbon graphs. The latter, however, generalizes more readily to Weinstein manifolds of higher dimension, and thus serves as a proof-of-concept for later work. Other differences between the two approaches are that we use dg-categories equipped with quasi-equivalences rather than Morita equivalences, and that our construction provides extremely concrete formulas, while avoiding much of the machinery used in \cite{DK}.

Future directions of this work will involve generalizing this and later results to dimension 4 or arbitrary dimension. Our results in dimension 2 represent the general approach we intend to use, but some new homotopical and geometric tools will be required for higher dimensions. Eventually, we intend to cast all statements in terms of Lurie's theory of $\infty$-categories developed in \cite{L1, L2}, as the concrete approach with dg-categories carries many technical subtleties.

Since our invariant $\mc L(W)$ is defined as a homotopy limit of dg-categories, we investigate the homotopy theory of dg-categories more broadly. Thus in \cref{hom} we sketch the homotopical theorem that is used to calculate $\mc L(W)$.

In constructing a global invariant for Weinstein manifolds in the vein of microlocal sheaf theory, this paper stands in relation to \cite{NS}, in which Nadler and Shende embed Weinstein manifolds in the standard Weinstein structure on cotangent bundles. Rather than show that a construction is independent of a choice of embedding, we show that one is independent of a choice of algebraic structure put on the poset of strata of a skeleton. In both cases, one must show that the result is independent of the choice of arborealization of a given Weinstein manifold.

\subsection{Symplectic results}\label{symp}

Given a Weinstein manifold $W$ with arboreal skeleton $\mathfrak X$, we propose a natural definition of an associated sheaf category obtained by gluing together Nadler's local sheaf categories on each singularity along the combinatorially-defined restriction functors given in \cite{N1}. We start by covering $\mathfrak X$ with \textit{arboreal charts}, open sets $U\sube \n(\mathfrak X)$ for which the pair $(U, \mathfrak X\cap U)$ is symplectomorphic to the arboreal singularity $(\mathbb R^{2n},\sf{L}_{\mc T})$ corresponding to some rooted tree $\mc T$. When $\dim W=2$, the combinatorics of $\mathfrak X$ are captured by a finite graph whose interior vertices are trivalent, along with a strict ordering of the half-edges incident to any given vertex. 

We label each singularity $\sf{L}_{\mc T}$ of $\mathfrak X$ with our replacement of its corresponding microlocal sheaf category, the dg-category of modules $\tx{Mod}(\mc T)$\footnote{This is the result of Nadler's calculation of the microlocal sheaf category of $\sf{L}_{\mc T}$ in \cite{N1}.}. That is, instead of $\mathrm{Mod}(A_{1})$ we use $\mc A$, and instead of $\mathrm{Mod}(A_{2})$ we use $\mc A^{\to}$. Including restriction functors yields a diagram of dg-categories. For example, shown below is an arboreal skeleton for the pair of pants

\begin{center}
\tikzset{every picture/.style={line width=0.75pt}} %set default line width to 0.75pt        

\begin{tikzpicture}[x=0.75pt,y=0.75pt,yscale=-1,xscale=1]\label{pantsgraph}

%Straight Lines [id:da3747371595046187] 
\draw    (70,110) -- (130.28,110) ;
\draw [shift={(130.28,110)}, rotate = 0] [color={rgb, 255:red, 0; green, 0; blue, 0 }  ][fill={rgb, 255:red, 0; green, 0; blue, 0 }  ][line width=0.75]      (0, 0) circle [x radius= 3.35, y radius= 3.35]   ;
\draw [shift={(70,110)}, rotate = 0] [color={rgb, 255:red, 0; green, 0; blue, 0 }  ][fill={rgb, 255:red, 0; green, 0; blue, 0 }  ][line width=0.75]      (0, 0) circle [x radius= 3.35, y radius= 3.35]   ;
%Straight Lines [id:da9378287641450184] 
\draw    (130.28,110) -- (190.56,110) ;
\draw [shift={(190.56,110)}, rotate = 0] [color={rgb, 255:red, 0; green, 0; blue, 0 }  ][fill={rgb, 255:red, 0; green, 0; blue, 0 }  ][line width=0.75]      (0, 0) circle [x radius= 3.35, y radius= 3.35]   ;
\draw [shift={(130.28,110)}, rotate = 0] [color={rgb, 255:red, 0; green, 0; blue, 0 }  ][fill={rgb, 255:red, 0; green, 0; blue, 0 }  ][line width=0.75]      (0, 0) circle [x radius= 3.35, y radius= 3.35]   ;
%Straight Lines [id:da06733478656686154] 
\draw    (190.56,110) -- (250.83,110) ;
\draw [shift={(250.83,110)}, rotate = 0] [color={rgb, 255:red, 0; green, 0; blue, 0 }  ][fill={rgb, 255:red, 0; green, 0; blue, 0 }  ][line width=0.75]      (0, 0) circle [x radius= 3.35, y radius= 3.35]   ;
\draw [shift={(190.56,110)}, rotate = 0] [color={rgb, 255:red, 0; green, 0; blue, 0 }  ][fill={rgb, 255:red, 0; green, 0; blue, 0 }  ][line width=0.75]      (0, 0) circle [x radius= 3.35, y radius= 3.35]   ;
%Straight Lines [id:da6816440545109562] 
\draw    (250.83,110) -- (311.11,110) ;
\draw [shift={(311.11,110)}, rotate = 0] [color={rgb, 255:red, 0; green, 0; blue, 0 }  ][fill={rgb, 255:red, 0; green, 0; blue, 0 }  ][line width=0.75]      (0, 0) circle [x radius= 3.35, y radius= 3.35]   ;
\draw [shift={(250.83,110)}, rotate = 0] [color={rgb, 255:red, 0; green, 0; blue, 0 }  ][fill={rgb, 255:red, 0; green, 0; blue, 0 }  ][line width=0.75]      (0, 0) circle [x radius= 3.35, y radius= 3.35]   ;
%Shape: Arc [id:dp9495430571818132] 
\draw  [draw opacity=0] (250.97,108.3) .. controls (250.99,108.73) and (251,109.15) .. (251,109.58) .. controls (250.88,130.02) and (223.71,146.78) .. (190.33,147.01) .. controls (156.95,147.25) and (129.99,130.86) .. (130.11,110.42) .. controls (130.11,109.98) and (130.13,109.54) .. (130.16,109.1) -- (190.56,110) -- cycle ; \draw   (250.97,108.3) .. controls (250.99,108.73) and (251,109.15) .. (251,109.58) .. controls (250.88,130.02) and (223.71,146.78) .. (190.33,147.01) .. controls (156.95,147.25) and (129.99,130.86) .. (130.11,110.42) .. controls (130.11,109.98) and (130.13,109.54) .. (130.16,109.1) ;  
%Straight Lines [id:da04525735539290798] 
\draw    (311.11,110) -- (371.39,110) ;
\draw [shift={(371.39,110)}, rotate = 0] [color={rgb, 255:red, 0; green, 0; blue, 0 }  ][fill={rgb, 255:red, 0; green, 0; blue, 0 }  ][line width=0.75]      (0, 0) circle [x radius= 3.35, y radius= 3.35]   ;
\draw [shift={(311.11,110)}, rotate = 0] [color={rgb, 255:red, 0; green, 0; blue, 0 }  ][fill={rgb, 255:red, 0; green, 0; blue, 0 }  ][line width=0.75]      (0, 0) circle [x radius= 3.35, y radius= 3.35]   ;
%Shape: Arc [id:dp24194277816374576] 
\draw  [draw opacity=0] (190.41,111.23) .. controls (190.39,110.81) and (190.38,110.38) .. (190.39,109.96) .. controls (190.67,89.52) and (217.96,72.97) .. (251.34,72.99) .. controls (284.73,73.01) and (311.56,89.6) .. (311.28,110.04) .. controls (311.27,110.48) and (311.26,110.92) .. (311.22,111.36) -- (250.83,110) -- cycle ; \draw   (190.41,111.23) .. controls (190.39,110.81) and (190.38,110.38) .. (190.39,109.96) .. controls (190.67,89.52) and (217.96,72.97) .. (251.34,72.99) .. controls (284.73,73.01) and (311.56,89.6) .. (311.28,110.04) .. controls (311.27,110.48) and (311.26,110.92) .. (311.22,111.36) ;  
\end{tikzpicture}
\end{center}

and its associated diagram
% https://q.uiver.app/#q=WzAsMTMsWzAsMSwiMCJdLFsxLDEsIlxcbWF0aGNhbCBBXzEiXSxbMiwxLCJcXG1hdGhjYWwgQV8yIl0sWzMsMSwiXFxtYXRoY2FsIEFfMSJdLFs0LDEsIlxcbWF0aGNhbCBBXzIiXSxbNSwxLCJcXG1hdGhjYWwgQV8xIl0sWzYsMSwiXFxtYXRoY2FsIEFfMiJdLFs3LDEsIlxcbWF0aGNhbCBBXzEiXSxbOCwxLCJcXG1hdGhjYWwgQV8yIl0sWzksMSwiXFxtYXRoY2FsIEFfMSJdLFsxMCwxLCIwIl0sWzQsMiwiXFxtYXRoY2FsIEFfMSJdLFs2LDAsIlxcbWF0aGNhbCBBXzEiXSxbMCwxXSxbMSwyLCJcXHJob18yIiwwLHsic3R5bGUiOnsidGFpbCI6eyJuYW1lIjoiYXJyb3doZWFkIn0sImhlYWQiOnsibmFtZSI6Im5vbmUifX19XSxbMiwzLCJcXHJob18xIl0sWzMsNCwiXFxyaG9fMSIsMCx7InN0eWxlIjp7InRhaWwiOnsibmFtZSI6ImFycm93aGVhZCJ9LCJoZWFkIjp7Im5hbWUiOiJub25lIn19fV0sWzYsNSwiXFxyaG9fMiIsMl0sWzYsNywiXFxyaG9fMSIsMl0sWzgsNywiXFxyaG9fMSJdLFs0LDUsIlxccmhvXzIiLDJdLFs4LDksIlxccmhvXzIiLDJdLFsxMCw5XSxbMiwxMSwiQyIsMV0sWzYsMTEsIkMiLDFdLFs0LDEyLCJDIiwxXSxbOCwxMiwiQyIsMV1d
\begin{small}
\be\begin{tikzcd}\label{pants}
	&&&&&& {\mc A} \\
	0 & {\mc A} & {\mc A^{\to}} & {\mc A} & {\mc A^{\to}} & {\mc A} & {\mc A^{\to}} & {\mc A} & {\mc A^{\to}} & {\mc A} & 0 \\
	&&&& {\mc A}
	\arrow[from=2-1, to=2-2]
	\arrow["{\rho_2}"', from=2-3, to=2-2]
	\arrow["{\rho_1}", from=2-3, to=2-4]
	\arrow["C"{description}, bend right =20,from=2-3, to=3-5]
	\arrow["{\rho_1}", from=2-5, to=2-4]
	\arrow["{[1]\circ C}"{description}, bend right =-20,from=2-5, to=1-7]
	\arrow["{\rho_2}"', from=2-5, to=2-6]
	\arrow["{\rho_2}"', from=2-7, to=2-6]
	\arrow["{\rho_1}", from=2-7, to=2-8]
	\arrow["{[1]\circ C}"{description},bend right =-20, from=2-7, to=3-5]
	\arrow["C"{description}, bend right =20,from=2-9, to=1-7]
	\arrow["{\rho_1}", from=2-9, to=2-8]
	\arrow["{\rho_2}"', from=2-9, to=2-10]
	\arrow[from=2-11, to=2-10]
\end{tikzcd}\ee
\end{small}

The shifts encode a Maslov index and will be explained shortly.

After a Weinstein homotopy and corresponding algebraic manipulation of the diagram associated to $\mathfrak{X}$, our results give the equivalent skeleton and diagram: \footnote{Meaning, the two diagrams have equivalent homotopy limits.}
\vspace{-0.3cm}
% https://q.uiver.app/#q=WzAsNyxbNCwwLCJcXG1hdGhjYWwgQV8xIl0sWzMsMSwiXFxtYXRoY2FsIEFfMiJdLFs0LDEsIlxcbWF0aGNhbCBBXzEiXSxbNSwxLCJcXG1hdGhjYWwgQV8yIl0sWzQsMiwiXFxtYXRoY2FsIEFfMSJdLFswLDEsIlxcY3VydmVhcnJvd3JpZ2h0Il0sWzIsMSwiXFxjdXJ2ZWFycm93cmlnaHQiXSxbMSwyLCJDIl0sWzMsMiwiQyIsMl0sWzEsMCwiXFxyaG9fMiJdLFszLDAsIlxccmhvXzEiLDJdLFsxLDQsIlxccmhvXzEiLDJdLFszLDQsIlxccmhvXzIiXSxbNSw2XV0=
\be\label{diagram}
\begin{tikzcd}
	&&&& {\mc A} \\
	{\bcd} && {\bcd} & {\mc A^{\to}} & {\mc A} & {\mc A^{\to}} \\
	&&&& {\mc A}
	\arrow["{[1]\circ C}", from=2-4, to=2-5]
	\arrow["C"', from=2-6, to=2-5]
	\arrow["{\rho_2}", bend right =-20, from=2-4, to=1-5]
	\arrow["{\rho_1}"', bend right =20, from=2-6, to=1-5]
	\arrow["{\rho_1}"', bend right =20, from=2-4, to=3-5]
	\arrow["{\rho_2}", bend right =-20, from=2-6, to=3-5]
	\arrow[no head, shorten <=-8pt, shorten >=-8pt, from=2-1, to=2-3]
		\arrow[no head, bend right = 90, shift right = -2.1, from=2-1, to=2-3]
			\arrow[no head, bend right = -90, shift right = 2, from=2-1, to=2-3]
\end{tikzcd}
\ee

Finally, we take the homotopy limit of this diagram to glue the categories together, yielding a single dg-category, up to quasi-equivalence. The singularities involved in the graph preceding \cref{pants} are the only (generalized) arboreal singularities in this dimension.

This paper achieves the above goal in dimension 2, which serves as a proof of concept for the more complicated situation in dimension 4. By Starkston's result, this gives a powerful invariant for the investigation of all Weinstein manifolds in dimension 2, and in dimension 4 if the same investigation can be done. In particular, we obtain a combinatorial way to compute a invariant which is an analogue to the Fukaya category.

For Weinstein surfaces, all Weinstein homotopies of arboreal Weinstein manifolds can be realized as particular homotopy equivalences of their skeleta.\\

\begin{lemma}[technical statement in \cref{geo}]
Given an arboreal Weinstein surface and the ribbon graph arising from its skeleton, every Weinstein homotopy can be realized by a combination of four \emph{arboreal moves}\footnote{See \cref{moves}.} on the graph.
\end{lemma}

This result uses the fact that generic families of Morse functions involve only handle-cancellations, handle-slides, and isotopies.

We now sketch the definition of the invariant $\mc L(W)$ for an arboreal Weinstein surface $W$. The skeleton of $W$ can be represented as a graph $G$ along with an immersion of $G$ in $\mathbb R^{2}$, as in the skeleton above \cref{pants}. As in \cref{pants}, we associate to $G$ a diagram $D$ of dg-categories, with specific labelings of $\{\rh_{1},\rh_{2},C\}$ determined by the orientation of $W$ and the direction corresponding to the Weinstein 1-handle.

We then add in shifts to some of the cone functors in $D$ in order to make $\holim D$ a Weinstein homotopy invariant. Invariance is easiest to prove when we descend to $\mathbb Z/2$-graded categories, so our definition need only describe which mapping cone functors have a shift of $[1]$. It turns out that the homotopy limit only depends on the \textit{difference} in the shifts of the two cone functors corresponding to each index 1 bone. Since each index 1 bone corresponds to an edge in $G$, we only need to indicate which of these edges gets an overall shift. A definition of these shifts is sketched in \cref{markedintro}. Finally, $\mc L(W)\coloneqq \holim D$ is the category referenced in \cref{maintheorem}. 

\begin{figure}
\centering
\tikzset{every picture/.style={line width=0.75pt}} %set default line width to 0.75pt        

\begin{tikzpicture}[x=0.75pt,y=0.75pt,yscale=-1,xscale=1]
%uncomment if require: \path (0,143); %set diagram left start at 0, and has height of 143

%Straight Lines [id:da801530065979462] 
\draw [line width=1.5]    (34,78) -- (244,78) ;
\draw [shift={(244,78)}, rotate = 0] [color={rgb, 255:red, 0; green, 0; blue, 0 }  ][fill={rgb, 255:red, 0; green, 0; blue, 0 }  ][line width=1.5]      (0, 0) circle [x radius= 2.61, y radius= 2.61]   ;
\draw [shift={(34,78)}, rotate = 0] [color={rgb, 255:red, 0; green, 0; blue, 0 }  ][fill={rgb, 255:red, 0; green, 0; blue, 0 }  ][line width=1.5]      (0, 0) circle [x radius= 2.61, y radius= 2.61]   ;
\draw  [line width=1.5]  (162.11,72) -- (174,78) -- (162.11,84) ;
%Straight Lines [id:da17696952549930367] 
\draw [line width=1.5]    (324,78) -- (424,78) ;
\draw [shift={(424,78)}, rotate = 0] [color={rgb, 255:red, 0; green, 0; blue, 0 }  ][fill={rgb, 255:red, 0; green, 0; blue, 0 }  ][line width=1.5]      (0, 0) circle [x radius= 2.61, y radius= 2.61]   ;
\draw [shift={(324,78)}, rotate = 0] [color={rgb, 255:red, 0; green, 0; blue, 0 }  ][fill={rgb, 255:red, 0; green, 0; blue, 0 }  ][line width=1.5]      (0, 0) circle [x radius= 2.61, y radius= 2.61]   ;
\draw  [line width=1.5]  (91,72) -- (102.89,78) -- (91,84) ;
\draw  [line width=1.5]  (351.11,72) -- (363,78) -- (351.11,84) ;
\draw  [line width=1.5]  (381.11,72) -- (393,78) -- (381.11,84) ;
%Curve Lines [id:da5604597089770382] 
\draw [color={rgb, 255:red, 240; green, 0; blue, 0 }  ,draw opacity=1 ][line width=1.5]    (49,78) .. controls (48.67,49.67) and (71.67,50.67) .. (71,78) ;
\draw [shift={(71,78)}, rotate = 91.4] [color={rgb, 255:red, 240; green, 0; blue, 0 }  ,draw opacity=1 ][fill={rgb, 255:red, 240; green, 0; blue, 0 }  ,fill opacity=1 ][line width=1.5]      (0, 0) circle [x radius= 2.61, y radius= 2.61]   ;
\draw [shift={(49,78)}, rotate = 269.33] [color={rgb, 255:red, 240; green, 0; blue, 0 }  ,draw opacity=1 ][fill={rgb, 255:red, 240; green, 0; blue, 0 }  ,fill opacity=1 ][line width=1.5]      (0, 0) circle [x radius= 2.61, y radius= 2.61]   ;
%Curve Lines [id:da6251511283368922] 
\draw [color={rgb, 255:red, 240; green, 0; blue, 0 }  ,draw opacity=1 ][line width=1.5]    (86,78) .. controls (85.67,112.67) and (115.67,113.67) .. (115,78) ;
\draw [shift={(115,78)}, rotate = 268.93] [color={rgb, 255:red, 240; green, 0; blue, 0 }  ,draw opacity=1 ][fill={rgb, 255:red, 240; green, 0; blue, 0 }  ,fill opacity=1 ][line width=1.5]      (0, 0) circle [x radius= 2.61, y radius= 2.61]   ;
\draw [shift={(86,78)}, rotate = 90.55] [color={rgb, 255:red, 240; green, 0; blue, 0 }  ,draw opacity=1 ][fill={rgb, 255:red, 240; green, 0; blue, 0 }  ,fill opacity=1 ][line width=1.5]      (0, 0) circle [x radius= 2.61, y radius= 2.61]   ;
%Curve Lines [id:da5526650719432145] 
\draw [color={rgb, 255:red, 240; green, 0; blue, 0 }  ,draw opacity=1 ][line width=1.5]    (205,78) .. controls (205.67,16.67) and (342.67,17.67) .. (342,78) ;
\draw [shift={(342,78)}, rotate = 90.63] [color={rgb, 255:red, 240; green, 0; blue, 0 }  ,draw opacity=1 ][fill={rgb, 255:red, 240; green, 0; blue, 0 }  ,fill opacity=1 ][line width=1.5]      (0, 0) circle [x radius= 2.61, y radius= 2.61]   ;
\draw [shift={(205,78)}, rotate = 270.62] [color={rgb, 255:red, 240; green, 0; blue, 0 }  ,draw opacity=1 ][fill={rgb, 255:red, 240; green, 0; blue, 0 }  ,fill opacity=1 ][line width=1.5]      (0, 0) circle [x radius= 2.61, y radius= 2.61]   ;
%Curve Lines [id:da22441876068002942] 
\draw [color={rgb, 255:red, 240; green, 0; blue, 0 }  ,draw opacity=1 ][line width=1.5]    (226,78) .. controls (226.67,144.67) and (372.67,140.67) .. (374,78) ;
\draw [shift={(374,78)}, rotate = 271.22] [color={rgb, 255:red, 240; green, 0; blue, 0 }  ,draw opacity=1 ][fill={rgb, 255:red, 240; green, 0; blue, 0 }  ,fill opacity=1 ][line width=1.5]      (0, 0) circle [x radius= 2.61, y radius= 2.61]   ;
\draw [shift={(226,78)}, rotate = 89.43] [color={rgb, 255:red, 240; green, 0; blue, 0 }  ,draw opacity=1 ][fill={rgb, 255:red, 240; green, 0; blue, 0 }  ,fill opacity=1 ][line width=1.5]      (0, 0) circle [x radius= 2.61, y radius= 2.61]   ;
%Curve Lines [id:da5064014904894845] 
\draw [color={rgb, 255:red, 0; green, 183; blue, 0 }  ,draw opacity=1 ][line width=1.5]    (154,78) .. controls (155.67,152.17) and (-18.67,161.83) .. (14,58) .. controls (46.67,-45.83) and (134.67,36.17) .. (134,78) ;
\draw [shift={(134,78)}, rotate = 77.46] [color={rgb, 255:red, 0; green, 183; blue, 0 }  ,draw opacity=1 ][fill={rgb, 255:red, 0; green, 183; blue, 0 }  ,fill opacity=1 ][line width=1.5]      (0, 0) circle [x radius= 2.61, y radius= 2.61]   ;
\draw [shift={(154,78)}, rotate = 114.73] [color={rgb, 255:red, 0; green, 183; blue, 0 }  ,draw opacity=1 ][fill={rgb, 255:red, 0; green, 183; blue, 0 }  ,fill opacity=1 ][line width=1.5]      (0, 0) circle [x radius= 2.61, y radius= 2.61]   ;
%Curve Lines [id:da0776827546460298] 
\draw [color={rgb, 255:red, 0; green, 183; blue, 0 }  ,draw opacity=1 ][line width=1.5]    (404,78) .. controls (407.67,130.17) and (465.33,97.83) .. (444,58) .. controls (422.67,18.17) and (180.67,-39.83) .. (185,78) ;
\draw [shift={(185,78)}, rotate = 118.48] [color={rgb, 255:red, 0; green, 183; blue, 0 }  ,draw opacity=1 ][fill={rgb, 255:red, 0; green, 183; blue, 0 }  ,fill opacity=1 ][line width=1.5]      (0, 0) circle [x radius= 2.61, y radius= 2.61]   ;
\draw [shift={(404,78)}, rotate = 58.26] [color={rgb, 255:red, 0; green, 183; blue, 0 }  ,draw opacity=1 ][fill={rgb, 255:red, 0; green, 183; blue, 0 }  ,fill opacity=1 ][line width=1.5]      (0, 0) circle [x radius= 2.61, y radius= 2.61]   ;
%Straight Lines [id:da5245998130334958] 
\draw [color={rgb, 255:red, 240; green, 0; blue, 0 }  ,draw opacity=1 ][line width=1.5]    (56,48) -- (61,66) ;
%Straight Lines [id:da19058935525569154] 
\draw [color={rgb, 255:red, 240; green, 0; blue, 0 }  ,draw opacity=1 ][line width=1.5]    (99,93) -- (104,111) ;
%Straight Lines [id:da03135218939245099] 
\draw [color={rgb, 255:red, 240; green, 0; blue, 0 }  ,draw opacity=1 ][line width=1.5]    (269,23) -- (274,41) ;
%Straight Lines [id:da4458115278003173] 
\draw [color={rgb, 255:red, 240; green, 0; blue, 0 }  ,draw opacity=1 ][line width=1.5]    (294,117) -- (299,135) ;
\end{tikzpicture}
\caption{This schematic indicates with a mark those index 1 bones which are assigned a shift in our $\zb/2$-graded diagram. The markings are determined by how the joints interact with a fixed choice of (ultimately auxiliary) orientations on the index 0 bones. See \cref{invariantdef} for more details.}
\label{markedintro}
\end{figure}
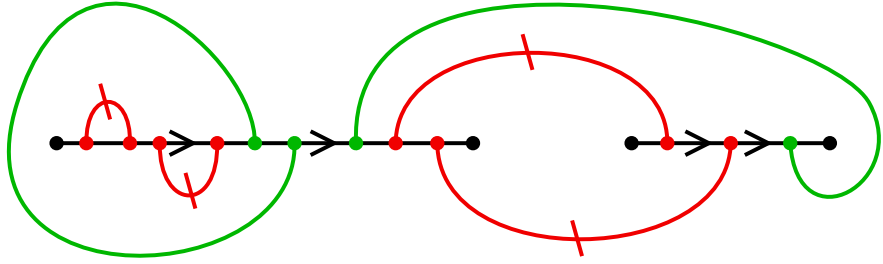

Invariance under each move on $G$ is proved using a particular algebraic model for the homotopy limit of dg-categories. In particular, we find that the algebraic effect of a handle slide can be combinatorially derived from that of a particular isotopy. For example, we find the following:\\

\begin{theorem} \label{exampletheorem}
Fix a strongly pre triangulated dg-category $\mc A$ with 2-periodic hom complexes. Let $W$ be the Weinstein manifold associated to the skeleton in \cref{diagram}. Then an object of $\mc L({W})$ can be presented as an $\mc A$-valued representation $A_{0}\overset{a_{0}}\to A_{1}\overset{a_{1}}\to A_{2}$ along with homotopy equivalences $A_{0}\simeq A_{2}$ and $C(a_{0})\simeq C(a_{1})[1]$.\end{theorem}

In future work, we plan to analyze the diagrams that arise in dimension 4. This involves representing the relevant categories and restriction functors arising from arboreal singularities of singular surfaces, proving invariance under a complete set of moves that relate $2$-dimensional arboreal skeleta, and demonstrating useful presentations of $\mc L ({W})$ for a wide class of Weinstein $4$-manifolds.

\subsection{Homotopical results}\label{hom}
\cref{maintheorem}, while satisfying, does not state an explicit way to compute our invariant $\mc L(W)$. In particular, to obtain the presentation given in \cref{exampletheorem}, a model for homotopy limits of dg-categories must be obtained. The usual philosophy of ``fattening up'' our diagram using path objects works, but proving so involves many details and generalizes in an elegant way to diagrams of more general quivers over more general model categories.

Given a small category $J$, we are working in the diagram category $\dgcat^{J}$, on which we use a Reedy model structure suitable to $J$ induced by the Tabuada model structure on $\dgcat$ defined in \cite{T}. More generally, we can replace $\dgcat$ with any model category $\mc M$.  If $J$ is equipped with a Reedy structure $(J^{+},J^{-})$ compatible with the (ordinary) limit functor, then the homotopy limit functor $\holim:\mc M^{J}\lr \mc M$ can be calculated by applying the limit functor to a fibrant replacement of diagrams. Compatibility is ensured if $(J^{+},J^{-})$ has cofibrant constants, which is a simple combinatorial condition described in \cite{H}.

In the case of Weinstein surfaces, each diagram arises from a graph $G=(V,E)$. In particular, we always have $\mc A^{\to}$ associated to each trivalent vertex and $\mc A$ associated to each edge, with three arrows out of each $\mc A^{\to}$ and two arrows into each $\mc A$. More generally, we allow $G$ to be any finite graph and consider any diagram over the quiver $Q(G)$ whose vertex set is $V\sqcup E$ and whose arrows arise from $G$ by replacing each edge $e=\{v,w\}$ with $v_{1}\to e\leftarrow w$.\\

\begin{theorem}[technical statement in \cref{Reedy}]\label{reed}
Fix a model category $\mc M$ in which all final maps and projections from Cartesian products are fibrations. Let $G$ be a finite simple graph and $D$ an arbitrary diagram in $\mc M^{Q(G)}$. Then $\holim D \simeq \lim D$ if there exists an acyclic orientation $\mathfrak o$ on $G$ such that, for each vertex $v$, a certain map $\vp_{\mathfrak o}^{D}(v)$ is a fibration in $\mc M$.
\end{theorem}

The map $\vp_{\mathfrak o}^{D}(v)$ is defined in the discussion preceding \cref{Reedy}, whose natural generalization, \cref{recognize}, enlarges our class of diagrams to those with shape any quiver of diameter 1. The hypotheses of \cref{reed} are satisfied by the following procedure. By definition, $D$ contains one cospan $\mc D \to \mc C \leftarrow \mc B$ for each edge of $G$. We construct the \textit{path extension} $P(D)$ by enlarging each cospan in the following way:
% https://q.uiver.app/#q=WzAsOSxbMCwwLCJcXG1hdGhjYWwgQSJdLFsxLDAsIlxcbWF0aGNhbCBDIl0sWzIsMCwiXFxtYXRoY2FsIEQiXSxbMywwLCJcXHJpZ2h0c3F1aWdhcnJvdyJdLFs0LDAsIlxcbWF0aGNhbCBBIl0sWzUsMCwiXFxtYXRoY2FsIEMiXSxbNiwwLCJQKFxcbWF0aGNhbCBDKSJdLFs3LDAsIlxcbWF0aGNhbCBDIl0sWzgsMCwiXFxtYXRoY2FsIEIiXSxbMCwxXSxbMiwxXSxbNCw1XSxbNiw1LCJcXHBpXzEiLDJdLFs2LDcsIlxccGlfMiJdLFs4LDddXQ==
\[\begin{tikzcd}
	{\mathcal D} & {\mathcal C} & {\mathcal B} & \rightsquigarrow & {\mathcal D} & {\mathcal C} & {P(\mathcal C)} & {\mathcal C} & {\mathcal B}
	\arrow[from=1-1, to=1-2]
	\arrow[from=1-3, to=1-2]
	\arrow[from=1-5, to=1-6]
	\arrow["{\pi_1}"', from=1-7, to=1-6]
	\arrow["{\pi_2}", from=1-7, to=1-8]
	\arrow[from=1-9, to=1-8]
\end{tikzcd}\]
Here $P(\mc C)$ is the path object of $\mc C$ with projections $\pi_{i}$. In the case $\mc M=\dgcat$, this is the category of homotopy equivalences in $\mc C$. In \cite{K}, Karaba\c{s} proved that $\holim D \simeq \lim P(D)$ for the case where $G$ is simple with one edge, giving an explicit formula for homotopy pullbacks in $\dgcat$. One direction of this research is in generalizing his approach: \\

\begin{proposition} [\cref{pathresult}]
Fix a model category $\mc M$ in which all final maps and projections from Cartesian products are fibrations. Given a finite graph $G$, we have $\holim D \simeq \lim P(D)$ for any diagram $D$ in $\mc M^{Q(G)}$.
\end{proposition}

We aim to further this direction of research in parallel to the investigation of skeleta of Weinstein $4$-manifolds. First, this involves generalizing \cref{Reedy} quivers $Q$ of diameter 2. We intend to use Reedy-categorical methods to find a wide range of conditions on a diagram $D$ in $\mc M^{Q}$ which ensure that $\holim D \simeq \lim D$, using this to find concrete presentations of a wider class of homotopy limits of dg-categories.

The rest of this paper contains the calculations which prove the results stated above. In Section 1 we develop the Reedy categorical information needed to characterize fibrant diagrams and calculate homotopy limits. In Section 2 we calculate explicit homotopy limits of dg-categories and establish small algebraic moves. In Section 3 we combine these moves to describe the algebraic effect of Weinstein homotopies and prove invariance of $\mc L(W)$, followed by explicit calculations covering all topological surfaces.

The reader interested in the dg-categorical algebra that builds our geometric results can safely skip most of Section 2. The main use of Section 2 in the later sections is \cref{pathresult}. Independently of our symplectic geometric goals, the main model categorical result of Section 2 is \cref{recognize}. 

\textbf{Acknowledgements:} The author thanks Laura Starkston for the inspiration and direction of this project, along with Honghao Gao and Roger Casals for their additional mentorship. The author is also grateful for helpful conversations with David Nadler, Hiro Tanaka, Do\u{g}ancan Karaba\c{s}, Thomas Blom, and Maxime Ramzi. Thank you to Merlin Christ for his pointing out the relevance of \cite{DK}. This project was supported by Laura Starkston's NSF CAREER Grant (204234): Symplectic 4-Manifolds and Singular Symplectic Surfaces.

%%%%%%%%%%%%%%%%%%% SECTION 1%%%%%%%%%%%%%
\section{Fibrancy recognition}
%%%%%%%%%%%%%%%%%%%%%%%%%%%%%%%%%%%%%%
Given a Weinstein surface $W$, we will define $\mc L(W)$ to be the homotopy limit of a diagram $D$ of dg-categories associated to an arboreal representative of the Weinstein homotopy class of $W$. It follows that we need a way to concretely compute homotopy limits of the diagrams arising from arboreal skeleta. 

As a derived functor, the operation of taking a homotopy limit in general involves some replacement of a diagram by a more well-behaved diagram. In theory, this can be done without altering the shape of the original diagram, but constructing this replacement directly from definitions involves complicated objects that arise through Quillen's Small Object Argument. Instead, a useful approach is to enlarge diagrams, in the vein of homotopy (co)limits in algebraic topology.

Thus we want to first replace $D$ with a suitably equivalent \textit{larger} diagram which is suitably \textit{fibrant}, meaning its homotopy limit is equivalent to its (ordinary) limit. This will be achieved by putting on $\dgcat^{J}$ a Reedy model structure suitable for the shape $J$ of our diagram.

In this section we characterize such Reedy structures on a large class of model categories $\mc M$, producing a lemma that allows us to recognize when a diagram over $\mc M$ is fibrant in this sense. Then we use our results to represent the homotopy limits we care about as explicit classical limits.

\subsection{Categorical preliminaries} 
A \textit{homotopical category} is a category $\mc M$ along with a specified class $W\sube \mathrm{Mor}(\mc M)$ of \textit{weak equivalences}; see \cite{R} for the general theory. With respect to these weak equivalences, one wants to compute derived functors, in particular the derived functor of the limit functor, if $\mc M$ is complete, which is then called the \textit{homotopy limit} functor. If $(\mc M, W)$ supports a \textit{model structure} in which \textit{(co)fibrations} have been specified and satisfy some coherence and lifting properties, then a derived functor can be computed by first applying a \textit{fibrant replacement} to objects and then applying the original functor, as in the classical theory in homological algebra. See \cite{DHKS, Hir, H} for the detailed theory of model categories. 

We now fix an arbitrary small category $J$ and a model category $\mc M$. In the case of homotopy limits, one is working with the homotopical category $\mc M^{J}$. Model categories carry a natural notion of homotopy between morphisms, so it will be convenient for us to create weak equivalences in $\mc M^{J}$ by using morphisms which only commute up to homotopy. Thus we will often implicitly use the following technical result.\\

\begin{lemma}\label{commute2} Suppose $J$ is the free category on a quiver. Fix diagrams $X,Y\in \ob(\mc M^{J})$. If there exist objectwise weak equivalences from $X$ to $Y$ such that the resulting squares commute up to homotopy in $\mc M$, then $X$ and $Y$ are weakly equivalent in $\mc M^{J}$. That is, the squares can be made to commute on the nose.
\end{lemma}
\begin{proof}
By the hypothesis, $X$ and $Y$ are isomorphic in $\Ho(\mc M)^{J}$. Since the natural functor $\Ho(\mc M^{J})\lr \Ho(\mc M)^{J}$ is the identity on objects and reflects isomorphisms by \cite[3.1.2]{RV}\footnote{We are implicitly using the fact that every model category presents an $\infty$-category.}, $X$ and $Y$ are also isomorphic in $\Ho(\mc M^{J})$.
\end{proof}

Depending on $J$, there are multiple different model structures that may be put on $\mc M^{J}$. In particular, if $J$ supports a \textit{Reedy structure}, then $\mc M^{J}$ has an induced \textit{Reedy model structure}, which is often simple to work with. Since the shapes $J$ we care about will be very simple quivers, we will employ Reedy categorical methods, determining which diagrams in $\mc M^{J}$ are fibrant with respect to a Reedy model structure suitable for taking homotopy limits. See \cite{R, Hir} for the theory of Reedy categories.

We record a useful way to recognize homotopy limits in a Reedy model category. We will say that an element $(J^{+},J^{-})$ of the set $\mathrm{R}(J)$ of Reedy structures on $J$ has \textbfit{cofibrant constants} if the latching category $\p (J^{+}/\a)$ is connected for each $\a\in \ob(J)$.\\

\begin{lemma}\label{p1} Given $X\in \ob(\mc M^{J})$, if there exists an element of $\mathrm{R}(J)$ with cofibrant constants such that all matching maps of $X$ are fibrations, then $\holim X\simeq \lim X$. If all objects in $\mc M$ are fibrant, then we only need to check the matching maps for those $\a\in \ob(J)$ for which there is some arrow $\a\to \b$ in $J^{-}$ with $\b\neq \a$.
\end{lemma}

\begin{proof}
By the definition of fibrations in a Reedy model structure, the condition on the matching maps ensures that $X$ is fibrant, so \cite[15.10.2, 15.10.8]{Hir} combined with Ken Brown's Lemma\footnote{See \cite{B,R}.} yields the result.
\end{proof}

%%%%%%%%%%%%Graph definitions%%%%%%%%%
\subsection{Reedy structures on graphic quivers}

Here we make restrictions on our shape category $J$ to suit our purposes. Given vertices $s$ and $t$ in a quiver $Q$, we write $Q(s,t)$ for the set of arrows $s\to t$, and $Q(t)$ for the set of all arrows incident to $t$. By convention we consider isolated vertices to be \textit{sources} and not sinks. We will occasionally write $d, c: Q_{1}\lr Q_{0}$ for the source and target functions, respectively.\\

\begin{definition}\label{quiverdefinitions}
Fix a quiver $Q$. An arrow $s\overset{x}\to t$ is called \textbfit{isolated} if $x$ is the only arrow $s\to t$, and \textbfit{parallel} otherwise. If $x$ is isolated, we say that \textbfit{$s$ is isolated by $t$}, that \textbfit{$t$ isolates $s$}, and that \textbfit{$t$ is isolating}. The \textbfit{simplification} of $Q$ is the quiver $Q^{\mathrm{s}}$ obtained from $Q$ by removing all parallel arrows and all non-isolating sinks. Similarly, the \textbfit{simplification} of a graph $G$ is the graph $G^{\mathrm{s}}$ obtained by removing all loops in $G$.\footnote{Note that $Q$ is \textbfit{simple}, that is $Q=Q^{\mathrm{s}}$, if and only if every arrow in $Q$ is isolated.}
\end{definition}

Given a graph $G=(V,E)$, the \textbfit{quiver associated to $G$} is the quiver $Q(G)$ with vertex set the ordered partition $V\sqcup E$, and one arrow $v\to e$ for each instance\footnote{If $G$ has a loop $e$ at $v$, then $(v,e)$ is counted twice in $V\times E$.} of $(v,e)\in V\times E$ with $v\in e$.\footnote{This justifies our asymmetric definition of sources and sinks: $S(Q(G))$ (resp. $T(Q(G))$) is in bijection with $V$ (resp. $E$). We also have $Q(G^{\mathrm{s}})=Q(G)^{\mathrm{s}}$.} A quiver isomorphic to some $Q(G)$ will be called \textbfit{graphic}.

A \textbfit{labeling} on a quiver $Q$ is a map $\s:Q_{1}\lr \{+,-\}$. A \textbfit{simple labeling} on $Q$ is a labeling on $Q^{\mathrm{s}}$. The set of all labelings on $Q$ will be denoted by $\s(Q)$, and the set of all orientations on a graph $G$ will be denoted $\mathfrak o(G)$. By reversing all negative arrows in $Q$, we obtain a bijection $O:\s(Q)\lr \mathfrak o(U(Q)\footnote{The underlying graph of $Q$.})$, and we call $\s\in \s(Q)$ \textbfit{acyclic} if $O(\s)$ is acyclic. We write $\s(Q)^{\mathrm{a}}$ for the set of acyclic labelings on $Q$. 

Graphic quivers in particular have diameter 1. This weaker condition is the natural general setting for our results, so we now fix a finite quiver $Q$ with diameter 1.\\

\begin{definition}
 Fix a labeling $\s\in \s(Q)$. Given a sink $t$, if there is exactly one positive arrow incident to $t$, then we say that $t$ is \textbfit{$\s$-strict} and that $\s$ is \textbfit{strict at $t$}. We say that $\s$ is \textbfit{strict} if it is strict at every sink.
\end{definition}

The next result shows that it's easy to find all strict acyclic labelings on a simple quiver. A \textbfit{graphic ordering} of $Q$ is a linear ordering on the set of sources of $Q$. Any graphic ordering $\leq$ of $Q$ specifies a simple labeling $\s_{\leq}\in \s(Q^{\mathrm{s}})$: given an arrow $x$ in $Q^{\mathrm{s}}$, $\s_{\leq }(x)=+$ if and only if $d(x)=\min_{\leq}d(Q(t))$.\\

\begin{lemma}\label{topologicalordering}
Fix a finite quiver $Q$ with diameter $1$. The mapping $\leq\, \mapsto \s_{\leq}$ on graphic orderings of $Q$ has image exactly the strict elements of the set $\s(Q^{\mathrm{s}})^{\mathrm{a}}$ of acyclic simple labelings.
\end{lemma}
\begin{proof}
It is straightforward that each $\s_{\leq}$ is strict and acyclic. Since $Q$ and its simplification have the same sources, we can assume that $Q$ is simple. Fix a strict element $\s\in \s(Q^{\mathrm{s}})^{\mathrm{a}}$. There is a linear ordering $\leq_{U(Q)}$ on $Q_{0}$ such that $x\leq_{U(Q)} y$ if and only if there is an $O(\s)$-path from $x$ to $y$ in $U(Q)$. 

Let $\leq$ be the restriction of $\leq_{U(Q)}$ to the sources of $Q$. We show $\s_{\leq}=\s$. Since $\s_{\leq}$ and $\s$ are both strict simple labelings on $Q$, it suffices to show, for each sink $t$, that $\s(x)=+$, where $x\in Q(t)$ is the (necessarily unique) arrow with $d(x)=\min_{\leq }d(Q(t))$. Indeed, let $z\in U(Q)_{1}$ be the unique element of $Q(t)\cap \s^{-1}(+)$. For any $y\in Q(t)$ distinct from $z$ we have $\s(y)=-$, so $d(a)\overset{+}\to t\overset{-}\to d(y)$ is the $O(\s)$-path $zy$ in $U(Q)$ from $d(z)$ to $d(y)$. Note that $d(z)\neq d(y)$ because $Q$ is simple,  so $d(z)<d(y)$. Therefore, $d(z)=\min_{\leq} d(Q(t))=d(x)$, so $z=x$ because $Q$ is simple.
\end{proof}

In the following, given a quiver $Q$, we write $F(Q)$ for the \textit{free category} on $Q$. We characterize the Reedy structures on $F(Q)$. From a labeling $\s\in \s(Q)$, we obtain two wide subcategories $Q^{+}_{\s}$ and $Q^{-}_{\s}$ of $F(Q)$, those generated by $\s^{-1}(+)$ and $\s^{-1}(-)$ respectively.\\

\begin{lemma}\label{bijection}
Fix a finite quiver $Q$ with diameter $1$. The mapping $\s\mapsto (Q^{+}_{\s},Q^{-}_{\s})$ restricts to a bijection $P:\s(Q)^{\mathrm{a}}\lr \mathrm{R}(Q)$\footnote{As we occasionally  abuse notation by conflating $Q$ with $F(Q)$, we will write $\mathrm{R}(Q)$ instead of $\mathrm{R}(F(Q))$.} between acyclic labelings and Reedy structures\footnote{Recall that a Reedy structure does not include the data of a specific degree function.}.
\end{lemma}
\begin{proof}

Injectivity is clear, so it remains to show that $P(\s)$ is a Reedy structure when $\s\in \s(Q)^{\mathrm{a}}$, and that this mapping is surjective. The pair $P(\s)$ has unique factorization because there are no nontrivial factorizations in $F(Q)$. We can construct a degree function as follows. Let $Q'$ be the quiver $(U(Q), O(\s))$. Since $\s$ is acyclic, $Q'$ is acyclic, so $Q'$ has a topological ordering $(v_{1},\dots,v_{n})$ of $(Q')_{0}=Q_{0}$. The association $v_{i}\mapsto i$ is a degree function for $F(Q)$, so $P(\s)$ is a Reedy structure.

To show surjectivity, let $(F(Q)^{+},F(Q)^{-})$ be a Reedy structure on $F(Q)$. Since $F(Q)$ has no nontrivial factorizations, every morphism in $F(Q)$ must be in $F(Q)^{+}$ or $F(Q)^{-}$. The non-identity morphisms of $F(Q)$ cannot be in both subcategories, otherwise factorization would be non-unique. So every non-identity morphism of $F(Q)$ (and therefore every arrow in $Q$) has a well-defined sign, meaning $(F(Q)^{+},F(Q)^{-})=(Q^{+}_{\s},Q^{-}_{\s})$ for a unique labeling $\s\in \s(Q)$. 

If $\s$ is not acyclic, then there is a cycle in the quiver $(U(Q), O(\s))$ traversing vertices $(v_{1},\dots,v_{n}=v_{1})$. Fix a degree function for the Reedy structure $(Q^{+}_{\s},Q^{-}_{\s})$. Then $\deg(v_{i})<\deg(v_{i+1})$ for all $i\in\{1,\dots,n-1\}$, because each edge in the cycle has sign $\s(v_{i}\to v_{i+1})=+$ or $\s(v_{i}\leftarrow v_{i+1})=-$. This results in the contradiction $\deg(v_{1})<\deg(v_{1})$.
\end{proof}

As a result, we can specify a Reedy structure on a finite quiver with diameter $1$ by assigning signs to every arrow, as long as the result is acyclic.

\subsection{Checking \cref{p1} for homotopy limits}
Having characterized the Reedy structures on a finite Quiver with diameter 1 by acyclic labelings of the edge set, we look at which of these make the limit functor amenable to fibrant replacement\footnote{In technical terms, right Quillen.} by checking connectedness of the latching categories. In these situations, we write down the matching maps that tell us which diagrams are fibrant. We start with the first condition of \cref{p1}.\\

\begin{lemma}\label{latching}
Fix $\s\in\s(Q)^{\mathrm{a}}$ for a finite quiver $Q$ with diameter 1. The Reedy structure $(Q^{+}_{\s},Q^{-}_{\s})$ has cofibrant constants if and only if for each vertex $v\in Q_{0}$, there is at most one positive arrow with target $v$.
\end{lemma}

\begin{proof}
The vertex set $Q_{0}=\ob(F(Q))$ partitions into the non-isolated sources, the (necessarily non-isolated) sinks, and the isolated vertices. The latching categories of $F(Q)$ are non-empty only at the sinks. For a sink $t$, the object set of the latching category $\p(Q^{+}_{\s}/t)$ is $Q(t)\cap \s^{-1}(+).$ Given $x,y\in \ob(\p(Q^{+}_{\s}/t))$, since $d(x)$ is a source there is no arrow $d(x)\to d(y)$ unless $d(x)=d(y)$. If $d(x)=d(y)$, then the only arrow $d(x)\to d(y)$ in $F(Q)$ is the identity. But if $x\neq y$, then $\id_{d(x})$ is not a morphism $x\to y$, so the latching category $\p(Q^{+}_{\s}/t)$ is discrete, and the result follows.
\end{proof}

Since a quiver with diameter 1 is acyclic, the two constant labelings are both acyclic. Therefore, in our setting there always exists an acyclic labeling satisfying the conditions in \cref{latching}: the constant negative labeling. But as we will see later, this labeling produces the smallest class of fibrant diagrams among all useful Reedy model structures. Given a small category $J$, a Reedy structure $R\in \mathrm{R}(J)$, and a model category $\mc M$, we will write $\mc M^{R}$ for the Reedy model structure on $\mc M^{J}$ induced by $R$.\\

\begin{lemma}\label{fibrantcheck}
Fix $\s\in\s(Q)^{\mathrm{a}}$ for a finite quiver $Q$ with diameter 1, and let $\mc M$ be a model category. Let $X\in \ob(\mc M^{F(Q)})$ be a diagram. Then $X$ is fibrant in $\mc M^{(Q_{\s}^{+},Q_{\s}^{-})}$ if and only if the value of $X$ at every sink is fibrant and for each source $s$ the induced map $X^{{s}}\lr \prod_{x\in Q(s)\cap \s^{-1}(-)}X^{c(x)}$ is a fibration.
\end{lemma}

\begin{proof}
Given a source $s$, an argument dual to the proof of \cref{latching} shows that the matching map of $X$ at $s$ is the one given in the statement.
\end{proof}
 
\begin{definition} Given a quiver $Q$ with diameter 1, the poset $(\s(Q),\leq)$ is defined by $\s_{1}\leq \s_{2}$ if and only if $\s_{1}^{-1}(+)\subseteq \s_{2}^{-1}(+)$. Given a model category $\mc M$, the poset $(\mathrm{R}(Q),\leq_{\mc M})$ is defined by $R_{1}\leq_{\mc M} R_{2}$ if and only if the class of fibrant objects of ${\mc M}^{R_{1}}$ is contained in that of ${\mc M}^{R_{2}}$.
\end{definition}

In the following, given a complete category $\mc M$, we will use the term \textbfit{projection} to refer to any of the canonical maps  $\prod_{i\in I} x_{i}\lr \prod_{i\in J} x_{i},$ where $I$ is an index set, $J\sube I$, and $x_{i}\in \ob(\mc M)$.\footnote{If $J$ is empty then this is the final map $\prod_{i\in I}x_{i}\lr 1$. In particular, if all projections in $\mc M$ are fibrations, then every object is fibrant.} In the following, $\mathrm{R}(Q)^{\mathrm{c}}\sube \mathrm{R}(Q)$ is the set of Reedy structures on $Q$ with cofibrant constants.\\

\begin{proposition}\label{monotone}
Let $Q$ be a finite quiver with diameter $1$, and fix a model category $\mc M$ in which every projection is a fibration. The bijection in \cref{bijection} is a monotone map $P:(\s(Q)^{\mathrm{a}},\leq)\lr (\mathrm{R}(Q),\leq_{\mc M})$. In particular, if $X\in \ob(\mc M^{F(Q)})$ is fibrant with respect to the Reedy model structure corresponding to the constant negative labeling, then $X$ is fibrant with respect to any Reedy model structure. Furthermore, every maximal element of $(\mathrm{R}(Q),\leq_{\mc M})$ (resp. $(\mathrm{R}(Q)^{\mathrm{c}},\leq_{\mc M})$) corresponds to a maximal element of $(\s(Q)^{\mathrm{a}},\leq)$ (resp. $(P^{-1}(\mathrm{R}(Q)^{\mathrm{c}}),\leq)$.
\end{proposition}

\begin{proof}
For $i=1,2$ fix $\s_{1},\s_{2}\in \s(Q)^{\mathrm{a}}$ and define $R_{i}\coloneqq P(\s_{i})$. Suppose $\s_{1}\leq \s_{2}$, and fix a diagram $X\in \ob(\mc M^{F(Q)})$ which is fibrant with respect to ${\mc M}^{R_{1}}$. We verify the hypotheses of \cref{fibrantcheck} for $\s_{2}$. Since $\s_{1}^{-1}(-)\supseteq \s_{2}^{-1}(-),$ the map $X^{{s}}\lr \prod_{x\in Q(s)\cap \s_{2}^{-1}(-)}X^{c(x)}$ factors through the projection \[\prod_{x\in Q(s)\cap \s_{1}^{-1}(-)}X^{c(x)}\lr \prod_{x\in Q(s)\cap \s_{2}^{-1}(-)}X^{c(x)}.\] Since both maps in the factorization are fibrations, $X$ is fibrant in $\mc M^{R_{2}}$. The other statements follow from the fact that, given a monotone map $f$ of posets, $f$ preserves least and greatest elements if $f$ is surjective, and $f^{-1}$ preserves minimal and maximal elements if $f$ is injective.
\end{proof}

A maximal element of $(\mathrm{R}(Q)^{\mathrm{c}},\leq_{\mc M})$ induces a Reedy model structure on $\mc M^{F(Q)}$ whose fibrancy conditions can't be weakened by moving to another element of $\mathrm{R}(Q)^{\mathrm{c}}$. Therefore, we will aim to characterize the corresponding labelings in $P^{-1}(\mathrm{R}(Q)^{\mathrm{c}})$. \\

\begin{definition}
Fixing $Q$ and $\mc M$ as above, a maximal element of $(\mathrm{R}(Q)^{\mathrm{c}},\leq_{\mc M})$ will be called \textbfit{ideal}, as will its corresponding labelings. A maximal element of $(P^{-1}(\mathrm{R}(Q)^{\mathrm{c}}),\leq)$ will be called \textbfit{adequate}, as will its corresponding Reedy structure.\footnote{The names reflect the fact that every ideal labeling/Reedy structure is adequate, but not vice versa, unless $P^{-1}$ is monotone.}
\end{definition}

Ideal Reedy structures are defined using $\mathrm{R}(Q)^{\mathrm{c}}$, while adequate Reedy structures are defined using $P^{-1}(\mathrm{R}(Q)^{\mathrm{c}})\sube \s(Q)^{\mathrm{a}}$. Therefore, since the latter poset relations are easier to check than the former, to stay within the combinatorics of labelings we seek to classify all adequate labelings. This will capture all ideal Reedy structures while leaving out many (but perhaps not all) weaker Reedy structures, meaning those in $\mathrm{R}(Q)^{\mathrm{c}}$ whose fibrancy conditions can be weakened within $\mathrm{R}(Q)^{\mathrm{c}}$.

%%%%%%%%%%%%%%%%%%%%%%%%%%%%%%%%%%%%%%%
%%%%%%%%%%% FINAL CHARACTERIZATION %%%%%%%%%%%%%
%%%%%%%%%%%%%%%%%%%%%%%%%%%%%%%%%%%%%%%
\subsection{Finding adequate Reedy structures}

As before, fix a finite quiver $Q$ with diameter $1$. Given a labeling $\s\in \s(Q)$, a necessary condition for $\s$ to be acyclic is that for any sink $t$, $\s$ is constant on any two parallel arrows $x,y\in Q(t)$ with common source. Therefore, to ensure that $P(\s)$ has cofibrant constants, \cref{latching} amounts to saying that all parallel arrows must be negative under $\s$.\\

\begin{definition} \label{lset} We set $L=L(Q)$ to be the subposet of $(\s(Q),\leq)$ consisting of those labelings which are negative at every parallel arrow and for which each sink has at most one positive arrow incident to it.
\end{definition}

Among the labelings in $L$, we'd like to pick those which have many positive arrows, because more positive arrows means more fibrant diagrams. In particular, the set $\max(L)$ of maximal elements of $L$ consists of those labelings in $L$ for which each sink has exactly one positive arrow incident to it.\footnote{Said another way, $\max(L)$ is the set of labelings on $Q$ which extend some strict simple labeling on $Q$ by all negative signs, and $L$ is the downward closure of $\max(L)$ in $(\s(Q),\leq)$.}

In general, not every labeling in $\max(L)$ is acyclic, so we set $L^{\mathrm{c}}=L^{\mathrm{c}}(Q)\coloneqq L(Q)\cap \s(Q)^{\mathrm{a}}$. Then $P(L^{\mathrm{c}}(Q))=\mathrm{R}(Q)^{\mathrm{c}}$, and in particular $\max(L^{\mathrm{c}})$ is the set of adequate labelings. Our goal can now be stated concisely: characterize the set $\max(L^{\mathrm{c}})$. Given $\s\in \s(Q^{\mathrm{s}})$, the \textbfit{negative filling} of $\s$ is the labeling $\widehat\s\in \s(Q)$ which extends $\s$ minimally---that is, by negative signs.\\

\begin{lemma}\label{extend}
Fix a quiver $Q$ with diameter 1. The set $\max(L(Q))\cap L^{\mathrm{c}}(Q)$ is equal to the set of negative fillings of the strict elements of $\s(Q^{\mathrm{s}})^{\mathrm{a}}$.\footnote{We saw in \cref{topologicalordering} that every strict element of $\s(Q^{\mathrm{s}})^{\mathrm{a}}$ arises from a linear ordering of the sources of $Q$, so these are easily enumerated.}
\end{lemma}
\begin{proof}
Fix $\s\in \s(Q^{\mathrm{s}})$. Note that $\s$ is acyclic if and only if $\widehat\s$ is acyclic, and $\s$ is strict if and only if $\widehat\s$ is strict. Then $\widehat \s \in \max(L)$ by strictness, and $\s\in L^{\mathrm{c}}$ by acyclicity.
Conversely, given $\s\in \max(L)\cap L^{\mathrm{c}}$, upon ignoring parallel arrows, $\s$ arises as the negative filling of a unique simple labeling. Again acyclicity and strictness follow.
\end{proof}

\begin{lemma}\label{astrict}
Let $Q$ be a quiver with diameter $1$, let $t$ be an isolating sink, and let $Q'\coloneqq Q\sm t$ be the quiver obtained from $Q$ by deleting $t$. Every labeling in $L^{\mathrm{c}}(Q')$ extends to a labeling in $L^{\mathrm{c}}(Q)$ which is strict at $t$.
\end{lemma}
\begin{proof}
We prove the contrapositive: given $\s\in L(Q')$, if no extension of $\s$ which is strict at $t$ is in $L^{\mathrm{c}}(Q)$, then $\s$ is not acyclic. Since we are only looking at labelings on the simplification of $Q$, we may assume that $Q$ is simple. Fix $\s\in L(Q')$, and assume $\s$ does not extend to an element of $L^{\mathrm{c}}(Q)$ strict at $t$. For each $x\in Q(t)$ let $\s_{x}$ be the unique element of $L(Q)$ extending $\s$ and satisfying $\s_{x}(x)=+$. Thus $\{\s_{x}\mid x\in Q(t)\}$ is the set of extensions of $\s$ to labelings in $L(Q)$ which are strict at $t$.

By our assumption, for each $x\in Q(t)$, there is a $\s_{x}$-directed cycle $C_{x}$ in $U(Q)$. If for some $x\in Q(t)$, $C_{x}$ does not contain $t$, then $C_{x}$ is a $\s$-directed cycle in $U(Q')$, so $\s$ is not acyclic. So assume each $C_{x}$ contains $t$ for each $x\in Q(t)$. Since $Q(t)$ is nonempty by our definition of sink, there exists a cycle through $t$, so $Q(t)$ contains at least two elements $x_{1}\neq x_{2}$ with sources $s_{1}$ and $s_{2}$, respectively. Under $\s_{1}\coloneqq \s_{x_{1}}$ and $\s_{2}\coloneqq\s_{x_{2}}$, these look like $s_{1}\overset{x_{1}(\s_{1})}\lr t\overset{x_{2}(\s_{1})}\lr s_{2}$ and $s_{1}\overset{x_{1}(\s_{2})}\longleftarrow t\overset{x_{2}(\s_{2})}\longleftarrow s_{2}$, where $x_{i}(\s_{j})$ means the edge $x_{i}\in U(Q)_{1}$ oriented by $\s_{j}$. 

We choose each $C_{x}$ to start at $x$, so that $C_{x_{1}}=x_{1}x_{2}f_{1}\cdots f_{n}$ and $C_{x_{2}}=x_{2}x_{1}g_{1}\cdots g_{m}$ for composable arrows $f_{\a}$ and $g_{\b}$ composed left to right. Note that none of the $f_{\a}$ or $g_{\b}$ are in $Q(t)$, so $\s_{1}$ and $\s_{2}$ agree on these edges. Therefore, $C_{x_1}$ and $C_{x_2}$ locally look like the following:
% https://q.uiver.app/#q=WzAsNyxbMSwxLCJ2XzEiXSxbMCwwLCJcXGJ1bGxldCJdLFswLDIsIlxcYnVsbGV0Il0sWzIsMSwiZSJdLFszLDEsInZfMiJdLFs0LDAsIlxcYnVsbGV0Il0sWzQsMiwiXFxidWxsZXQiXSxbMSwwLCJmX24iXSxbMCwyLCJnX20iXSxbMCwzLCIodl8xLGUpX3tcXHNpZ21hX3t2XzF9fSIsMCx7ImN1cnZlIjotMX1dLFszLDQsIih2XzIsZSlfe1xcc2lnbWFfe3ZfMX19IiwwLHsiY3VydmUiOi0xfV0sWzMsMCwiKHZfMSxlKV97XFxzaWdtYV97dl8yfX0iLDAseyJjdXJ2ZSI6LTF9XSxbNCwzLCIodl8yLGUpX3tcXHNpZ21hX3t2XzJ9fSIsMCx7ImN1cnZlIjotMX1dLFs0LDUsImZfMSJdLFs0LDYsImdfMSIsMl1d
\[\begin{tikzcd}
	\bullet &&&& \bullet \\
	& {s_1} & t & {s_2} \\
	\bullet &&&& \bullet
	\arrow["{f_n}", from=1-1, to=2-2]
	\arrow["{g_1}", from=2-2, to=3-1]
	\arrow["{\!\!\!\!\!\!\!\!\!\!\!x_{1}(\s_{1})\quad \,\,\,x_{2}(\s_{1})}",bend right = -30, from=2-2, to=2-3]
	\arrow["", bend right = -30, from=2-3, to=2-4]
	\arrow["{\!\!\!\!\!\!\!\!\!\!\!x_{1}(\s_{2})\quad \,\,\,x_{2}(\s_{2})}", bend right = -30, from=2-3, to=2-2]
	\arrow["", bend right = -30, from=2-4, to=2-3]
	\arrow["{f_1}", from=2-4, to=1-5]
	\arrow["{g_m}", to=2-4, from=3-5]
\end{tikzcd}\]
Then $f_{1}\cdots f_{n}g_{1}\cdots g_{m}$ is a well-defined $\s$-directed cycle in $U(Q')$ based at $s_{2}$, so $\s$ is not acyclic.
\end{proof}

\begin{lemma}\label{equal}
Let $Q$ be a finite quiver with diameter 1. Then $\max(L^{\mathrm{c}}(Q))=\max(L(Q))\cap L^{\mathrm{c}}(Q)$.
\end{lemma}
\begin{proof}
It suffices to prove $\max (L^{\mathrm{c}})\sube \max(L)$, as the other direction is a general fact about subsets of posets. Fix $\s\in L^{\mathrm{c}}\sm \max(L)$, so there is an isolating sink $t$ such that $\s$ is negative on all arrows incident to $t$. Now $\s|_{{Q\sm t}}\in L^{\mathrm{c}}(Q\sm t)$ because $\s\in L^{\mathrm{c}}(Q)$, so \cref{extend} produces a labeling $\s'\in L^{\mathrm{c}}(Q)$ which extends $\s|_{{Q\sm t}}$ and is strict at $t$. In particular, there is an element $x\in Q(t)$ such that $\s'(x)=+$. Consequently, $\s< \s'$, so $\s$ is not maximal in $L^{\mathrm{c}}$.
\end{proof}

We now come to our main result.\\

\begin{proposition}[Fibrancy recognition]\label{recognize}Fix a finite quiver $Q$ with diameter $1$ and a model category $\mc M$ in which all projections are fibrations. A diagram $X\in \ob(\mc{M}^{F(Q)})$ is fibrant with respect to a Reedy structure on $F(Q)$ with cofibrant constants if and only if there exists a strict, acyclic, simple labeling $\s$ on $Q$ such that, for each source $s$, the map $X^{{s}}\lr \prod_{x\in Q(s)\cap \widehat \s^{-1}(-)}X^{c(x)}$\footnote{This is the product over the maps $X(x):X^{{s}}\lr X^{c(x)}$.}  is a fibration. Under these equivalent conditions, $\holim X \simeq \lim X.$
\end{proposition}

\begin{proof}
Suppose the first condition, and let $R\in \mathrm{R}(Q)^{\mathrm{c}}$ be the given Reedy structure on $F(Q)$. Let $R'\in \mathrm{R}(Q)^{\mathrm{c}}$ be an ideal Reedy structure such that $R\leq R'$. Then $R'$ is also adequate, so the corresponding labeling $P^{-1}(R')\in \max(L^{\mathrm{c}})$ on $Q$ is adequate. Hence, by \cref{astrict} and \cref{equal}, $P^{-1}(R')$ is the extension $\widehat \s$ of an acyclic strict simple labeling $\s$ on $Q$. Since $X$ is also Reedy fibrant under $R'$, the second condition holds by \cref{fibrantcheck}.

Suppose the second condition. Again by \cref{fibrantcheck}, $X$ is Reedy fibrant with respect to $P(\widehat \s)$, which has cofibrant constants because $\widehat\s\in L^{\mathrm{c}}$.
\end{proof}

In conclusion, given a finite quiver $Q$ with diameter $1$, the maximal Reedy structures on $F(Q)$ with cofibrant constants\footnote{Those for which there does not exist a Reedy structure with cofibrant constants with strictly weaker fibrancy conditions for objects.} are \textit{among} those which correspond to the acyclic strict labelings on the simplification $Q^{\mathrm{s}}$ of $Q$, all of which can be found as the labelings induced by linear orderings on the set of sources. The set of these maximal Reedy structures yield all possible variants of homotopy limit calculations that fit into our framework of Reedy model structures.

Finally, we translate \cref{recognize} to the relevant case where $Q$ is graphic.\\

\begin{construction}
Given a graph $G$, fix an orientation $\mathfrak o$ on the simplification $G^{\mathrm{s}}$, and fix a vertex $v\in V$. Let $S({v})=S_{G}(v)\sube E$ be the set of self-loops of $v$ in $G$. For each $e\in S({v})$, write $e_{v,1}$ and $e_{v,2}$ for the two arrows $v\to e$ in $Q$\footnote{The choice of ordering among these two will not matter.}. Let $E_{\mathfrak o}(v)\sube E^{\mathrm{s}}$ be the set of edges whose $\mathfrak o$-target is $v$. For each $e\in E_{\mathfrak o}(v)$, recall that $(v,e)$ denotes the unique edge $v\to e$ in $Q$. Define the following set of arrows in $Q$: 
\[A_{\mathfrak o}(v)\coloneqq\{(v,e)\mid e\in E_{\mathfrak o}(v)\}\cup \{e_{v,1},e_{v,2}\mid e\in S({v})\}\sube Q_{1}.\]
Given a diagram $X\in \ob(\mc M^{F(Q)})$, define $\vp_{\mathfrak o}^{X}(v)$ to be the map $X^{v}\lr \prod_{x \in A_{\mathfrak o}(v)}X^{c(x)}$ in $\mc M$ given by the product over the maps $X(x):X^{v}\lr X^{c(x)}$.\\
\end{construction}

\begin{theorem} \label{Reedy}
Let $\mc M$ be a model category in which all projections are fibrations. Fix a graph $G$ and a diagram $X\in \ob(\mc M^{F(Q(G))})$. Then $X$ is fibrant in some Reedy model structure with cofibrant constants if and only if there exists an acyclic orientation $\mathfrak o$ on $G^{\mathrm{s}}$ such that, for each vertex $v$, the map $\vp_{\mathfrak o}^{X}(v)$ is a fibration. Under these equivalent conditions, $\holim X \simeq \lim X.$
\end{theorem}
\begin{proof}
Write $Q=Q(G)$. Since $Q^{\mathrm{s}}=Q(G^{\mathrm{s}})$, a strict labeling $\s\in \s(Q^{\mathrm{s}})$ is identified with an orientation $\mathfrak o$ on $G^{\mathrm{s}}$. Then for a vertex $v$ of $G$, we have $Q(v)\cap \widehat \s^{-1}(-) = A_{\mathfrak o}(v)$, and the result follows from \cref{recognize}.
\end{proof}

%%%%%%%%%% PATH DIAGRAM CALCULATIONS%%%%%%%
\subsection{Homotopy limits via path objects}
Fix a model category $\mc M$ in which all projections are fibrations. For each $x\in \ob(\mc M)$, we fix a path object $P(x)$ with the canonical map $P(x)\overset{\pi_{1}\times \pi_{2}}\lr  x\times x$.

The \textbfit{subdivision} of a quiver $Q$ is the quiver $\olin Q\coloneqq Q(U(Q))$. Pictorially, each arrow $s\overset{e}\lr t$ in $Q$ changes in $\olin Q$ to $s\overset{(s,e)}\lr e\overset{(t,e)}\longleftarrow t$. We also use the notation $e_{v}\coloneqq (v,e)$. In particular, if $Q=Q(G)$ is graphic, $\olin Q$ is a double subdivision. Starting in $G$, then $Q$, and then $\olin Q$, we have the local progression of an edge $e$:
% https://q.uiver.app/#q=WzAsMTIsWzcsMCwidl8xIl0sWzgsMCwiZV97dl8xfSJdLFs5LDAsImUiXSxbMTAsMCwiZV97dl8yfSJdLFsxLDAsInZfMiJdLFswLDAsInZfMSJdLFsyLDAsIlxccmlnaHRzcXVpZ2Fycm93Il0sWzUsMCwidl8yIl0sWzMsMCwidl8xIl0sWzQsMCwiZSJdLFs2LDAsIlxccmlnaHRzcXVpZ2Fycm93Il0sWzExLDAsInZfMiJdLFswLDEsIih2LGVfe3ZfMX0pIl0sWzIsMSwiKGUsZV97dl8xfSkiLDJdLFsyLDMsIihlLGVfe3ZfMn0pIl0sWzQsNSwiZSIsMix7InN0eWxlIjp7ImhlYWQiOnsibmFtZSI6Im5vbmUifX19XSxbNyw5LCJlX3t2XzJ9IiwyXSxbOCw5LCJlX3t2XzF9Il0sWzExLDMsIih2XzIsZV97dl8yfSkiLDJdXQ==
\[\begin{tikzcd}
	v & w & \rightsquigarrow & {v} & e & {w} & \rightsquigarrow & {v} & {e_{v}} & e & {e_{w}} & {w}
	\arrow["{(v,e_{v})}", from=1-8, to=1-9]
	\arrow["{(e,e_{v})}"', from=1-10, to=1-9]
	\arrow["{(e,e_{w})}", from=1-10, to=1-11]
	\arrow["e"', no head, from=1-2, to=1-1]
	\arrow["{e_{w}}"', from=1-6, to=1-5]
	\arrow["{e_{v}}", from=1-4, to=1-5]
	\arrow["{(w,e_{w})}"', from=1-12, to=1-11]
\end{tikzcd}\]

\begin{definition}\label{pathextension}
Fix a graphic quiver $Q=Q(G)$ and a diagram $X\in \ob(\mc M^{F(Q)})$. The \textbfit{path extension} of $X$ is the diagram $P (X)\in \ob( \mc M^{F(\olin Q)})$ obtained by performing the following insertion at each edge $e=\{v,w\}$.
% https://q.uiver.app/#q=WzAsOSxbNCwwLCJYXnt2XzF9Il0sWzUsMCwiWF5lIl0sWzYsMCwiUChYXmUpIl0sWzcsMCwiWF5lIl0sWzIsMCwiWF57dl8yfSJdLFswLDAsIlhee3ZfMX0iXSxbMSwwLCJYXmUiXSxbMywwLCJcXHJpZ2h0c3F1aWdhcnJvdyJdLFs4LDAsIlhee3ZfMn0iXSxbMCwxLCJYKGVfe3ZfMX0pIl0sWzIsMSwiXFxwaV8xIiwyXSxbMiwzLCJcXHBpXzIiXSxbNCw2LCJYKGVfe3ZfMn0pIiwyXSxbNSw2LCJYKGVfe3ZfMX0pIl0sWzgsMywiKHZfMixlX3t2XzJ9KSIsMl1d
\[\begin{tikzcd}
	{X^{v}} & {X^e} & {X^{w}} & \rightsquigarrow & {X^{v}} & {X^e} & {P(X^e)} & {X^e} & {X^{w}}
	\arrow["{X(e_{v})}", from=1-5, to=1-6]
	\arrow["{\pi_1}"', from=1-7, to=1-6]
	\arrow["{\pi_2}", from=1-7, to=1-8]
	\arrow["{X(e_{w})}"', from=1-3, to=1-2]
	\arrow["{X(e_{v})}", from=1-1, to=1-2]
	\arrow["{X(e_{w})}"', from=1-9, to=1-8]
\end{tikzcd}\]
Similarly, the \textbfit{padding} of $X$ is the diagram $\olin X\in \ob( \mc M^{F(\olin Q)})$ obtained from $P(X)$ by replacing $P(X^{e})$ by $X^{e}$ and each $\pi_{i}$ by $\id$.\\
\end{definition}

\begin{lemma} Let $Q=Q(G)$ be a finite graphic quiver, let $\mc M$ be a model category in which all projections are fibrations, and fix a diagram $X\in \ob(\mc M^{F(Q)})$. Then $\holim X \simeq \holim \olin X$.
\end{lemma} 
\begin{proof}
Choose $X'$ to be a fibrant replacement of $X$ in some Reedy model structure on $\mc M^{F(Q)}$ with cofibrant constants. The weak equivalence $X\lr X'$ of diagrams over $Q$ induces another weak equivalence $\olin X\lr \olin {X'}$ of diagrams over $\olin Q$. Let $\mathfrak o$ be the acyclic orientation on the simplification $G^{\mathrm{s}}$ corresponding to an ideal Reedy structure for which $X'$ is fibrant. Then for each vertex $v\in V$ of $G$, $\vp_{\mathfrak o}^{X'}(v)$ is a fibration in $\mc M$.

By definition, $\olin Q$ is graphic, being the quiver associated to the simple graph $U(Q)$. The orientation $\mathfrak o$ on $G^{\mathrm{s}}$ induces an orientation $\mathfrak o'$ on $U(Q)^{\mathrm{s}}= U(Q)$ as follows. Take $e=\{v,w\}$ in $E$ as before. If $v\neq w$ and we write the $\mathfrak o$-orientation of $e$ as above, then the corresponding $\mathfrak o'$-orientation in $U(Q)$ is
% https://q.uiver.app/#q=WzAsNixbMCwwLCJ2Il0sWzEsMCwidyJdLFsyLDAsIlxccmlnaHRzcXVpZ2Fycm93Il0sWzMsMCwidiJdLFs0LDAsImUiXSxbNSwwLCJ3Il0sWzAsMSwiZSJdLFszLDQsImVfdiJdLFs0LDUsImVfdyJdXQ==
\[\begin{tikzcd}
	v & w & \rightsquigarrow & v & e & w
	\arrow["e", from=1-1, to=1-2]
	\arrow["{e_v}", from=1-4, to=1-5]
	\arrow["{e_w}", from=1-5, to=1-6]
\end{tikzcd}\]
The remaining edges in $U(Q)$ come in parallel edge pairs
% https://q.uiver.app/#q=WzAsMixbMCwwLCJ3Il0sWzEsMCwiZSJdLFswLDEsImVfe3csMX0iLDAseyJjdXJ2ZSI6LTEsInN0eWxlIjp7ImhlYWQiOnsibmFtZSI6Im5vbmUifX19XSxbMCwxLCJlX3t3LDJ9IiwyLHsiY3VydmUiOjEsInN0eWxlIjp7ImhlYWQiOnsibmFtZSI6Im5vbmUifX19XV0=
\begin{tikzcd}
	w & e
	\arrow["{e_{w,1}}", bend right=-30, no head, from=1-1, to=1-2]
	\arrow["{e_{w,2}}"', bend right=30, no head, from=1-1, to=1-2]
\end{tikzcd}\!\!, each one corresponding to a loop $e=\{w,w\}$ in $G$. Define $\mathfrak o'$ on this pair as
% https://q.uiver.app/#q=WzAsNCxbMCwwLCJ2Il0sWzEsMCwiZSJdLFsyLDAsInYiXSxbMywwLCJlIl0sWzAsMSwiIiwwLHsiY3VydmUiOi0xfV0sWzAsMSwiIiwyLHsiY3VydmUiOjF9XSxbMywyLCIiLDIseyJjdXJ2ZSI6LTF9XSxbMywyLCIiLDAseyJjdXJ2ZSI6MX1dXQ==
\begin{tikzcd}
	w & e
	\arrow["{e_{w,1}}"', bend right=30, from=1-2, to=1-1]
	\arrow["{e_{w,2}}", bend right=-30, from=1-2, to=1-1]
\end{tikzcd}\!\!. This choice makes $e$ into a source, so $e$ will never be involved in a $\mathfrak o'$-directed cycle in $U(Q)$. Furthermore, since all $\mathfrak o$-directed arrows have been split into two $\mathfrak o'$-directed arrows, and $\mathfrak o$ is acyclic, $\mathfrak o'$ is also acyclic. Therefore, there is an adequate Reedy model structure on $\mc M^{F(\olin Q)}$ corresponding to $\mathfrak o'$. 

The diagram $\olin {X'}\in \ob(\mc M^{F(\olin Q)})$ is fibrant in this Reedy model structure if and only if for each vertex $p\in V\sqcup E$ of $U(Q)$, the resulting map $\vp_{\mathfrak o'}^{\olin{X'}}(p)$ arising from the set $A_{\mathfrak o'}(p) {\sube} U(Q)_{1}$ is a fibration in $\mc M$. Fixing a vertex $p$ of $U(Q)$, we now describe the set of arrows $A_{\mathfrak o'}(p)$. First, since $U(Q)$ is simple, the set $S_{U(Q)}(p)$ of self-loops is empty. So $A_{\mathfrak o'}(p)$ is determined by $E_{\mathfrak o'}(p)$, the set of edges in $U(Q)$ oriented by $\mathfrak o'$ with target $p$. 

\emph{Case 1:} $p=w\in V$ is a \textit{vertex} in $G$. The edges of $U(Q)$ containing $w$ all come from edges in $G$ containing $v$. For each self-loop $e\in S_{G}(w)$ of $w$ in $G$, there are two edges $e_{w,1},e_{w,2}\in U(Q)_{1}$ containing $w$ in $U(Q)$, and both of these have target $v$ under $\mathfrak o'$. For each non-self-loop edge $e\in E^{\mathrm{s}}$ which contains $w$, there is one edge $e_{w}\in U(Q)_{1}$ containing $w$ in $U(Q)$, and $e_{w}$ has target $w$ under $\mathfrak o'$ if and only if $e$ has target $w$ under $\mathfrak o$. So $E_{\mathfrak o'}(w)=\{e_{w}\mid e\in E_{\mathfrak o}(w)\}\cup \{e_{w,1},e_{w,2}\mid e\in S_{G}({w})=A_{\mathfrak o}(w)\}$ and $A_{\mathfrak o'}(w)= \{(w,x)\mid x\in A_{\mathfrak o}(w)\}.$

\emph{Case 2:} $p=e=\{v,w\}\in E$ is an \textit{edge} in $G$. Since $G$ is a graph, there are exactly two arrows $e_{v}=\{v,e\}$ and $e_{w}=\{w,e\}$ in $U(Q)$ containing $e$. If $e$ is a self-loop in $G$, then $e$ is an $\mathfrak o'$-source, so neither of these edges have target $p$. Then $A_{\mathfrak o'}(p)=E_{\mathfrak o'}(p)=\es$. If $e$ is not a self-loop, then $v\neq w$, and $\mathfrak o$ orients $e$ as, say $v\overset{e}\lr w$. Then $\mathfrak o'$ at $e$ looks like $v\overset{e_{v}}\to e\overset{e_{w}}\to w.$, so $E_{\mathfrak o'}(p)=\{e_{d(e)}\}$ and $A_{\mathfrak o'}(e)=\{(e,e_{d(e)})\}.$

We now describe the maps $\vp_{\mathfrak o'}^{\olin X}(p)$. When $p=e\in E$ is a loop, $A_{\mathfrak o'}(p)$ is empty, so there is no fibrancy condition (since all objects of $\mc M$ are fibrant). Otherwise, $A_{\mathfrak o'}(e)=\{(e,e_{d(e)})\}$, so $\vp_{\mathfrak o'}^{\olin X}(e)=\olin {X'}(e,e_{d(e)})=\id_{X'^{e}}$ is a fibration in $\mc M$. Finally, let $p=w\in V$ be a vertex in $G$. Recall that for any $x\in Q(w)$ we have $\olin {X'}(w,x)=X'(x)$, so $\vp^{\olin {X'}}_{\mathfrak o'}(w)= \prod_{x\in A_{\mathfrak o}(w)}\olin{X'}(w,x)=\prod_{x\in A_{\mathfrak o}(w)}X'(x)=\vp^{ {X'}}_{\mathfrak o}(w)$ is a fibration as well, by our assumption on $X'$. Therefore, $X'$ is fibrant in the adequate Reedy model structure corresponding to $\mathfrak o'$.

To conclude, we have $\holim X \simeq \holim X' \simeq \lim X' = \lim \olin{X'} \simeq \holim \olin {X'} \simeq \holim {\olin X}.$ The first and fifth equivalences follow because $X\lr X'$ and the induced map $\olin X\lr \olin{X'}$ are both weak equivalences. The second and fourth equivalences follow because $X'$ and $\olin {X'}$ are both fibrant in a Reedy model structure with cofibrant constants. The equality in the middle is elementary because the two diagrams differ only by inserted identity arrows.
\end{proof}

\begin{remark}This proof was local at each edge, so we can always pad any subset of edges without changing the homotopy limit. In particular, if we pad all self-loops, we end up with a diagram over a graph without self-loops. With this in mind, we can always assume without loss of generality that $G$ has no self-loops.\\
\end{remark}

\begin{lemma}\label{patheq}
Fix a finite graph $G=(V,E)$. Let $\mc M$ be a model category in which all projections are fibrations, and let $X\in \ob(\mc M^{F(Q(G))})$ be an arbitrary diagram. Then $\holim \olin {X} \simeq \lim P(X)$.
\end{lemma}
\begin{proof}
Write $G=(V,E)$. For each edge $e\in E$, the canonical weak equivalence $X^{e}\lr P(X^{e})$ extends with identities to a weak equivalence $\olin X\lr P(X)$. Then $\holim \olin {X} \simeq \holim P(X)$, so it remains to show $\holim P(X) \simeq \lim P(X)$. 

Recall that $\olin Q$ is the quiver associated to the graph $U(Q)$. Consider $Q$ as an orientation $\mathfrak o$ on $U(Q)$. Since 
$Q$ is graphic, $Q$ has diameter $1$, so $\mathfrak o$ is acyclic, and there is a corresponding adequate Reedy model structure on $\mc M^{F(\olin Q)}$. Note that a vertex $p$ in $Q$ has an arrow pointing into it if and only if $p=e\in E$ is an edge in $G$, and in this case there are exactly two arrows in $Q$ with target $e$, namely $v\overset{e_{v}}\to e\overset{e_{w}}\leftarrow w$, where $v$ and $w$ are the endpoints of $e$ in $G$ (allowing for $v=w$ or $v\neq w$). 

Therefore, a diagram $Y\in \ob(\mc M^{F(\olin Q)})$ is fibrant in this Reedy model structure if and only if, for each edge $e\in E$ in $G$, the induced map $\vp^{Y}_{\mathfrak o}(e)=Y(e,e_{v})\times Y(e,e_{w})$ is a fibration in $\mc M$. In the case $Y=P(X)$, this map is $\pi_{1}\times \pi_{2}$, which is a fibration. It follows that $P(X)$ is fibrant in the adequate Reedy model structure corresponding to the orientation $\mathfrak o$ from $Q$, so $\holim P(X) \simeq \lim P(X)$.
\end{proof}

\begin{lemma}\label{patheq2}
Take the same situation as above. Fix an acyclic subgraph $T\sube G$. Replace $\olin{X}$ (resp. $P(X)$) with $\widetilde{X}$ (resp. $\widetilde{P}(X)$), which is formed from $X$ by padding (resp. path extending) all edges except for those in $T$. Assume that $X(e_{v})$ is a fibration in $\mc M$ for each edge $e$ in $T$ and vertex $v\in e$.\footnote{This can be weakened to require fibrancy for only half of these edges, as the proof will show. Specifically, we only need there to exist a root vertex $\rh$ in $T$ such that $X(e_{v})$ is a fibration whenever $e=\{\hat v,v\}$ connects $v$ to its parent in the rooted tree $(T,\rho)$.\label{half}} Then the same result above holds. 
\end{lemma}

\begin{proof}
 We only need to check that $\widetilde{P}(X)$ is still fibrant. Fixing any vertex $\rh$ in $T$ as a root, introduce the quiver $\widetilde Q= \widetilde Q(G,\rh)$ which differs from $Q= Q(G)$ by not subdividing any edge in $T$ and orienting $T$ out from $\rh$. Then just as $\olin X$ and $P(X)$ are defined over $\olin Q\coloneqq Q(U(Q))$, so $\widetilde X$ and $\widetilde P(X)$ are defined over $\olin{\widetilde Q} Q(U(\widetilde Q))$. Since this latter quiver is graphic, we can proceed.
 
Having made these replacements, we mimic the above proof. Consider $\widetilde Q$ as an orientation on $U(\widetilde Q)$. Although $\widetilde Q$ may not have diameter 1 (in particular, may not be graphic), it is still acyclic: since $T$ is still a tree in $\widetilde Q$, any cycle would involve going through some edge $e_{v}\in (\widetilde Q)_{1}$ where $e\in E\sm T_{1}$, but this path must terminate in the sink $e\in (\widetilde Q)_{0}$.

We now check the fibrancy condition on the vertices of $U(\widetilde Q)$ equipped with the orientation $\mathfrak o$ coming from $\widetilde Q$. The vertices are  $U(\widetilde Q)_{0}=V\cup (E\setminus T_{1})$. Any $e'\in E\setminus T_{1}$ still has $\vp_{\mathfrak o}^{\widetilde P(X)}(e')=\pi_{1}\times \pi_{2}$ a fibration. In the previous proof, there was nothing to check on the vertices coming from $V$ because these were all sources under $\mathfrak o$. Now this is true for all elements of $V\sm (T_{0}\sm \{\rh\})$. Any vertex $v\in T_{0}\sm \{\rh\}$ has exactly one edge with target $v$, the edge $e$ with source $\hat v$, the parent of $v$ with respect to the root $\rh$. So $A_{\mathfrak o}(v)=\{e\overset{e_{v}}\to v\}$ and $\vp_{\mathfrak o}^{\widetilde P(X)}(v)=X(e_{v})$, which is a fibration by assumption.
\end{proof}

\begin{proposition}\label{pathresult}
Let $Q$ be a graphic quiver. Let $\mc M$ be a model category in which all projections are fibrations, and let $X\in \ob(\mc M^{F(Q)})$ be an arbitrary diagram. Then $\holim {X} \simeq \lim P(X)$. The same result holds if we choose an acyclic subgraph $T\subseteq G$ along whose edges we don't path extend, as long as each $X(e_{w})$ from $T$ is a fibration (or only half of these, see \cref{half} ).\label{maincor}
\end{proposition}
\begin{proof}
Combine the previous three results.
\end{proof}

We will see in Section 2 that when $\mc M$ is the category of dg-categories, path objects have a concrete description, so we will obtain explicit formulas for $\holim X$. Most of the arrows in our diagrams will be fibrations, so by \cref{pathresult} we will be able to choose a forest's worth of edges to \textit{not} path extend. The resulting smaller path extension $\widetilde{P}(X)$ still has the property that $\holim X \simeq \lim \widetilde{P}(X)$.

%%%%%%%%%%KAN RESULT%%%%%%%%%%%%%
\subsection{Homotopy limits in steps}
We finish with a technical result that will be used in Section 3 to calculate homotopy limits in multiple steps, rather than all at once.\\

\begin{lemma}\label{kanresult}
Fix a diagram $D\in \ob(\mc M^{J})$ over any category $\mc M$. Let $I\subseteq J$ be a faithful embedding of $F(Q(T))$ for some acyclic graph $T$. Assume that there are no morphisms in $J\sm I$ with codomain in $I$. Define $D'\in \ob(\mc M^{J/I})$ by replacing $D|_{I}$ with $\lim D|_{I}$. Then $\lim D'=\lim D$.
\end{lemma}

\begin{proof}
The replacement is well-defined because the vertices in $T$ are all sources in $J$. The result will follow from the fact that right Kan extensions preserve limits once we show that $D'$ is the right Kan extension of $D$ along the quotient $p: J\lr J/I$. This in turn follows from the pointwise formula for Kan extensions.
\end{proof}

\begin{proposition}\label{steps}
Let $G$ be a graph and $\mc M$ a model category in which all projections are fibrations. Fix a diagram $D\in \ob(\mc M^{J=Q(G)})$. Let $I\subseteq J$ be the subquiver given by an acyclic subgraph $T\sube G$ as above. Suppose $D(e_{v})$ is a fibration for each edge $e$ in $T$ and vertex $v\in e$.\footnote{Or only half of these, see \cref{maincor}.} Define $D'\in \ob(\mc M^{J/I})$ by replacing $D|_{I}$ with $\lim D|_{I}$. Then $\holim D'\simeq \holim D$.
\end{proposition}
\begin{proof}

First form the path extensions $\widetilde P(D)$ and $P(D')$, where the former is obtained by path extending all edges except for $e$. By \cref{maincor}, we have $\holim D\simeq\lim \widetilde P(D)$ and $\holim D'\simeq\lim P(D')$. It now remains to show that these two limits are equal. 

This will follow from \cref{kanresult} once we show that $\widetilde P(D)$ stands in the same relation to $ P(D')$ as $D$ stands to $D'$, meaning $P(D')=\widetilde P(D)'$. Of course this is true: in $D\rightsquigarrow D'\rightsquigarrow P(D')$, we first collapse $T$ and then path extend all remaining edges, while in $D\rightsquigarrow \widetilde P(D)\rightsquigarrow \widetilde P(D)'$, we first path extend all edges other than $T$ and then collapse $T$. The operations of collapsing $T$ and path extending all other edges do not interact because the latter does not change the diagram on the original vertices of $G$.
\end{proof}

%%%%%%%%%%%%%%%%%%% SECTION 2%%%%%
\section{Homotopy limits of dg-categories}
%%%%%%%%%%%%%%%%%%%%%%%%%%%%%%
In \cite[4.1.13]{K}, Karaba\c{s} gives an explicit formula for homotopy pullbacks of dg-categories. Since our diagrams are made of pullbacks, we will generalize this formula. In Section 1, we proved that path extending is a valid way of representing homotopy limits as limits, and in this section we calculate that representation explicitly. We give straightforward generalizations of Karaba\c{s}' Lemmas, and then finish with local manipulations of diagrams that will be crucial for Section 3.

\subsection{Differential graded categories}
Throughout the rest of this paper, we fix a commutative unital ring $k$ and a cyclic abelian group $R\in\{\zb/n\mid n\in \mathbb Z_{\geq 0}\}$. We will usually think of the case $R=\mathbb Z$, but our final invariant will be $R=\mathbb Z/2$-graded. A \textbfit{dg-category} is a category enriched over cohomologically $R$-graded chain complexes over $k$. We will write $\mc C_{\mathrm{dg}}^{(n)}(k)$ for the dg-category of $\zb/n$-graded chain complexes.\\

\begin{definition}
A dg-category is \textbfit{pointed} if it has a distinguished zero object 0. A dg-functor between pointed dg-categories is \textbfit{pointed} if it respects 0. Fixing $k$ and $R=\zb/n$, the category of pointed dg-categories and pointed dg-functors will be denoted $\dgcat_{k}^{(n)}$, or simply $\dgcat$.\\
\end{definition}

\begin{remark}
In the literature, a general dg-category is usually not required to be pointed, but all of the dg categories we use in this paper will be pointed. Similarly, the few dg-functors we employ are also pointed. So from here on we will omit the adjective ``pointed,'' taking it to be implicit.\\
\end{remark}

\begin{proposition}[Tabuada \cite{T}]\label{tabuadathm}There exists a unique model structure on $\dgcat_{k}^{(1)}$ whose weak equivalences are the quasi-equivalences, and whose \textbfit{fibrations} are those full dg-functors $F:\mc A\lr \mc B$ such that $H^{0}F$ is an isofibration: any isomorphism $F(a)\to b$ in $H^{0}\mc B$ lifts to an isomorphism $a\to a'$ in $H^{0}\mc A$.\\
\end{proposition}

\begin{remark}
Although Tabuada's theorem is formulated for dg-categories graded by $R=\zb$, by \cite{DK}, the same result holds for $R=\zb/2$, which suffices for our applications. Additionally, though the modern form of Tabuada's theorem is given for the category of unpointed dg-categories, the original proof was for pointed dg-categories, so the result holds for our purposes.
\end{remark}

All projections in $\dgcat_{k}^{(n)}$ are fibrations, so our results from Section 1 will apply.

Let $ \mc B$ and $\mc C$ be dg-categories. We write $\mathrm{Fun}(\mc B,\mc C)$ for the dg-category of unital $A_{\infty}$-functors from $\mc B$ to $\mc C$.\footnote{Since the only categories $\mathrm{Fun}(\mc B,\mc C)$ we will look at are those for which $\mc B$ is a very simple quiver, we choose to define each instance of $\mathrm{Fun}(\mc B,\mc C)$ explicitly rather than give the full definition of an $A_{\infty}$-functor. See \cite{Sei} for more on $A_{\infty}$-functors.} 

\begin{example}
If $\mc C=\mc C_{\tx{dg}}^{(n)}(k)$, then $\mathrm{Mod}(\mc B):=\mathrm{Fun}(\mc B,\mc C_{\tx{dg}}^{(n)}(k))$ is the dg-category of $A_{\infty}$-modules over $\mc B$.\\
\end{example}

\begin{example} When $\mc B$ is the free dg-category on the $A_{2}$-quiver, an object of $\mc B^{\to}\coloneqq\mathrm{Fun}(A_{2},\mc C)$ is $X=(X_{1}\overset{x}\to X_{2}),$ where $x\in Z^{0}\mc C(X_{1},X_{2})$ The morphism complexes are \[\mc C^{\to}(X,Y)^{k}\coloneqq \mc C(X_{1},Y_{1})^{k}\oplus \mc C(X_{1},Y_{2})^{k-1}\oplus \mc C(X_{2},Y_{2})^{k},\] with differential $d(f_{1},h,f_{2})=(df_{1},dh +(-1)^{k}(yf_{1}-f_{2}x),df_{2})$ and composition $(f_{1}',h',f_{2}')(f_{1},h,f_{2})=(f_{1}'f_{1},f_{2}'h+(-1)^{k}h'f_{1},f_{2}'f_{2})$.
\end{example}

We will often implicitly use the following characterization, due to Karaba\c{s}, of homotopy equivalences in $\mc C^{\to}$.\\

\begin{lemma}[{\cite[Lemma 4.1.4]{K}}]\label{simpler}
Let $\mc C$ be a dg-category and fix objects $X=(X_{1}\overset{x}\to X_{2})$ and $Y=(Y_{1}\overset{y}\to Y_{2})$ in  $\mc C^{\to}$. For each $i=1,2$, let $f_{i}:X_{i}\lr Y_{i}$ be a homotopy equivalence\footnote{A \textbfit{homotopy equivalence} in a dg-category $\mc C$ is a morphism in $Z^{0}\mc C$ which becomes an isomorphism in $H^{0}\mc C$.}. Then $f_{2}x\sim yf_{1}$\footnote{The notation $x\sim y$ means that $x-y$ is exact.} if and only if there exists a map $h\in \mc C(X_{1},Y_{2})^{-1}$ such that $(f_{1},h,f_{2})\in \mc C^{\to}(X,Y)^{0}$ is a homotopy equivalence.
\end{lemma}

We use $\rho_{1},\rho_{2}:\mc C^{\to}\lr \mc C$ to denote the projection functors, which are fibrations. The \textbfit{path object} $P(\mc C)$ of a dg-category $\mc C$ is the full dg-subcategory of $\mc C^{\to}$ spanned by the homotopy equivalences of $\mc C$. In particular, we have the canonical dg-functors $\mc C\overset{i}\lr P(\mc C)\overset{\pi_{1}\times \pi_{2}}\lr \mc C\times \mc C,$ where $i$ acts by $X\mapsto \id_{X}$, $f\mapsto (f,0,f)$, and $\pi_{j}\coloneqq \rh_{j}|_{P(\mc C)}$. The functor $i$ is a quasi-equivalence, and $\pi_{1}\times \pi_{2}$ is a fibration.

We conclude by verifying that it is sufficient for equivalences of diagrams in $\dgcat$ to commute only up to natural homotopy equivalence. This will be useful for our subsequent calculations in which we will have multiple homotopic models for a given dg-functor. \\

\begin{lemma}\label{right} Two dg-functors $F,G:\mc B\lr\mc C$ are right homotopic in the Tabuada model structure if they are naturally homotopy equivalent.
\end{lemma}

\begin{proof} A homotopy $\ph:\mc B\lr P(\mc C)$ realizing $F\overset{r}\sim G$ via the path object $P(\mc C)$ is the following data: for each $B\in \ob(\mc B)$, a homotopy equivalence $\ph_{B}:FB\to G$, and for each $f\in \mc B(B,B')^{n}$, a morphism $\ph(f)\in P(\mc C)(\ph_{B},\ph_{B'})^{n}$ of the form $(Ff, \ph(f)_{0},Gf)$ for some $\ph(f)_{0}\in \mc C(FB,GB')^{n-1}$. To be a dg-functor, these data must satisfy (i) identities: $\ph(\id_{B})_{0}=0$; (ii) composition: $\ph(g\circ f)_{0}=Gg \circ \ph(f)_{0}+(-1)^{n}\ph(g)_{0}\circ Ff$; and (iii) chain map: $f\mapsto \ph(f)_{0}$ is linear, and $\ph(df)_{0}=d\ph(f)_{0}+(-1)^{n}(\ph_{B}\circ Ff-Gf\circ \ph_{B})$. Any natural homotopy equivalence $\ph:F\lr G$ immediately yields a right homotopy by setting $\ph(f)_{0}=0$ for all $f$.
\end{proof}

\begin{proposition}\label{onthenose} Fix two diagrams $X,Y$ in $\dgcat^{J}$ for some shape $J$. If we specify dg-functors $\ph\coloneqq(\ph^{\a}:X^{\a}\lr Y^{\a})_{\a\in\ob(J)}$ which commute with $X$ and $Y$ only up to natural homotopy equivalence, then $\ph$ yields a legitimate map in $\Ho(\dgcat)^{J}$.
\end{proposition}
\begin{proof}
By \cref{right}, squares in $\dgcat$ which commute up to natural homotopy equivalence also commute up to right homotopy, so by \cite[5.10]{DS} the image of such a square under the localization $\g:\dgcat\lr \Ho(\dgcat)$ commutes in $\Ho(\dgcat)$. So the association $\a\mapsto \g(\ph^{\a})$ defines a morphism $\g(X)\lr \g(Y)$ in $\Ho(\dgcat)^{J}$.
\end{proof}

The next result will often be used implicitly in what follows.\\
\begin{corollary}\label{abstractsimple}
Let $J$ be the free category on a quiver. Fix diagrams $X,Y$ in $\dgcat^{J}$. If there exist objectwise quasi-equivalences from $X$ to $Y$ such that the resulting squares commute up to natural homotopy equivalence, then $X\simeq Y$. In particular, $\holim X\simeq\holim Y$.
\end{corollary}
\begin{proof}
Immediate from \cref{commute2} and \cref{onthenose}.
\end{proof}

\subsection{Explicit formulas}

Here we generalize \cite[4.1]{K} to our setting. We will use these results in the rest of the paper, in particular to construct homotopy equivalences within dg-categories arising as the homotopy limit of a diagram.\footnote{The proofs are the same as in \cite{K}, so we claim no originality for this subsection.} 

In the following, a \textbfit{graphic diagram} in $\dgcat$ will mean an object $X$ in $\dgcat^{F(Q(G))}$ for some finite graph $G$. Note that the definition of path extension $P(X)$ of $X$ carries an ambiguity over where to put $\pi_{1}$ versus $\pi_{2}$ on each edge in $G$. This ambiguity is removed by specifying an orientation on $G$. If we fix an orientation, then given an edge $e$ in $G$, we write $e_{1}$ and $e_{2}$ for the source and target vertices of $e$, respectively. We can now justify an explicit formula for the homotopy limit of $X$, which is a straightforward generalization of \cite[4.1.13]{K}.\\

\begin{lemma}\label{formula}
Fix a finite graph $G=(V,E)$ and a graphic diagram $X$ in $\dgcat^{F(Q(G))}$. Given an orientation of $G$, the following dg-category $\mc L$ presents $\holim X$: Objects are
\[\ob(\mc L)=\{M=((M_{v})_{v\in V},(m_{e})_{e\in E})\mid M_{v}\in \ob(X^{v}), m_{e}\in X^{e}(M^{e_{1}},M^{e_{2}})\, \tx{ are homotopy equivalences}\},\]
where $M^{e_{i}}\coloneqq X(e_{i},e)(M_{e_{i}})$ (and $e$ is understood from context). Morphisms are
\[\mc L(M,N)^{k}=\{\m=((\m_{v})_{v\in V},(\m_{e})_{e\in E})\mid \m_{v}\in X^{v}(M_{v},N_{v})^{k},\m_{e}\in X^{e}(M^{e_{1}},N^{e_{2}})^{k-1}\},\]
with differential $d\m=((d\m_{v})_{v\in V},(d\m_{e}+(-1)^{k}(n_{e}\m^{e_{1}}-\m^{e_{2}}m_{e}))_{e\in E})$ and composition $\n\m=((\n_{v}\m_{v})_{v\in V},(\n^{e_{2}}\m_{e}+(-1)^{k}\n_{e}\m^{e_{1}})_{e\in E}),$
where $\m^{e_{i}}\coloneqq X(e_{i},e)(\m_{e_{i}})$ (and $e$ is understood from context). Also, $\id_{M}=((\id_{M_{v}})_{v\in V},(0)_{e\in E})$.
\end{lemma}

\begin{proof}
Limits of dg-categories are computed concretely, so the formulas given follow on the nose from \cref{pathresult} and the definition of path objects in $\dgcat$.
\end{proof}

Now fix $M,N\in \ob(\mc L)$ and $\m,\n\in \mc L(M,N)^{k}$.\\

\begin{lemma}[Homotopy Lemma, {\cite[4.1.15]{K}}]We have $\m\sim\n$ if and only if there exist $(\xi_{v}\in X^{v}(M_{v},N_{v})^{k-1})_{v\in V}$ such that $d\xi_{v}=\m_{v}-\n_{v}$ for all $v\in V$ and $(-1)^{k-1}(n_{e}\xi^{e_{1}}-\xi^{e_{2}}m_{e})\sim \m_{e}-\n_{e}$ for all $e\in E$.\\
\end{lemma}

\begin{lemma}[{\cite[4.1.16]{K}}]
Suppose $d\m=0$. If for each $v\in V$, $\m_{v}$ is a homotopy equivalence with homotopy inverse $\n_{v}$, then there exist $(\n_{e}\in X^{e}(M^{e_{1}},N^{e_{2}})^{0})_{e\in E}$ such that $d((\n_{v})_{v\in V},(\n_{e})_{e\in E})=0$.\\
\end{lemma}

\begin{lemma}[Equivalence Lemma, {\cite[4.1.17]{K}}] The morphism $\m$ is a homotopy equivalence if and only if $d\m=0$, $\m_{v}$ is a homotopy equivalence for each $v\in V$, and $n_{e}\m^{e_{1}}\sim \m^{e_{2}}m_{e}$ for all $e\in E$.\\
\end{lemma}

\begin{proposition}
Suppose only the data $(\m_{v}\in X^{v}(M_{v},N_{v})^{0})_{v\in V}$ is given, where each $\m_{v}$ is a homotopy equivalence. If in addition, $n_{e}\m^{e_{1}}\sim \m^{e_{2}}m_{e}$ for all $e\in E$, then there exist 
$(\m_{e}\in X^{e}(M^{e_{1}},N^{e_{2}})^{-1})_{e\in E}$ such that $\m\coloneqq((\m_{v})_{v\in V},(\m_{e})_{e\in E})\in \mc L(M,N)^{0}$ is a homotopy equivalence.\footnote{Compare to \cref{simpler} and \cref{abstractsimple}.}
\end{proposition}

\begin{proof}
Set $\m_{e}$ so that $n_{e}\m^{e_{1}}- \m^{e_{2}}m_{e}=-d\m_{e}$. Then $d\m=0$ by construction, and the result follows from the Equivalence Lemma.
\end{proof}

Therefore, to prove that two objects in $\mc L$ are homotopy equivalent, it suffices to supply homotopy equivalences $(\m_{v}\in X^{v}(M_{v},N_{v})^{0})_{v\in V}$ and check that $n_{e}\m^{e_{1}}\sim \m^{e_{2}}m_{e}$ for each $e\in E$.\\

\begin{proposition}\label{simplify}
Let $G$ be a finite graph and $D$ a diagram in $\dgcat^{F(Q(G))}$. Fix a tree $T\sube G$, and assume that for each edge in $T$, the corresponding cospan (coming from $Q(T)\sube Q(G)$) in $D$ has at least one fibration among its two arrows. Let $\mc L \coloneqq \holim D$ be our model for the homotopy limit of a graphic diagram of dg-categories. Define the dg-subcategory $\mc L'\sube \mc L$ given by the condition that, for each edge $f=\{u,w\}$ in $T$, objects $M$ satisfy $M^{u}=M^{w}$ and $m_{f}=\id$, and morphisms $\m$ satisfy $\m^{u}=\m^{w}$ and $\m_{f}=0$. The inclusion of this subcategory is a quasi-equivalence.
\end{proposition}
\begin{proof}
We use the notation of \cref{steps}. We have $\mc L=\holim \mc D'\simeq \holim D$ from \cref{steps}, so it suffices to identify $\mc L'$ with $\holim  D'$. Below is a sketch of an \textit{isomorphism} between these two categories.

The diagram $ D'$ is obtained by collapsing $T\sube G$ to a point, at which we place $\lim \mc D|_{T}$. The category $\lim  D|_{T}$ is precisely the subcategory of $\holim  D|_{T}$ obtained using the restrictions defining $\mc L'$. Then $\holim  D'$ is obtained by gluing $\lim  D|_{T}$ to all other vertices of $G$ using path objects. This is the same as including all other information in $\mc L'$, namely objects and morphisms coming from $G\sm T$.
\end{proof}

%%%%%%%%%%%LOCAL MOVES%%%%%%%%%%%
\subsection{Local moves}\label{localmoves}

We apply the results above to arboreal diagrams, meaning diagrams of dg-categories arising as the microlocal sheaf categories associated to arboreal singularities in the skeleta of Weinstein surfaces. These diagrams arise from the local picture presented in \cite{N1}. 

The categories and functors used in arboreal diagrams are very simple in this dimension. Fix once and for all a strongly pretriangulated dg-category $\mc A\in \ob(\dgcat_{k}^{(n)})$, which the reader, if desired may take to be $\mc C_{\tx{dg}}^{(n)}(k)$. In this setting, the mapping cone $C:\mc A^{\to}\lr \mc A$ is a dg-functor. The quasi-equivalence $T:\mc A^{\to}\lr \mc A^{\to}$ defined by $ {T(X_{1}\overset{x}\to X_{2})\coloneqq(X_{2}\to C(x))}$ has quasi-inverse $T'(X_{1}\overset{x}\to X_{2})\coloneqq(C(x)[-1]\to X_{1})$ and satisfies the natural homotopy equivalence $C\circ T \simeq [1]\circ \rho_{1}$.

Given a dg-functor $f:\mc C\lr \mc B$, we write $\mc C_{f}$ for the full dg-subcategory of $\mc C$ on the objects $X$ in $\mc C$ which satisfy $f(X)\simeq 0$. By \cref{formula} we have
\be\label{edge}\mc C_{f}\simeq \holim(\mc C\overset{f}\lr \mc B\longleftarrow0).\ee
We denote the canonical quasi-equivalence $\mc A\lr P(\mc A)=\mc A_{C}^{\to}$ by $j_{C}$, and we define quasi-equivalences $j_{\rh_{i}}:\mc A^{\to}\lr \mc A_{\rh_{i}}$ by $j_{\rh_{1}}(X)\coloneqq(0\to X)$ and $j_{\rh_{2}}(X)\coloneqq(X\to 0)$.\\

\begin{lemma} [Rotation Lemma]\label{rotationlemma}
Let $X$ be a graphic diagram in $\dgcat$ which at a trivalent vertex looks like the following, for some $i,j,k\in \zb$, with modification $X'$.
% https://q.uiver.app/#q=WzAsMTAsWzIsMSwiXFxtYXRoY2FsIEFfMiJdLFszLDEsIlxcbWF0aGNhbCBBXzEiXSxbMSwwLCJcXG1hdGhjYWwgQV8xIl0sWzEsMiwiXFxtYXRoY2FsIEFfMSJdLFswLDEsIlg6Il0sWzUsMSwiWCc6Il0sWzcsMSwiXFxtYXRoY2FsIEFfMiJdLFs4LDEsIlxcbWF0aGNhbCBBXzEiXSxbNiwyLCJcXG1hdGhjYWwgQV8xIl0sWzYsMCwiXFxtYXRoY2FsIEFfMSJdLFswLDEsIltrXVxcY2lyYyBDIl0sWzAsMiwiW2ldXFxjaXJjIFxccmhvXzEiLDJdLFswLDMsIltqXVxcY2lyYyBcXHJob18yIl0sWzYsOSwiW2ktMV1cXGNpcmMgQyIsMl0sWzYsOCwiW2pdXFxjaXJjIFxccmhvXzEiXSxbNiw3LCJba11cXGNpcmMgXFxyaG9fMiJdXQ==
\[\begin{tikzcd}
	& {\mc A} &&&&& {\mc A} \\
	{X:} && {\mc A^{\to}} & {\mc A} && {X':} && {\mc A^{\to}} & {\mc A} \\
	& {\mc A} &&&&& {\mc A}
	\arrow["{[k]\circ C}", from=2-3, to=2-4]
	\arrow["{[i]\circ \rho_1}"', from=2-3, to=1-2]
	\arrow["{[j]\circ \rho_2}", from=2-3, to=3-2]
	\arrow["{[i-1]\circ C}"', from=2-8, to=1-7]
	\arrow["{[j]\circ \rho_1}", from=2-8, to=3-7]
	\arrow["{[k]\circ \rho_2}", from=2-8, to=2-9]
\end{tikzcd}\]
Then $X\simeq X'$, and in particular $\holim X \simeq \holim X'$.
\end{lemma}

\begin{proof}
By \cref{abstractsimple}, the quasi-equivalence $T:\mc A^{\to}\lr \mc A^{\to}$ defines a weak equivalence $X\lr X'$ with identities elsewhere.
\end{proof}

\begin{lemma}[Shift Lemma]
Fix $n\in \zb$. Let $X$ be a graphic diagram in $\dgcat$ which at a trivalent vertex looks as shown below, and let $X'$ be the shown local modification.
% https://q.uiver.app/#q=WzAsMTAsWzIsMSwiXFxtYXRoY2FsIEFfMiJdLFszLDEsIlxcbWF0aGNhbCBBXzEiXSxbMSwwLCJcXG1hdGhjYWwgQV8xIl0sWzEsMiwiXFxtYXRoY2FsIEFfMSJdLFswLDEsIlg6Il0sWzUsMSwiWCc6Il0sWzcsMSwiXFxtYXRoY2FsIEFfMiJdLFs4LDEsIlxcbWF0aGNhbCBBXzEiXSxbNiwyLCJcXG1hdGhjYWwgQV8xIl0sWzYsMCwiXFxtYXRoY2FsIEFfMSJdLFswLDEsIltrXVxcY2lyYyBDIl0sWzAsMiwiW2ldXFxjaXJjIFxccmhvXzEiLDJdLFswLDMsIltqXVxcY2lyYyBcXHJob18yIl0sWzYsOSwiW2ktMV1cXGNpcmMgQyIsMl0sWzYsOCwiW2pdXFxjaXJjIFxccmhvXzEiXSxbNiw3LCJba11cXGNpcmMgXFxyaG9fMiJdXQ==
\[\begin{tikzcd}
	& {\mc A} &&&&& {\mc A} \\
	{X:} && {\mc A^{\to}} & {\mc A} && {X':} && {\mc A^{\to}} & {\mc A} \\
	& {\mc A} &&&&& {\mc A}
	\arrow["{[k]\circ C}", from=2-3, to=2-4]
	\arrow["{ [i] \circ \rho_1}"', from=2-3, to=1-2]
	\arrow["{[j]\circ \rho_2}", from=2-3, to=3-2]
	\arrow["{[i+n]\circ \rho_{1}}"', from=2-8, to=1-7]
	\arrow["{[j+n]\circ \rho_{2}}", from=2-8, to=3-7]
	\arrow["{[k+n]\circ C}", from=2-8, to=2-9]
\end{tikzcd}\]
Then $X\simeq X'$, and in particular $\holim X \simeq \holim X'$.
\end{lemma}
\begin{proof}
Apply \cref{rotationlemma} $3n$ times.
\end{proof}

We finish with a local move which corresponds geometrically to a Weinstein homotopy which removes from the skeleton a univalent vertex and the corresponding $A_{2}$-singularity. This will be used in constructing arboreal Moves 0 and 1, to come in \cref{moves}. \\

\begin{lemma}\label{move1}
Let $X$ be a graphic diagram in $\dgcat$ with the following form at a trivalent vertex, and let $X'$ be the shown local modification.
% https://q.uiver.app/#q=WzAsMTIsWzEsMSwiXFxzcXVhcmUiXSxbMiwwLCJcXG1hdGhjYWwgQV8xIl0sWzIsMiwiXFxtYXRoY2FsIEFfMSJdLFszLDEsIlxcbWF0aGNhbCBBXzIiXSxbNCwxLCJcXG1hdGhjYWwgQV8xIl0sWzUsMSwiMCJdLFswLDEsIlg6Il0sWzcsMSwiWCc6Il0sWzgsMSwiXFxzcXVhcmUiXSxbOSwwLCJcXG1hdGhjYWwgQV8xIl0sWzksMiwiXFxtYXRoY2FsIEFfMSJdLFsxMCwxLCJcXG1hdGhjYWwgQV8xIl0sWzAsMV0sWzAsMl0sWzMsMSwiXFxyaG9fMSIsMl0sWzMsMiwiXFxyaG9fMiJdLFszLDQsIkMiXSxbNSw0XSxbOCw5XSxbOCwxMF0sWzExLDksIlxcbWF0aHJte2lkfSIsMl0sWzExLDEwLCJcXG1hdGhybXtpZH0iXV0=
\[\begin{tikzcd}
	&& {\mc A} &&&&&&& {\mc A} \\
	{X:\!\!\!\!\!\!\!\!} & \square && {\mc A^{\to}} & {\mc A} & 0 && {\!\!\!\!\!\!\!\!X':\!\!\!\!\!\!\!\!} & {\square} && {\mc A} \\
	&& {\mc A} &&&&&&& {\mc A}
	\arrow[from=2-2, to=1-3]
	\arrow[from=2-2, to=3-3]
	\arrow["{[j]\circ \rho_1}"', from=2-4, to=1-3]
	\arrow["{[k]\circ\rho_2}", from=2-4, to=3-3]
	\arrow["{[i]\circ C}", from=2-4, to=2-5]
	\arrow[from=2-6, to=2-5]
	\arrow[from=2-9, to=1-10]
	\arrow[from=2-9, to=3-10]
	\arrow["{[j]}"', from=2-11, to=1-10]
	\arrow["{[k]}", from=2-11, to=3-10]
\end{tikzcd}\]
(Here the square indicates the rest of $X$.) Then $\holim X'\simeq \holim X$. The same result holds with the tuple $(C,\rh_{1},\rh_{2})$ replaced by $(\rh_{1},\rh_{2},C)$, and similarly with $(\rh_{2},C,\rh_{1})$ except for the change $j\mapsto j+1$ in $X'$.

\end{lemma}

\begin{proof}
The general situation follows from the special case $i=j=k=0$, which we assume. From \cref{edge} we have
% https://q.uiver.app/#q=WzAsNyxbMywxLCJcXHNxdWFyZSJdLFs0LDAsIlxcbWF0aGNhbCBBXzEiXSxbNCwyLCJcXG1hdGhjYWwgQV8xIl0sWzUsMSwiXFxtYXRoY2FsIEFfMl5DIl0sWzIsMSwiXFxtYXRocm17aG9saW19Il0sWzEsMSwiXFxjb25nIl0sWzAsMSwiXFxvcGVyYXRvcm5hbWV7aG9saW19IFgiXSxbMCwxXSxbMCwyXSxbMywxLCJcXHJob18xIiwyXSxbMywyLCJcXHJob18yIl1d
\[\begin{tikzcd}
	&&&& {\mc A} \\
	{\operatorname{holim} X} & \simeq & {\mathrm{holim}\!\!\!\!\!\!\!\!\!\!\!\!} & {\square} && {\mc A^{\to}_C} \\
	&&&& {\mc A}
	\arrow[from=2-4, to=1-5]
	\arrow[from=2-4, to=3-5]
	\arrow["{\rho_1}"', from=2-6, to=1-5]
	\arrow["{\rho_2}", from=2-6, to=3-5]
\end{tikzcd}\]
There is a weak equivalence from $X'$ to this latter diagram given by $j_{C}:\mc A\lr \mc A_{C}^{\to}$ and identities elsewhere. The final sentence then follows from the Rotation Lemma.
\end{proof}

%%%%%%%%%%%%%% SECTION 3%%%%%%%%%%
\section{Application to Weinstein surfaces}
%%%%%%%%%%%%%%%%%%%%%%%%%%%%%%

Below we introduce our geometric setting and show how to represent Weinstein surface homotopies as a sequence of four moves on graphs. We then demonstrate the effect on diagrams of dg-categories that arise from each of our moves. Finally, we prove invariance of our construction $\mc L(W)$ and find a way to present it in cases of arbitrary genus.

\subsection{Arboreal Weinstein surfaces}

See \cite{CE} for the general theory of Weinstein manifolds and \cite{S} for the specific setting being assumed in this paper. In particular, given a Weinstein manifold $(W,\omega,V,\phi)$, rather than require that the Lyopunov function $\phi$ is Morse-Bott, we instead only require that the Liouville vector field $V$ is \textit{Morse-Bott*}, which allows for the components of the zero set of $V$ to have boundary. See \cite[Section 2.2]{S} for the precise definition.

Following \cite{S}, we describe the anatomy of the zero set of a Morse-Bott* vector field $V$. Given a connected component $Z$ of the zero set, if $\D:=\mathrm{Stab}(Z)$\footnote{The union of the stable manifolds of all points in $Z$.} is a Lagrangian submanifold-with-boundary, then we say that $\D$ is a \textbfit{bone} of the skeleton $\mathfrak X$ of $W$, with \textbfit{index} the index of any point in $Z$. Given bones $\D_{1}\neq \D_{2}$, if $H\coloneqq\olin{\D_{2}}\cap \D_{1}\neq\es$, then we say that $H$ is a \textbfit{joint} of $\D_{2}$ on $\D_{1}$. We say that $W$ is \textbfit{built of bones} if $\mathfrak{X}$ is a union of bones, i.e. $\mathrm{Stab}(Z)$ is a bone for each connected component $Z$ of the zero set of $V$.\\

\begin{proposition}[{\cite[Corollary 3.5]{S}}]
Every Weinstein manifold is Weinstein homotopic to one built of bones.
\end{proposition}

We now specify to Weinstein \textit{surfaces}, meaning Weinstein manifolds of dimension 2.\\

\begin{definition}\label{arborealdef}
A Weinstein surface $W$ with skeleton $\mathfrak X$ is \textbfit{arboreal} if $W$ is built of contractible bones and any joint in the skeleton has a neighborhood $U$ such that the pair $(U, U\cap \mathfrak X)$ is symplectomorpic to $((\mathbb R^{2},dx\wedge dy),\{y=0\}\cup\{x=0,y\geq0\})$ in a way such that the Liouville vector field is given locally by $y \p_{y}$. If the contractibility condition fails, then $W$ is \textbfit{non-generic arboreal}.\footnote{Given a non-generic arboreal surface, a small perturbation of its Weinstein structure will generically result in an arboreal surface, so we develop our theory mainly for the latter.}\\
\end{definition}

\begin{lemma}[{\cite[Theorem 3.11]{S}}]
Every Weinstein surface is Weinstein homotopic to an arboreal Weinstein surface.\\
\end{lemma}

\begin{definition}\label{assignmentconvention}  \label{graphdef}
An \textbfit{arboreal graph} is a finite graph $G$ where each interior vertex has valence 3, along with a labeling $(A, B, C)$ of the half edges at each interior vertex which satisfies: at any interior edge, if one half is labeled with a $C$, then so is the other half.\footnote{When we construct arboreal diagrams, this condition on the labeling will not matter. In particular, we can manipulate algebraic diagrams on any labeled trivalent graph.\label{caveat}}\\
\end{definition}

\begin{remark}\label{graphconvention}We will often indicate the strict ordering of the half edges of a vertex $v$ by drawing $G$ immersed in the plane, with exactly two half edges meeting smoothly at $v$. The remaining half edge is $C$, and $A$ is the half-edge on the counterclockwise side of $C$.\\
\end{remark}

\begin{definition}
Fix an arboreal or non-generic arboreal surface $W$ with skeleton $\mathfrak X$. The \textbfit{arboreal graph associated to $W$} is the arboreal graph $G(W)$ given by the underlying combinatorial graph of $\mathfrak{X}$ with the following orderings: at a 3-valent vertex, the half-edge corresponding to $\{x=0,y\geq0\}$ in \cref{arborealdef} is labeled $C$, and $A$ is the next half-edge in the positive direction according to the orientation of $W$.\\
\end{definition}

\begin{remark}
Every arboreal graph is the arboreal graph associated to some possibly non-generic arboreal surface. Accordingly, we will refer to any edge as a 1-bone if it has a $C$ label, and as a 0-bone otherwise.
\end{remark}

See \cref{arborealgraphs} for two examples of arboreal graphs, and one example of a non-arboreal graph.\\

\begin{figure}
\centering
\includegraphics[scale=0.74]{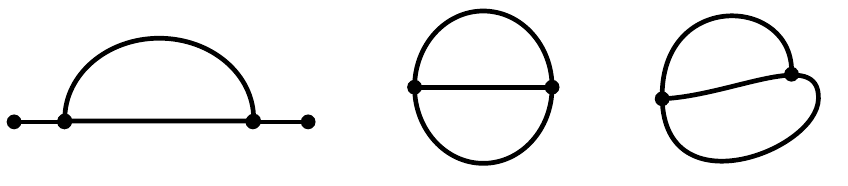}
\caption{We use the convention in \cref{graphconvention}. The graph on the left is the arboreal graph associated to an arboreal surface. The graph in the middle is the arboreal graph associated to a non-generic arboreal surface. The graph on the right is not arboreal, but by \cref{caveat} it will still encode a valid diagram when we come to \cref{arborealdiagramchapter1}.}
\label{arborealgraphs}
\end{figure}

\begin{definition}\label{moves1}
Two arboreal graphs with contractible 0-bones\footnote{In \cref{examples}, we will move between arboreal graphs that arise from arboreal and non-generic arboreal surfaces. This can be formalized by another type of move, the simplification of graphs shown in that section. This move is compatible with homotopy limits mostly through the application of \cref{move1}.} are \textbfit{arboreal equivalent} if one is obtained from the other by a sequence of the \textbfit{arboreal moves} defined in \cref{moves}.\\
\end{definition}

\begin{figure}
\centering
\includegraphics[scale=0.81]{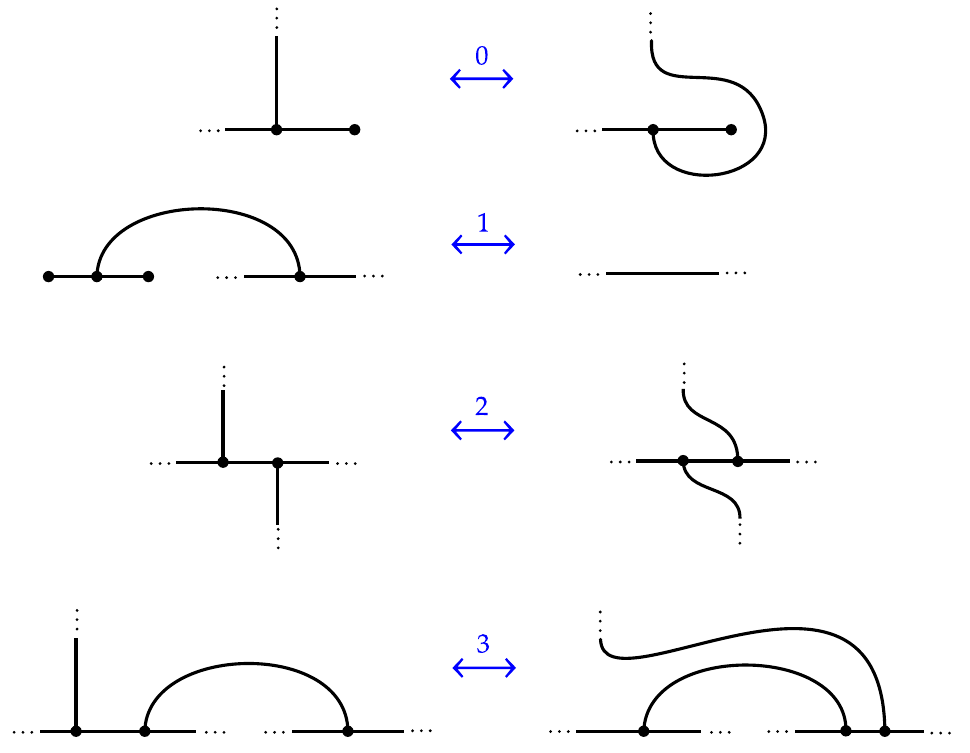}
\caption{The four arboreal moves on arboreal graphs, representing isotopy around a boundary, handle cancellation, isotopy of two 1-handles, and handle slide. The arboreal graph data is implied by \cref{graphconvention}.}
\label{moves}
\end{figure}

\begin{proposition}\label{geo}
The associated arboreal graph, defined up to arboreal moves, is a Weinstein homotopy invariant of arboreal Weinstein surfaces.
\end{proposition}
\begin{proof}
Fix arboreal Weinstein surfaces $ W_{0}$ and $ W_{1}$ which are Weinstein homotopic. Using a standard construction described in \cite[Section 12.3]{CE}, we can Weinstein homotope each $W_{i}$ into a Weinstein surface $M_{i}$ whose Lyopunov function is an honest Morse function. Furthermore, this perturbation can be chosen so that the index 0 critical points of the vector field in $M_{i}$ are in bijection with the index 0 bones of $W_{i}$. By transitivity of Weinstein homotopy, there is a Weinstein homotopy $( M_{s})_{s\in[0,1]}$ from $M_{0}$ to $M_{1}$, and by standard Morse theory we can arrange for the Lyopunov function of each $M_{s}$ to be Morse, except at finitely many values $0<t_{1}<\cdots<t_{n}<1$ of $s$.

This gives a family of handlebodies, which by standard Morse theory can be realized as a finite sequence of (i) handle cancellations between one handle each of indices 0 and 1, and (ii) handle slides between two handles of index 1, one such move happening at each $t_{i}$.\footnote{Recall that a Weinstein surface has no $2$-handles, so we don't see any handle cancellations between handles of indices 1 and 2.} Choose regular values $0=s_{0}<\cdots <s_{n}=1$ such that $s_{i}<t_{i+1}<s_{i+1}$.  Then $ M_{s_{0}}$ and $ M_{s_{i+1}}$ are two Morse handlebodies which differ by a single handle cancellation or handle slide.

The process of perturbing $W_{i}$ to $M_{i}$ can be reversed. By \cite[Proposition 3.4]{S}, we may start with an index 0 critical point $p\in M_{s_{i}}$, embed into $M_{s_{i}}$ a copy of the standard Weinstein surface $T^{*}D^{1}$ such that $p$ lies in the image of the zero section, and then interpolate the (normal to the zero section) Liouville vector field on the cotangent bundle with that in $M_{s_{i}}$ in such a way that introduces no new critical points. The result is an arboreal Weinstein surface $W^{(i)}$ which is Weinstein homotopic to $M_{s_{i}}$, and whose index 0 bones correspond to the index 0 critical points in $M_{s_{i}}$. In particular, we can choose $ W^{(0)}=W_{0}$ and $W^{(n)}=W_{1}$. 

For $1\leq i\leq n-1$, there is an ambiguity of how to define $W^{(i)}$. For each index 0 critical point $p$ in $M_{s_{i}}$, there are some $1$-handles attached to $p$. Other than the fixed cyclic ordering of the 1-handles, the choice of how to arrange them does not matter--the resulting skeleta are always arboreal equivalent via some iterations of Moves 0 and 2, as shown in \cref{mfld2}.

\begin{figure}
\centering
\includegraphics[scale=0.9]{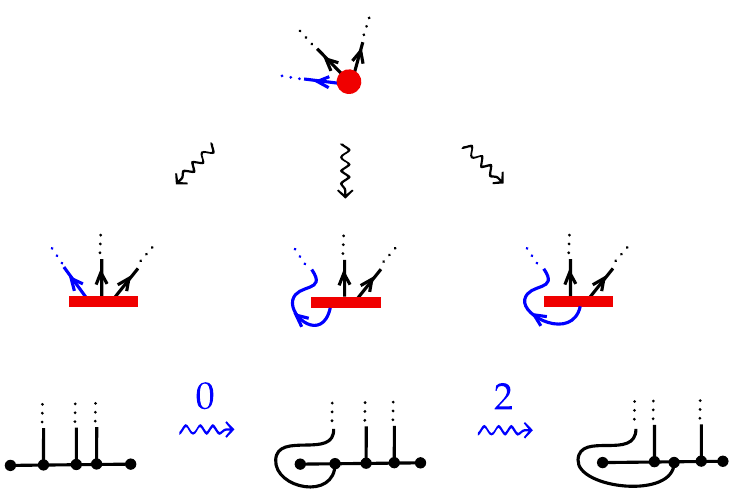}
\caption{Isotopies realize all ways to arrange 1-handles attached to a fixed 0-bone. The first row is $M_{s_{i}}$, the second is three possible choices for $W^{(i)}$, and the third shows how the associated arboreal graphs $G(W^{(i)})$ are arboreal equivalent via Moves 0 and 2.}
\label{mfld2}
\end{figure}

To prove the proposition, we will show that $G(W^{(i)})$ and $G(W^{(i+1)})$ are arboreal equivalent. 

Each handle cancellation or handle slide between $ M_{s_{i}}$ and $ M_{s_{i+1}}$ has a local model which can be drawn in the plane, and in this local model the passage $ M_{s_{i}}\mapsto W^{(i)}$ corresponds to thickening each index 0 critical point to an index 0 bone. So we can draw a picture in the plane that represents locally the passage between the skeleton of $ W^{(i)}$ to the skeleton of $ W^{(i+1)}$. It will be shown that these pictures, interpreted as their associated arboreal graphs, are arboreal equivalent.

Handle slide: Slide one 1-handle $h$ over another $h'$. There are two cases: $h'$ has both joints on the same $0$-handle, or not. 

\begin{figure}
\centering
\includegraphics[scale=0.57]{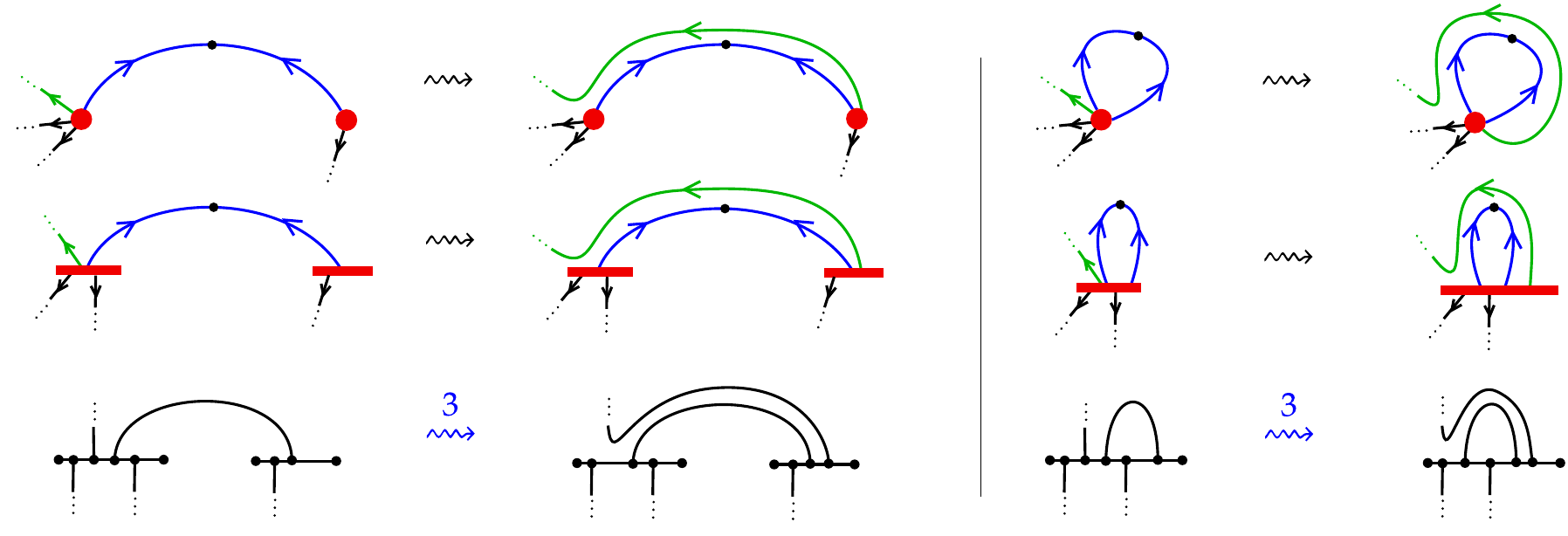}
\caption{The two types of handle slides, each shown first on Morse skeleta, then arboreal skeleta, and finally on the associated arboreal graphs via Move 3.}
\label{mfld3}
\end{figure}

Expanding in either situation, we see in \cref{mfld3} that $G(W^{(i)})$ and $G(W^{(i+1)})$ are related by Move 3.

Handle cancellation: We necessarily have a 1-handle $h$ attached to two different 0-handles, $A$ and $B$. Say we want to cancel $h$ with $B$. By handle-sliding all other 1-handles on $B$ to $A$, we can assume that $h$ is the only 1-handle on $B$. Then the handle cancellation on the arboreal graphs is the result of Move 1, shown in \cref{mfld1}.
\end{proof}

\begin{figure}
\centering
\includegraphics[scale=0.8]{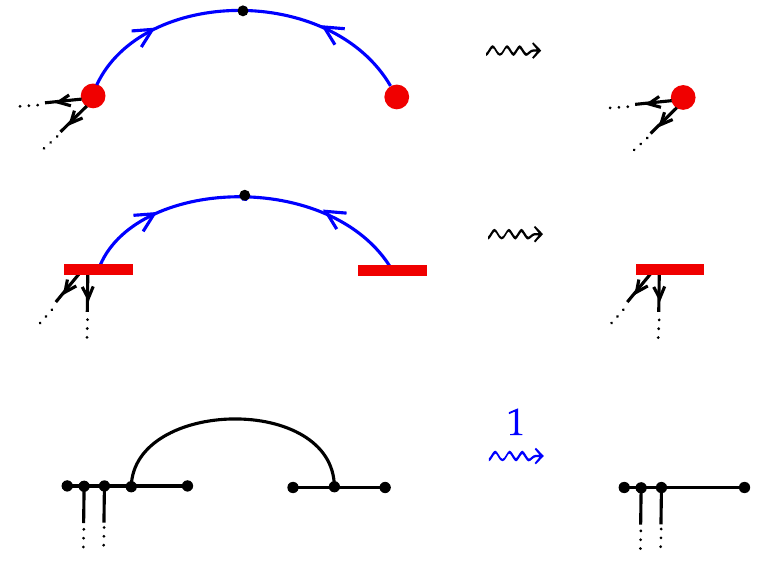}
\caption{Handle cancellation, shown on Morse skeleta, arboreal skeleta, and associated arboreal graphs via Move 1.}
\label{mfld1}
\end{figure}

\begin{corollary}\label{invcor}
Let $G\mapsto \mc L(G)$ be an assignment defined on the class of arboreal graphs which is invariant under arboreal moves. Then $\mc L$ gives a well-defined invariant of all Weinstein surfaces, given by $\mc L(W)\coloneqq \mc L(G(W'))$, where $W'$ is any arboreal surface Weinstein homotopic to $W$.
\end{corollary}

The remainder of this paper is dedicated to defining and proving invariance for an assignment $G\mapsto \mathcal L(G)$. The output will be an object in $\dgcat_{k}^{(2)}$, considered up to quasi-equivalence, so everything from here on out will use $R=\zb/2$-gradings.\\

%%%%%% MOVES 0 AND 1 %%%%%%%%%%%%%
\subsection{Moves 0 and 1}

Going forward we will encode arboreal diagrams by labeling the underlying arboreal graph as follows: each cospan $\mc A^{\to}\overset{[m]\circ f}\lr \mc A\overset{[n]\circ g}\longleftarrow \mc A^{\to}$ will be represented by the decorated edge  % https://q.uiver.app/#q=WzAsNixbMSwwLCJcXG1hdGhjYWwgQV8yIl0sWzIsMCwiXFxtYXRoY2FsIEFfMSJdLFswLDAsIlxcbWF0aGNhbCBBXzEiXSxbMywwLCI9Il0sWzQsMCwiXFxidWxsZXQiXSxbNSwwLCJcXGJ1bGxldCJdLFswLDEsImciXSxbMCwyLCJmIiwyXSxbNCw1LCJcXGJhciBmXFxxdWFkXFxiYXIgZyIsMCx7InN0eWxlIjp7ImhlYWQiOnsibmFtZSI6Im5vbmUifX19XV0=
\begin{tikzcd}
	 \bullet & \bullet
	\arrow["{ n\quad m}", shorten >=-8pt, shorten <=-8pt, no head, from=1-1, to=1-2]
\end{tikzcd},
where $f,g\in \{\rh_{1},\rh_{2}, C\}$ are recovered by \cref{graphconvention,localvertex}. An absent label will indicate no shifting.

In this subsection we derive the algebraic effect of Moves 0 and 1 on arboreal diagrams. A clockwise isotopy around the boundary of an index 0 bone induces a shift by 1 in the Cone functor:\\

\begin{lemma}[Move 0]\label{algebraicmove1.2}
The following local moves preserve homotopy limits.

%%%%%%%%%%%%%%%boundarymove%%%%%%%%%%%%%
\begin{figure}[H]
\centering
\includegraphics[scale=.9]{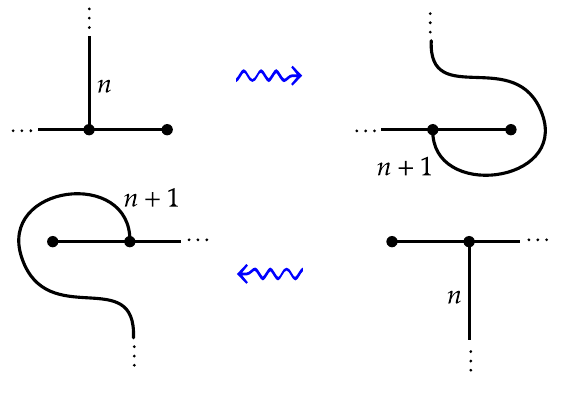}
\label{boundarymove}
\end{figure}
%%%%%%%%%%%%%%%boundarymove%%%%%%%%%%%%%

\end{lemma}

\begin{proof}
Each isotopy involves two applications of \cref{move1}: one simplifying, and one de-simplifying. The result follows because exactly one of these two applications involves the third case listed in \cref{move1}.
\end{proof}

\begin{lemma}[Move 1]\label{algebraicmove1}
The following local move preserves homotopy limits.

%%%%%%%%%%%%%%%handle1%%%%%%%%%%%%%%
\begin{figure}[H]
\centering
\includegraphics[scale=1]{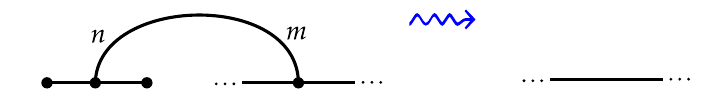}
\label{handle1}
\end{figure}
%%%%%%%%%%%%%%%handle1%%%%%%%%%%%%%%

\end{lemma}

\begin{proof}
Algebraically, we have the following move:
% https://q.uiver.app/#q=WzAsMTUsWzEsMiwiXFxtYXRoY2FsIEFfMSJdLFsyLDIsIlxcbWF0aGNhbCBBXzIiXSxbMCwyLCIwIl0sWzMsMiwiXFxtYXRoY2FsIEFfMSJdLFs0LDIsIjAiXSxbMiwxLCJcXG1hdGhjYWwgQV8xIl0sWzIsMCwiXFxtYXRoY2FsIEFfMiJdLFszLDAsIlxcbWF0aGNhbCBBXzEiXSxbMSwwLCJcXG1hdGhjYWwgQV8xIl0sWzQsMCwiXFxjZG90cyJdLFswLDAsIlxcY2RvdHMiXSxbNSwxLCJcXHJpZ2h0c3F1aWdhcnJvdyJdLFs2LDAsIlxcY2RvdHMiXSxbOCwwLCJcXG1hdGhjYWwgQV8xIl0sWzEwLDAsIlxcY2RvdHMiXSxbMSwwLCJbal1cXGNpcmMgXFxyaG9fMSJdLFsyLDBdLFsxLDMsIltrXVxcY2lyY1xccmhvXzIiLDJdLFs0LDNdLFsxLDUsIltpXVxcY2lyYyBDIiwyXSxbNiw1LCJbaSddXFxjaXJjIEMiXSxbNiw3LCJbaiddXFxjaXJjIFxccmhvXzEiXSxbNiw4LCJbaSddXFxjaXJjIFxccmhvXzIiLDJdLFs5LDcsImciLDJdLFsxMCw4LCJmIl0sWzEyLDEzLCJbaSddXFxjaXJjIGYiXSxbMTQsMTMsIltqJ11cXGNpcmMgZyIsMl1d
\[\begin{tikzcd}
	\cdots & {\mc A} & {\mc A^{\to}} & {\mc A} & \cdots && \cdots && {\mc A} && \cdots \\
	&& {\mc A} &&& \rightsquigarrow \\
	0 & {\mc A} & {\mc A^{\to}} & {\mc A} & 0
	\arrow["f", from=1-1, to=1-2]
	\arrow["{ \rho_2}"', from=1-3, to=1-2]
	\arrow["{ \rho_1}", from=1-3, to=1-4]
	\arrow["{[n]\circ C}", from=1-3, to=2-3]
	\arrow["g"', from=1-5, to=1-4]
	\arrow["{ f}", from=1-7, to=1-9]
	\arrow["{ g}"', from=1-11, to=1-9]
	\arrow[from=3-1, to=3-2]
	\arrow["{[m]\circ  C}"', from=3-3, to=2-3]
	\arrow["{ \rho_1}", from=3-3, to=3-2]
	\arrow["{\rho_2}"', from=3-3, to=3-4]
	\arrow[from=3-5, to=3-4]
\end{tikzcd}\]

Apply \cref{move1} to the the bottom $ \rho_{1}$, and then apply \cref{edge} to the resulting arrow $\id_{\mc A}$. The quasi-equivalence $\mc A^{\id }\simeq 0$ then yields the equivalent diagram 
% https://q.uiver.app/#q=WzAsNyxbMiwyLCIwIl0sWzIsMSwiXFxtYXRoY2FsIEFfMSJdLFsyLDAsIlxcbWF0aGNhbCBBXzIiXSxbMywwLCJcXG1hdGhjYWwgQV8xIl0sWzEsMCwiXFxtYXRoY2FsIEFfMSJdLFs0LDAsIlxcY2RvdHMiXSxbMCwwLCJcXGNkb3RzIl0sWzAsMV0sWzIsMSwiW25dXFxjaXJjIEMiXSxbMiwzLCJcXHJob18xIl0sWzIsNCwiXFxyaG9fMiIsMl0sWzUsMywiZyIsMl0sWzYsNCwiZiJdXQ==
\[\begin{tikzcd}
	\cdots & {\mc A} & {\mc A^{\to}} & {\mc A} & \cdots \\
	&& {\mc A} \\
	&& 0
	\arrow["f", from=1-1, to=1-2]
	\arrow["{\rho_2}"', from=1-3, to=1-2]
	\arrow["{\rho_1}", from=1-3, to=1-4]
	\arrow["{[n]\circ C}", from=1-3, to=2-3]
	\arrow["g"', from=1-5, to=1-4]
	\arrow[from=3-3, to=2-3]
\end{tikzcd}\]
Another application of \cref{edge} produces
% https://q.uiver.app/#q=WzAsNSxbMiwwLCJcXG1hdGhjYWwgQV8yXkMiXSxbMywwLCJcXG1hdGhjYWwgQV8xIl0sWzEsMCwiXFxtYXRoY2FsIEFfMSJdLFs0LDAsIlxcY2RvdHMiXSxbMCwwLCJcXGNkb3RzIl0sWzAsMSwiXFxyaG9fMSJdLFswLDIsIlxccmhvXzIiLDJdLFszLDEsImciLDJdLFs0LDIsImYiXV0=
\begin{tikzcd}
	\cdots & {\mc A} & {\mc A^{\to}_C} & {\mc A} & \cdots
	\arrow["f", from=1-1, to=1-2]
	\arrow["{\rho_2}"', from=1-3, to=1-2]
	\arrow["{\rho_1}", from=1-3, to=1-4]
	\arrow["g"', from=1-5, to=1-4]
\end{tikzcd}.
The functor $j_{C}$ now provides a weak equivalence to a zig-zag of functors between copies of $\mc A$. From here, whatever the ellipses contain, the two relevant pullbacks are fibrant, so a calculation of the limit using \cref{kanresult} yields the result.
\end{proof}

%%%%%%% MOVES 2 AND 3 %%%%%%%%%%%%%
\subsection{Moves 2 and 3}\label{gen2}

Here we present our main \cref{gen2} and its consequences, the algebraic effects of Moves 2 and 3. We often implicitly use the straightforward fact that homotopy limits are preserved by the local move
% https://q.uiver.app/#q=WzAsNSxbMCwwLCJcXGJ1bGxldCJdLFsyLDAsIlxcYnVsbGV0Il0sWzMsMCwiXFxyaWdodHNxdWlnYXJyb3ciXSxbNCwwLCJcXGJ1bGxldCJdLFs2LDAsIlxcYnVsbGV0Il0sWzAsMSwiW25dXFxjaXJjIGZcXHFxdWFkW21dXFxjaXJjIGciLDAseyJzdHlsZSI6eyJoZWFkIjp7Im5hbWUiOiJub25lIn19fV0sWzMsNCwiW24ra11cXGNpcmMgZlxccXVhZCBbbStrXVxcY2lyYyBnIiwwLHsic3R5bGUiOnsiaGVhZCI6eyJuYW1lIjoibm9uZSJ9fX1dXQ==
\begin{tikzcd}
	\bullet & \bullet & \rightsquigarrow & \bullet & \bullet
	\arrow["{n\,\,\,\quad m}", shorten <=-8pt, shorten >=-8pt, no head, from=1-1, to=1-2]
	\arrow["{n+k\,\,\, \,\,m+k}", shorten <=-8pt, shorten >=-8pt, no head, from=1-4, to=1-5]
\end{tikzcd}.
In particular, taking $k=-m$ shows that we can always absorb the two shifts on an edge into just one. However, for the sake of symmetry, we show all possible labelings in the following theorems. We state our next result with more generality than required by Move 2 because doing so will aid us in understanding Move 3.\\

\begin{lemma}[Generalized Move 2]\label{gen2}
The following move preserves homotopy limits, where $\D\coloneqq p-q$.
% https:\\q.uiver.app#q=WzAsMTQsWzEsMiwiXFxidWxsZXQiXSxbMCwyLCJcXGNkb3RzIl0sWzIsMiwiXFxidWxsZXQiXSxbMywyLCJcXGNkb3RzIl0sWzUsMiwiXFxkb3RzIl0sWzYsMiwiXFxidWxsZXQiXSxbNywyLCJcXGJ1bGxldCJdLFs4LDIsIlxcZG90cyJdLFsxLDAsIlMiXSxbNiwwLCJTJyJdLFsxLDMsIlxcdmRvdHMiXSxbMiwxLCJcXHZkb3RzIl0sWzcsMSwiXFx2ZG90cyJdLFs2LDMsIlxcdmRvdHMiXSxbMCwxLCIyIiwyLHsic3R5bGUiOnsiaGVhZCI6eyJuYW1lIjoibm9uZSJ9fX1dLFswLDIsIjFcXHF1YWQgMSIsMCx7InN0eWxlIjp7ImhlYWQiOnsibmFtZSI6Im5vbmUifX19XSxbMiwzLCIyIiwwLHsic3R5bGUiOnsiaGVhZCI6eyJuYW1lIjoibm9uZSJ9fX1dLFs1LDQsIjEiLDIseyJzdHlsZSI6eyJoZWFkIjp7Im5hbWUiOiJub25lIn19fV0sWzUsNiwiMlxccXVhZCAyIiwwLHsic3R5bGUiOnsiaGVhZCI6eyJuYW1lIjoibm9uZSJ9fX1dLFs2LDcsIjEiXSxbMCwxMCwiMyIsMCx7InN0eWxlIjp7ImhlYWQiOnsibmFtZSI6Im5vbmUifX19XSxbMiwxMSwiMyIsMix7InN0eWxlIjp7ImhlYWQiOnsibmFtZSI6Im5vbmUifX19XSxbNSwxMiwiMyIsMCx7InN0eWxlIjp7ImhlYWQiOnsibmFtZSI6Im5vbmUifX19XSxbNiwxMywiMyIsMCx7InN0eWxlIjp7IfmhlYWQiOnsibmFtZSI6Im5vbmUifX19XV0=
\[\begin{tikzcd}
	&& \vdots &&&&& \vdots \\
	\cdots & \bullet & \bullet & \cdots &\rightsquigarrow& \cdots & \bullet & \bullet & \dots \\
	& \vdots &&&&& \vdots
	\arrow["{i}"',shorten <=-8pt, no head, from=2-2, to=2-1]
	\arrow["{p\quad q}", shorten <=-8pt, shorten >=-8pt, no head, from=2-2, to=2-3]
	\arrow["{j}",shorten <=-8pt, no head, from=2-3, to=2-4]
	\arrow["{i}"',shorten <=-8pt, no head, from=2-7, to=2-6]
	\arrow["{p\quad q}", shorten <=-8pt, shorten >=-8pt, no head, from=2-7, to=2-8]
	\arrow["{j}",shorten <=-8pt, no head, from=2-8, to=2-9]
	\arrow["{n}",shorten <=-8pt, no head, from=2-2, to=3-2]
	\arrow["{m}"',shorten <=-8pt, no head, from=2-3, to=1-3]
	\arrow["{m+\D}",shorten <=-8pt, no head, from=2-7, to=1-8]
	\arrow["{n-\D}",shorten <=-8pt, no head, from=2-8, to=3-7]
\end{tikzcd}\]
\end{lemma}

We will prove \cref{gen2} in the next subsection. In particular, by taking all shifts but $n$ and $m$ to be $0$, we get a simple statement about the algebraic effect of Move 2: \textit{shifts do not change under these isotopies}.

With little more effort, we can derive the algebraic effect of Move 3 from that of Move 2. \\
\begin{lemma}\label{move3}
The following moves preserve homotopy limits, where $\D\coloneqq p-q$.  
% https://q.uiver.app/#q=WzAsMTUsWzAsMywiXFxjZG90cyJdLFsxLDMsIlxcYnVsbGV0Il0sWzIsMywiXFxidWxsZXQiXSxbMywzLCJcXGNkb3RzIl0sWzIsMSwiXFx2ZG90cyJdLFsxLDEsIlxcdmRvdHMiXSxbNSwzLCJcXGNkb3RzIl0sWzgsMywiXFxjZG90cyJdLFs3LDMsIlxcYnVsbGV0Il0sWzcsMiwiXFxidWxsZXQiXSxbNywxLCJcXHZkb3RzIl0sWzYsMSwiXFx2ZG90cyJdLFs0LDIsIlxccmlnaHRzcXVpZ2Fycm93Il0sWzEsMCwiUyJdLFs2LDAsIlMnIl0sWzAsMSwiMSIsMCx7InN0eWxlIjp7ImhlYWQiOnsibmFtZSI6Im5vbmUifX19XSxbMSwyLCIyXFxxdWFkIDEiLDAseyJzdHlsZSI6eyJoZWFkIjp7Im5hbWUiOiJub25lIn19fV0sWzIsMywiMiIsMCx7InN0eWxlIjp7ImhlYWQiOnsibmFtZSI6Im5vbmUifX19XSxbMiw0LCIzIiwwLHsibGFiZWxfcG9zaXRpb24iOjMwLCJzdHlsZSI6eyJoZWFkIjp7Im5hbWUiOiJub25lIn19fV0sWzEsNSwiMyIsMCx7ImxhYmVsX3Bvc2l0aW9uIjozMCwic3R5bGUiOnsiaGVhZCI6eyJuYW1lIjoibm9uZSJ9fX1dLFs4LDcsIjIiLDAseyJzdHlsZSI6eyJoZWFkIjp7Im5hbWUiOiJub25lIn19fV0sWzgsOSwiXFxvdmVyc2V0ezF9ezN9IiwwLHsic3R5bGUiOnsiaGVhZCI6eyJuYW1lIjoibm9uZSJ9fX1dLFs5LDEwLCIyIiwwLHsic3R5bGUiOnsiaGVhZCI6eyJuYW1lIjoibm9uZSJ9fX1dLFs5LDExLCIzIiwwLHsibGFiZWxfcG9zaXRpb24iOjMwLCJjdXJ2ZSI6LTIsInN0eWxlIjp7ImhlYWQiOnsibmFtZSI6Im5vbmUifX19XSxbOCw2LCIxIiwyLHsibGFiZWxfcG9zaXRpb24iOjMwLCJzdHlsZSI6eyJoZWFkIjp7Im5hbWUiOiJub25lIn19fV1d
\[\begin{tikzcd}
	& \vdots & \vdots &&&& \vdots & \vdots \\
	&&&& \rightsquigarrow &&& \bullet \\
	\cdots & \bullet & \bullet & \cdots && \cdots && \bullet & \cdots
	\arrow["{i}", shorten >=-8pt, no head, from=3-1, to=3-2]
	\arrow["{p\quad q}", shorten <=-8pt, shorten >=-8pt,no head, from=3-2, to=3-3]
	\arrow["{j}", shorten <=-8pt, no head, from=3-3, to=3-4]
	\arrow["{n}"{pos=0.3}, shorten <=-8pt, no head, from=3-3, to=1-3]
	\arrow["{m}"{pos=0.3}, shorten <=-8pt, no head, from=3-2, to=1-2]
	\arrow["{j}", shorten <=-8pt, no head, from=3-8, to=3-9]
	\arrow["{\overset{p}{q}}", shorten <=-8pt, shorten >=-8pt, no head, from=3-8, to=2-8]
	\arrow["{n+\D}"', shorten <=-8pt, no head, from=2-8, to=1-8]
	\arrow["{m-1}"{pos=0.3}, shorten <=-8pt, bend right = -45, no head, from=2-8, to=1-7]
	\arrow["{i-\D}"'{pos=0.5}, shorten <=-8pt, no head, from=3-8, to=3-6]
\end{tikzcd}\]
% https://q.uiver.app/#q=WzAsOCxbMCwyLCJ+Il0sWzUsMiwiXFxjZG90cyJdLFs4LDIsIlxcY2RvdHMiXSxbNCwxLCJcXHJpZ2h0c3F1aWdhcnJvdyJdLFs2LDIsIlxcYnVsbGV0Il0sWzYsMSwiXFxidWxsZXQiXSxbNiwwLCJcXHZkb3RzIl0sWzcsMCwiXFx2ZG90cyJdLFs0LDEsIltpLVxcRGVsdGFdXFxjaXJjMVxcLFxcLFxcLCIsMix7InN0eWxlIjp7ImhlYWQiOnsibmFtZSI6Im5vbmUifX19XSxbNCwyLCJbal1cXGNpcmMgMiIsMCx7ImxhYmVsX3Bvc2l0aW9uIjozMCwic3R5bGUiOnsiaGVhZCI6eyJuYW1lIjoibm9uZSJ9fX1dLFs0LDUsIlxcb3ZlcnNldHtbcF1cXGNpcmMyfXtbcV1cXGNpcmMzfSIsMix7InN0eWxlIjp7ImhlYWQiOnsibmFtZSI6Im5vbmUifX19XSxbNSw2LCJbbStcXERlbHRhXVxcY2lyYyAxIiwwLHsic3R5bGUiOnsiaGVhZCI6eyJuYW1lIjoibm9uZSJ9fX1dLFs1LDcsIltuXVxcY2lyYyAzIiwyLHsibGFiZWxfcG9zaXRpb24iOjIwLCJjdXJ2ZSI6Miwic3R5bGUiOnsiaGVhZCI6eyJuYW1lIjoibm9uZSJ9fX1dXQ==
\[\begin{tikzcd}
	&&&&&&&& \vdots & \vdots \\
	&&&&&& \,\,\,\,\,\,\rightsquigarrow && \bullet \\
	{~} &&&&&&& \cdots & \bullet && \cdots
	\arrow["{m}", shorten <=-8pt, no head, from=2-9, to=1-9]
	\arrow["{n+\D}"'{pos=0.2}, shorten <=-8pt, bend right=45, no head, from=2-9, to=1-10]
	\arrow["{\overset{p}{q}}"', shorten <=-8pt, shorten >=-8pt, no head, from=3-9, to=2-9]
	\arrow["{i-\Delta}"', shorten <=-8pt, no head, from=3-9, to=3-8]
	\arrow["{j}"{pos=0.3}, shorten <=-8pt, no head, from=3-9, to=3-11]
\end{tikzcd}\]
\end{lemma}

\begin{proof}Performing the Rotation Lemma is not compatible with the convention of labeling an arboreal graph with \textit{only} the shifts of functors, as we have been relying on \cref{graphconvention} to tell us the ordering of half edges at a trivalent vertex. We will instead explicitly indicate the half edge $C$ by drawing a bar over its respective shift.

With this temporary notation, the first move above is notated as
\[\begin{tikzcd}
	& \vdots & \vdots &&&& \vdots & \vdots \\
	&&&& \rightsquigarrow &&& \bullet \\
	\cdots & \bullet & \bullet & \cdots && \cdots && \bullet & \cdots
	\arrow["{i}", shorten >=-8pt, no head, from=3-1, to=3-2]
	\arrow["{p\quad q}", shorten <=-8pt, shorten >=-8pt, no head, from=3-2, to=3-3]
	\arrow["{j}", shorten <=-8pt,no head, from=3-3, to=3-4]
	\arrow["{\olin n}"{pos=0.3}, shorten <=-8pt, no head, from=3-3, to=1-3]
	\arrow["{\olin m}"{pos=0.3}, shorten <=-8pt, no head, from=3-2, to=1-2]
	\arrow["{j}", shorten <=-8pt, no head, from=3-8, to=3-9]
	\arrow["{\overset{p}{\olin q}}", shorten <=-8pt, shorten >=-8pt,no head, from=3-8, to=2-8]
	\arrow["{n+\D}"', shorten <=-8pt, no head, from=2-8, to=1-8]
	\arrow["{\olin{m-1}}"{pos=0.3}, shorten <=-8pt,bend right =-45, no head, from=2-8, to=1-7]
	\arrow["{i-\D}"',{pos=0.5}, shorten <=-8pt, no head, from=3-8, to=3-6]
\end{tikzcd}\]

Applying the Rotation Lemma to three of the four vertices above, we arrive at the following equivalent problem:
% https://q.uiver.app/#q=WzAsMTMsWzAsMiwiXFxjZG90cyJdLFsxLDIsIlxcYnVsbGV0Il0sWzIsMiwiXFxidWxsZXQiXSxbMywyLCJcXGNkb3RzIl0sWzIsMCwiXFx2ZG90cyJdLFsxLDAsIlxcdmRvdHMiXSxbNSwyLCJcXGNkb3RzIl0sWzgsMiwiXFxjZG90cyJdLFs3LDIsIlxcYnVsbGV0Il0sWzcsMSwiXFxidWxsZXQiXSxbNywwLCJcXHZkb3RzIl0sWzYsMCwiXFx2ZG90cyJdLFs0LDEsIlxccmlnaHRzcXVpZ2Fycm93Il0sWzAsMSwiWy0xXVxcY2lyYzMiLDAseyJzdHlsZSI6eyJoZWFkIjp7Im5hbWUiOiJub25lIn19fV0sWzEsMiwiMVxccXVhZCAxIiwwLHsic3R5bGUiOnsiaGVhZCI6eyJuYW1lIjoibm9uZSJ9fX1dLFsyLDMsIjIiLDAseyJzdHlsZSI6eyJoZWFkIjp7Im5hbWUiOiJub25lIn19fV0sWzIsNCwiMyIsMCx7ImxhYmVsX3Bvc2l0aW9uIjozMCwic3R5bGUiOnsiaGVhZCI6eyJuYW1lIjoibm9uZSJ9fX1dLFsxLDUsIjIiLDAseyJsYWJlbF9wb3NpdGlvbiI6MzAsInN0eWxlIjp7ImhlYWQiOnsibmFtZSI6Im5vbmUifX19XSxbOCw3LCIxIiwwLHsic3R5bGUiOnsiaGVhZCI6eyJuYW1lIjoibm9uZSJ9fX1dLFs4LDksIlxcb3ZlcnNldHsyfXsyfSIsMCx7InN0eWxlIjp7ImhlYWQiOnsibmFtZSI6Im5vbmUifX19XSxbOSwxMCwiMyIsMCx7InN0eWxlIjp7ImhlYWQiOnsibmFtZSI6Im5vbmUifX19XSxbOSwxMSwiWzFdXFxjaXJjMSIsMCx7ImxhYmVsX3Bvc2l0aW9uIjozMCwiY3VydmUiOi0yLCJzdHlsZSI6eyJoZWFkIjp7Im5hbWUiOiJub25lIn19fV0sWzgsNiwiWy0xXVxcY2lyYyAzIiwyLHsibGFiZWxfcG9zaXRpb24iOjMwLCJzdHlsZSI6eyJoZWFkIjp7Im5hbWUiOiJub25lIn19fV1d
\[\begin{tikzcd}
	& \vdots & \vdots &&&& \vdots & \vdots \\
	&&&& \rightsquigarrow &&& \bullet \\
	\cdots & \bullet & \bullet & \cdots && \cdots && \bullet & \cdots
	\arrow["{\olin{i-1}}", shorten >=-8pt,no head, from=3-1, to=3-2]
	\arrow["{p\quad q}", shorten <=-8pt, shorten >=-8pt,no head, from=3-2, to=3-3]
	\arrow["{j}", shorten <=-8pt, no head, from=3-3, to=3-4]
	\arrow["{\olin n}"{pos=0.3}, no head, from=3-3, to=1-3]
	\arrow["{m}"{pos=0.3}, shorten <=-8pt, no head, from=3-2, to=1-2]
	\arrow["{j}", no head, shorten <=-8pt, from=3-8, to=3-9]
	\arrow["{\overset{p}{ q}}", shorten <=-8pt, shorten >=-8pt,no head, from=3-8, to=2-8]
	\arrow["{\olin{n+\D}}"', shorten <=-8pt, no head, from=2-8, to=1-8]
	\arrow["{m}"{pos=0.3}, shorten <=-8pt, bend right =-45, no head, from=2-8, to=1-7]
	\arrow["{\olin{i-\D-1}}"',{pos=0.5}, shorten <=-8pt, no head, from=3-8, to=3-6]
\end{tikzcd}\]

This is a case of Generalized Move 2, proving the first move. The second move is obtained by applying the Rotation Lemma to the top vertex of the result of the first move.
\end{proof}

Each of the moves above is half of a handle slide. By combining them we get a rule for performing an entire handle slide.\\

\begin{lemma}[Generalized Move 3]\label{algebraichandleslide}
The following moves preserve homotopy limits, where $\D\coloneqq p-q$ and $\D'\coloneqq n-k$.

%%%%%%%%%%%%algebraicslide%%%%%%%%%%%%%%
\begin{figure}[H]
\centering
\includegraphics[scale=.8]{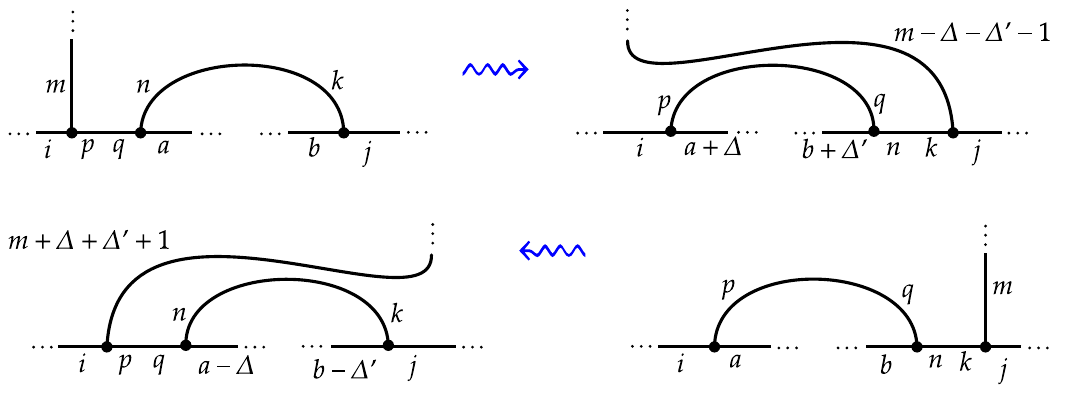}
\label{algebraicslide}
\end{figure}
%%%%%%%%%%%%algebraicslide%%%%%%%%%%%%%%

\end{lemma}

\begin{proof}
To realize the first handle slide, we apply the first move of \cref{move3}, then perform the shifts shown below, and finally apply the second move of \cref{move3}.
%%%%%%%%%proof1small%%%%%%%%%%%%%%%%%
\begin{figure}[H]
\centering
\includegraphics[scale=.8]{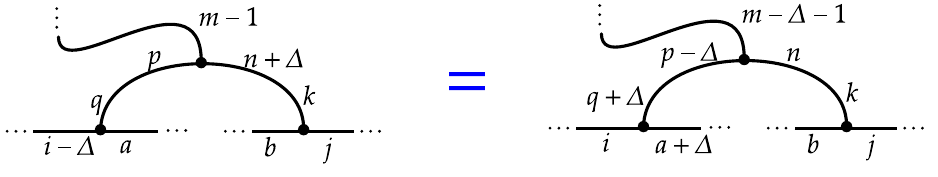}
\label{proof1small}
\end{figure}
%%%%%%%%%proof1small%%%%%%%%%%%%%%%%%
The second handle slide is formally the inverse of the first.
\end{proof}

For our purposes, we only need to shift the Cone functors.\\

\begin{lemma}[Move 3]\label{algebraicmove3}
The following moves preserve homotopy limits. 

%%%%%%%%%%%%simplehandle%%%%%%%%%%
\begin{figure}[H]
\centering
\includegraphics[scale=.8]{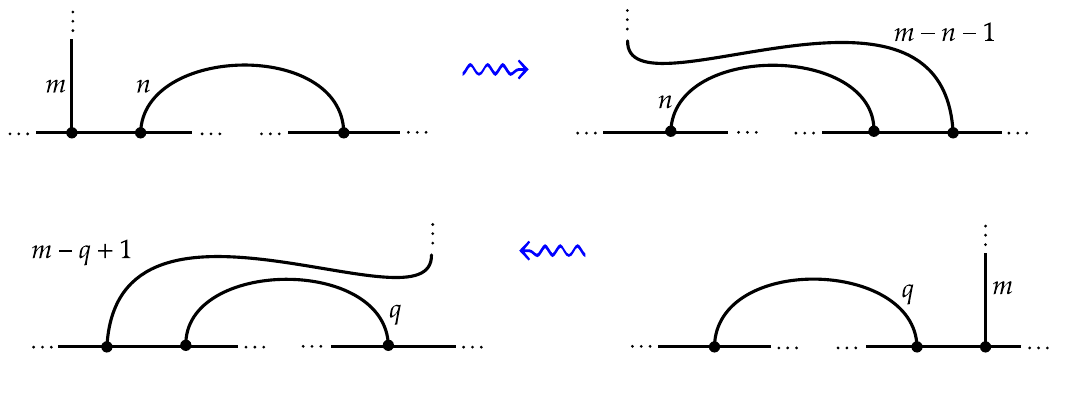}
\label{simplehandle}
\end{figure}
%%%%%%%%%%%simplehandle%%%%%%%%%%%

\end{lemma}

\begin{proof}
As before, we only need to show the first move. Applying \cref{algebraichandleslide} to the special case where only $m$ and $n$ are allowed to be nonzero, we get the following move:

%%%%%%%%%%%%%proof2%%%%%%%%%%%%%%%%%%
\begin{figure}[H]
\centering
\includegraphics[scale=.85]{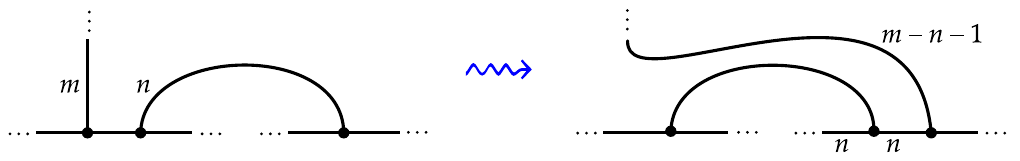}
\label{proof2}
\end{figure}
%%%%%%%%%%%%%%proof2%%%%%%%%%%%%%%%%%

Now shift the middle vertex by $-n$, and add $n$ to both halves of the associated 1-handle.
\end{proof}

\subsection{Proof of Generalized Move 2}

\begin{proof}[Proof of \cref{gen2}]
By applying the Shift Lemma, we can reduce to the case $p=q=0$, which we now assume. Call the left diagram $S$ and the right diagram $S'$. These represent two ends of a Weinstein isotopy. The skeleton in the middle of the isotopy is a 4-valent vertex. This is not arboreal, but there are two natural diagrams associated to it, resulting from $S$ and $S'$ respectively. Here they are, along with the proposed equivalence:

% https://q.uiver.app/#q=WzAsMTUsWzIsMiwiXFxtYXRoY2FsIEJeXFx0ZXh0e291dH0iXSxbMywyLCJcXG1hdGhjYWwgQV8xIl0sWzIsMSwiXFxtYXRoY2FsIEFfMSJdLFsxLDIsIlxcbWF0aGNhbCBBXzEiXSxbMiwzLCJcXG1hdGhjYWwgQV8xIl0sWzUsMiwiXFxtYXRoY2FsIEFfMSJdLFs2LDIsIlxcbWF0aGNhbCBCXlxcdGV4dHtpbn0iXSxbNiwxLCJcXG1hdGhjYWwgQV8xIl0sWzcsMiwiXFxtYXRoY2FsIEFfMSJdLFs2LDMsIlxcbWF0aGNhbCBBXzEiXSxbNCwyLCJcXHNpbWVxIl0sWzAsMCwiUyJdLFs4LDAsIlMnIl0sWzEsMV0sWzcsMV0sWzAsMSwiXFxwaV9SIl0sWzAsMiwiQ19SIiwyXSxbMCwzLCJcXHBpX0wiLDJdLFswLDQsIkNfTCJdLFs2LDUsIlxccGlfTCIsMl0sWzYsNywiQ19MIl0sWzYsOCwiXFxwaV9SIl0sWzYsOSwiQ19SIiwyXSxbMTEsMTMsIiIsMCx7InN0eWxlIjp7ImJvZHkiOnsibmFtZSI6ImRhc2hlZCJ9fX1dLFsxMiwxNCwiIiwwLHsic3R5bGUiOnsiYm9keSI6eyJuYW1lIjoiZGFzaGVkIn19fV1d

\be\label{outin}\begin{tikzcd}
	& {} & {\mc A} &&&& {\mc A} & {} \\
	& {\mc A} & {\mathcal B^\text{out}} & {\mc A} & \simeq & {\mc A} & {\mathcal B^\text{in}} & {\mc A} \\
	&& {\mc A} &&&& {\mc A}
	\arrow["{[j]\circ\pi_R}", from=2-3, to=2-4]
	\arrow["{[m]\circ C_R}"', from=2-3, to=1-3]
	\arrow["{[i]\circ \pi_L}"', from=2-3, to=2-2]
	\arrow["{[n]\circ C_L}", from=2-3, to=3-3]
	\arrow["{[i]\circ\pi_L}"', from=2-7, to=2-6]
	\arrow["{[m]\circ C_L}", from=2-7, to=1-7]
	\arrow["{[j]\circ\pi_R}", from=2-7, to=2-8]
	\arrow["{[n]\circ C_R}"', from=2-7, to=3-7]
\end{tikzcd}\ee

Here $\mc B^{\tx{out}}\coloneqq \mathrm{Fun}(\bcd \leftarrow \bcd \to \bcd,\mc A)=\mc A^{\to}\times_{\rho_{1}}\mc A^{\to}$ and $\mc B^{\tx{in}}\coloneqq\mathrm{Fun}(\bcd \to \bcd \leftarrow \bcd, \mc A)=\mc A^{\to}\times_{\rho_{2}}\mc A^{\to}$. Since $\rho_{1}$ and $\rho_{2}$ are fibrations, these pullbacks are also homotopy pullbacks.  The functors shown in \cref{outin} indicate projections and cones along the left and right arrows in each category. By replacing the two distinguished vertices in $S$ by one of the 4-valent vertices in \cref{outin}, we obtain two new graphic (though not arboreal) diagrams. Call these new diagrams $S^{\tx{out}}$ and $S^{\tx{in}}$, respectively. The full edges in $S$ and $S'$ have homotopy limits $\mc B^{\tx{out}}$ and $\mc B^{\tx{in}}$, respectively, so \cref{steps} yields $\holim S\simeq\holim S^{\tx{out}}$ and $\holim S'\simeq\holim S^{\tx{in}}$. It remains to prove the weak equivalence in the middle of \cref{outin}.\footnote{The cospans defining $\mc B^{\tx{out}}$ and $\mc B^{\tx{in}}$ are themselves weakly equivalent via the quasi-equivalence $T:\mc A^{\to}\lr \mc A^{\to}$. This already implies that $\mc B^{\tx{out}}\simeq\mc B^{\tx{in}}$, but we construct the quasi-equivalence explicitly to show that relevant squares commute.}

 In \cite[IV.4]{GM}, Gelfand and Manin use a triangulated version of $\mc B^{\tx{out/in}}$, namely as the derived categories of quiver representations in $k$-modules\footnote{Equivalently, the derived categories of modules over path algebras.}. The authors provide an equivalence between these derived categories. Our categories specializing to dg-enhancements of theirs, we will use the same formulas from their construction.

We define dg-functors $\varphi^{\text{out/in}}:\mc B^{\tx{out/in}}\lr \mc A^{\to}$. On objects
$V=(V_{L}\overset{\ell}\leftarrow V_{1}\overset{r}\to V_{R})$ and $W=(W_{L}\overset{\ell}\to W_{1}\overset{r}\leftarrow W_{R})$, $\vp^{\tx{out}}$ returns  $(V_{1}\overset{(\ell,r)^{\sf T}}\lr V_{L}\oplus V_{R})$ and $\vp^{\tx{in}}$ returns $(W_{L}\oplus W_{R}\overset{(\ell, r)}\lr W_{1})$, respectively. On a morphism $f=(f_{L},h_{L},f_{1},h_{R},f_{R})$, we define $\vp^{\tx{out}}(f)\coloneqq (f_{1}, (h_{L},h_{R})^{\sf T}, \mathrm{diag}(f_{L},f_{R}))$ and 
$\vp^{\tx{in}}(f)\coloneqq(\mathrm{diag}(f_{L},f_{R}), (h_{L},h_{R}), f_{1})$.

From these we construct a pair of quasi-inverses\footnote{Meaning that the composition each way is naturally homotopy equivalent to the identity.} $F^{+}:\mc B^{\tx{out}}\simeq \mc B^{\tx{in}}: F^{-}$. On objects, $F^{+}$ returns $(V_{L}\to{C(\varphi^{\text{out}}(V))}\leftarrow V_{R})$\footnote{\label{rest}The arrows are the restrictions of $T((\ell,r)^{\sf T}):V_{L}\oplus V_{R}\lr C(\vp^{\tx{out}}(V))$.} and $F^{-}$ returns $(W_{L}\leftarrow{C(\varphi^{\text{in}}(W))[-1]}\to W_{R})$\footnote{The arrows are the projections of $T'((\ell,r)):C(\varphi^{\text{in}}(W))[-1]\lr W_{L}\oplus W_{R}$.}. On morphisms, we define $F^{+}(f)\coloneqq(f_{L},0,C(\vp^{\tx{out}}(f)),0,f_{R})$ and $F^{-}(f)\coloneqq(f_{L},0,C(\vp^{\tx{in}}(f))[-1],0,f_{R})$. The composition $F^{-}F^{+}(V)$ is formed by taking the canonical projections of $T'T((\ell,r)^{\sf T})\simeq (\ell,r)^{\sf T}=\vp^{\tx{out}}(V)$. This recovers $V$ up to homotopy equivalence. Similarly, $F^{+}F^{-}(W)$ is formed by taking the canonical restrictions of $TT'((\ell,r))\simeq \vp^{\tx{in}}(W)$, which recovers $W$ up to homotopy equivalence.

In particular, $F^{\pm}$ are quasi-equivalences. Define $S^{\tx{out}}\lr S^{\tx{in}}$ by using $F^{+}$ at $\mc B^{\tx{out}}$ and identities everywhere else. We show that these commute up to natural homotopy equivalence. The only nontrivial squares are $C_{R}\circ F^{+}\simeq C_{L}$ and $C_{L}\circ F^{+}\simeq C_{R}$. Write $i_{L}$ and $i_{R}$ for the canonical inclusions to $V_{L}\oplus V_{R}$ and $j$ for the map in \cref{rest}. Then $C_{R}\circ F^{+}(V)=C(ji_{R})\simeq C(C(j)[-1]\to C(i_{R}))\simeq C(\ell)=C_{L}(V),$ and similarly when $C_{L}$ and $C_{R}$ are swapped.
\end{proof}

The argument above generalizes. Given a quiver $Q$ and a vertex $v$ in $Q$, define $\s_{v}Q$ to be the quiver obtained from $Q$ by reversing all arrows incident to $v$. \\

\begin{proposition}\label{repequivalence}
Fix an acyclic simple quiver $Q$ and a source or sink $v$ in $Q$. There is a quasi-equivalence $\mathrm{Fun}(Q,\mc A)\lr  \mathrm{Fun}(\s_{v}Q, \mc A)$.
\end{proposition}

\begin{proof}
If $v$ is a source (resp. sink) then $F^{+}$ (resp. $F^{-}$) generalizes via the formulas in \cite[IV.4]{GM} to the desired dg-functor $F_{v}^{+}$ (resp. $F_{v}^{-}$) with quasi-inverse $F_{v}^{-}$ (resp. $F_{v}^{+})$.
\end{proof}
Thus all equivalences in the aforementioned result in \cite[IV.4]{GM} can be upgraded to quasi-equivalences on their dg-enhancements. In particular, we have $\mc B^{\tx{out}}\,\,\simeq\,\, \mc B^{\tx{in}}\,\,\simeq\,\,\mathrm{Fun}(\bcd\to\bcd\to\bcd, \mc A).$

%%%%%%%%% CONSTRUCTION OF THE INVARIANT %%%%%%
\subsection{Proof of invariance}
Having developed all necessary tools, we clarify our invariant $\mc L(W)$ by making more explicit the prescription in \cref{markedintro}. 

Since $\mc L(W)$ is $\zb/2$-graded, we have $[n]=[-n]$, so there is no need to indicate on a 1-bone whether the shift is applied at one joint or another. In other words, every 1-bone either involves a shift or it doesn't, and in the former case we will be agnostic about which joint is given the shift. We will indicate this binary data by leaving a 1-bone unmarked if it has no shift, and by marking it with a tick if it has a shift.

As a result, to define a diagram in $\dgcat_{k}^{(2)}$ from an arboreal graph, it suffices to mark some of the 1-bones. Our definition is straightforward. First notice that if we orient a 0-bone, then there is a well-defined notion a joint being on its ``left'' or ``right''.\\

\begin{definition}\label{arborealdiagramchapter1}\label{invariantdef}
Given an arboreal graph $G$\footnote{Or any labeled trivalent graph, see \cref{caveat}.} and a strongly pretriangulated dg-category $\mc A\in\ob(\dgcat_{k}^{(2)})$, we define a graphic diagram $D(G)= D_{\mc A}(G)$ in $\big(\dgcat_{k}^{(2)}\big)^{F(Q(G))}$. Univalent vertices of $G$ are assigned the dg-category 0, trivalent vertices are assigned $\mc A^{\to}$, and edges are assigned $\mc A$. Given a trivalent vertex of $G$, the half edges $(A,B,C)$ are assigned the functors $(\rh_{1},\rh_{2}, [n]\circ C)$ respectively, with $n$ according to \cref{markedintro}. That is, fix an arbitrary orientation on each 0-bone of $G$. Given a 1-bone, mark it if and only if both of its joints are on the same side of their 0-bone(s). Of those marked 1-bones, set one of its halves to have $n=1$, and set $n=0$ otherwise.
\end{definition}

\cref{marked} again shows all six ways 1-bones can attach to oriented 0-bones in an arboreal graph arising from an arboreal surface: both joints on the left, both on the right, and mixed, along with accounting for 1-bones whose joints live on different 0-bones. Markings are given by \cref{invariantdef}, with marked 1-bones in red and unmarked in green. Define $\mc L(G)=\mc L_{\mc A}(G)$ to be the homotopy limit of $D(G)$, defined up to quasi-equivalence of dg-categories.

%%%%%%%%%%%%%%%marked%%%%%%%%%%%%%%%%
\begin{figure}
\centering
\includegraphics[scale=1]{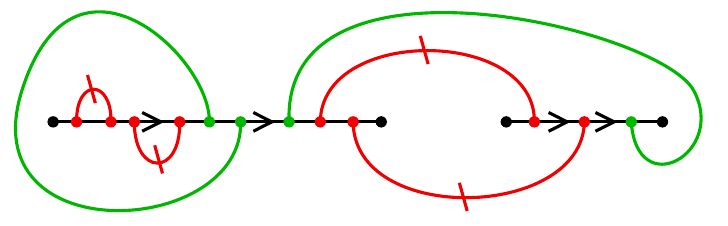}
\caption{Pictured are all six ways that 1-bones can be attached to oriented contractible 0-bones. Each 1-bone represents a cospan \break$\mc A^{\to}\overset{C}\lr \mc A\overset{[n]\circ C}\longleftarrow\mc A^{\to}$, where $n=0$ for an unmarked bone and $n=1$ for a marked bone.}
\label{marked}
\end{figure}
%%%%%%%%%%%%%%%marked%%%%%%%%%%%%%%%%

\begin{lemma}\label{invariancefinal}
The homotopy limit $\mc L(G)\coloneqq \holim D(G)$ is unchanged under the four arboreal moves given in \cref{graphdef}. 
\end{lemma}

\begin{proof}
We will prove the case where all 0-bones are contractible, as geometrically, any non-generic arboreal surface can be made arboreal through an arbitrarily small Weinstein homotopy. Suppose two arboreal graphs $G$ and $G'$ differ by Move 0 acting on a 1-bone $X$. If $X$ is marked (unmarked) in $D(G)$, then $X$ will be unmarked (marked) in $D(G')$, so $\mc L(G)\simeq\mc L(G')$ by \cref{algebraicmove1.2}. See \cref{move0mod2} for an illustration of how Move 0 works in our $\zb/2$-graded setting. If  $G$ and $G'$ differ by Move 1, then \cref{algebraicmove1} shows that $\mc L(G)\simeq\mc L(G')$. Similarly, Move 2 is accounted for by \cref{gen2}.
%%%%%%%%%%%%%%%%move0mod2%%%%%%%%%%%%%%%%%
\begin{figure}
\centering
\includegraphics[scale=1]{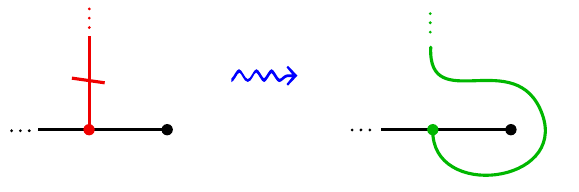}
\caption{The $\zb/2$-graded effect of boundary isotopy.}
\label{move0mod2}
\end{figure}
%%%%%%%%%%%%%%%%move0mod2%%%%%%%%%%%%%%%%%

Finally, suppose $G$ and $G'$ differ by Move 3, with the 1-bone $X$ sliding over the 1-bone $Y$. If $Y$ is \textit{marked}, then the moving joint of $X$, if starting on top (bottom), will remain on top (bottom), so $X$ will retain its (un)marked status. If $Y$ is \textit{unmarked}, then this joint will flip sides, so $X$ will change its (un)marked status. These are exactly the situations guaranteed by \cref{algebraicmove3} to preserve homotopy limits. See \cref{move3mod2} for these two types of handle slides.
%%%%%%%%%%%%%%%%move3mod2%%%%%%%%%%%%%%%%%
\begin{figure}
\centering
\includegraphics[scale=.91]{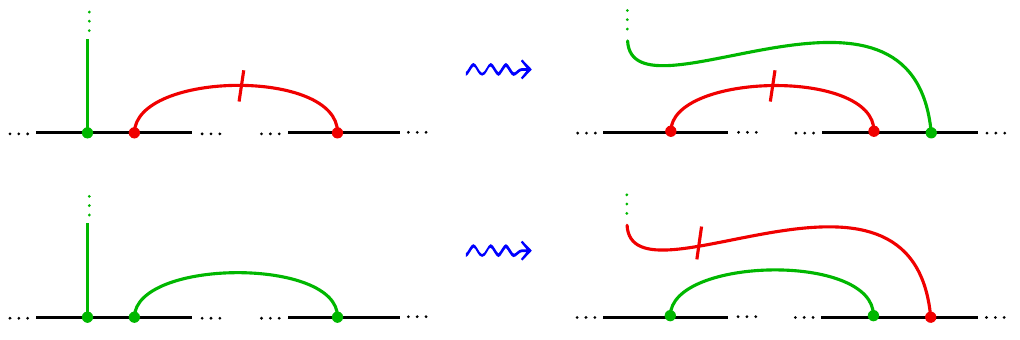}
\caption{The $\zb/2$-graded effect of handle sliding. Here it is assumed that both 0-bones in a given graph are oriented in the same direction}
\label{move3mod2}
\end{figure}
%%%%%%%%%%%%%%move3mod2%%%%%%%%%%%%%%%%%%
\end{proof}

\begin{lemma}\label{auxiliary}
The homotopy limit of $D(G)$ does not depend on the choice of orientations on the 0-bones.
\end{lemma}

\begin{proof}
Again we treat the case where all 0-bones are contractible. Suppose $G$ and $G'$ are the same arboreal graph, differing only in the orientation of one 0-bone $B$. Then $D(G)$ and $D(G')$ differ in a precise way: the shift in a 1-bone $X$ differs in $D(G)$ and $D(G')$ if and only if $X$ has exactly one joint on $B$, shown in \cref{step1}. Let $\mc S$ be the set of all such 1-bones.

%%%%%%%%%%%%%%%%%%%%%step1%%%%%%%%%%%%%%%%
\begin{figure}
\centering
\includegraphics[scale=.8]{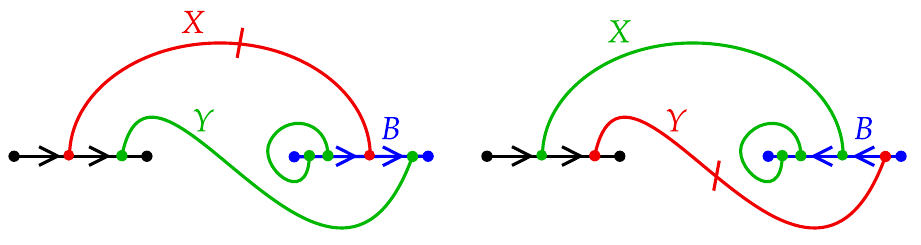}
\caption{An example of the proof. The diagrams $D(G)$ and $D(G')$ are shown on the left and right, respectively. The differing bone $B$ is the 0-bone on the right. In each case, the set $\mc S$ consists of the two 1-bones with joints on both 0-bones, and we arbitrarily choose the one on top to call $X$.}
\label{step1}
\end{figure}
%%%%%%%%%%%%%%%%%%%%%step1%%%%%%%%%%%%%%%%

We now argue that $\holim D(G)\simeq \holim D(G')$.\footnote{In the following, we abuse notation by using the same symbols $D(G)$ and $D(G')$, even as these objects change throughout our moves. This is justified by the fact that each move preserves homotopy limits.} Fix a 1-bone $X\in \mc S$ and let $\mc J$ be the set of joints on the 0-bone connected to $B$ by $X$. Order $\mc J=\{j_{0}<\dots<j_{n}\}$ by starting from the joint on $X$ and moving counterclockwise. Now, starting with $j_{1}$ and proceeding in order, slide all joints in $\mc J$ over $X$. Do this process identically in $G$ and $G'$. Thus each $j\in \mc J\sm\{j_{0}\}$ has undergone exactly one slide, and is now attached to $B$.

For each $Y\neq X$ in $\mc S$, exactly one of its two joints has been slid. Therefore, the marking of $Y$ will flip if and only if $X$ is not marked. Since $X$ has opposite markings in $G$ and $G'$, as $Y$ slides over $X$ it will flip its marking in \textit{exactly one} of $G$ or $G'$. Since the marking of $Y$ differed in $G$ versus $G'$ before the slide over $X$, after the slide it is now identical in $G$ and $G'$. See \cref{step2} for an illustration in our working example. 

%%%%%%%%%%%%%%%%%%%%%%%%step2%%%%%%%%%%%%%
\begin{figure}
\centering
\includegraphics[scale=.8]{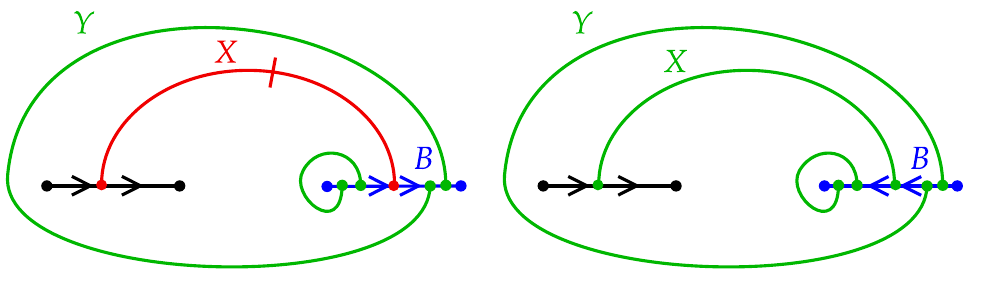}
\caption{Modified $D(G)$ and $D(G')$. In both diagrams, we've performed two instances of Move 0 and one instance of Move 1. The two Move 0's cancel in each case. In $D(G)$, $Y$ does not change its marking under Move 3, because $X$ is unmarked. In $D(G')$, $X$ is marked, so $Y$ does change its marking. Since before these moves $Y$ differed in marking between the two diagrams, it now has the same marking in both.}
\label{step2}
\end{figure}
%%%%%%%%%%%%%%%%%%%%%%step2%%%%%%%%%%%%%%%%%%%%%%

For any 1-bone $Y$ with \textit{both} joints in $\mc J$, both of these joints have been slid over $X$. Thus $Y$ flips markings either 0 times (if $X$ is marked) or 2 times (if $X$ is unmarked). In either case, $Y$ does not change its marking. Since $Y$ started out with the same marking in $G$ and $G'$, it also ends with the same marking in both. Now $D(G)$ and $D(G')$ differ only in the marking of $X$, which is part of a cancelling pair with a 0-bone. By \cref{algebraicmove1}, the cancellation of this pair results in the same picture regardless of whether $X$ is marked or not. Performing this cancellation, $D(G)=D(G')$.\footnote{Recall that the orientations are auxiliary and do not affect the algebraic diagrams.}
\end{proof}

\begin{theorem}\label{final}
Fix a commutative unital ring $k$ and a strongly pretriangulated dg-category $\mc A\in\ob(\dgcat_{k}^{(2)})$. Given a Weinstein surface $W$ and any arborealization $W'$ of $W$, the dg-category $\mc L(W)\coloneqq \holim D_{\mc A}(G(W'))\in \ob(\dgcat_{k}^{(2)})$ is well-defined and invariant up to quasi-equivalence under Weinstein homotopies of $W$.
\end{theorem}

\begin{proof}
Combine \cref{invariancefinal} and \cref{auxiliary} with \cref{invcor}.
\end{proof}

%%%%%%%%% SOME EXAMPLES %%%%%%%%%%%
\subsection{Simplification and examples}\label{examples}
We will now provide simplified calculations of $\mc L(W)$ for all topological types of $W$. To do so, we will represent every arboreal graph with an arboreal equivalent one which has no boundary vertices and a single non-contractible 0-bone. Equivalently, we find a canonical way to Weinstein homotope a given arboreal surface into a non-generic arboreal surface whose skeleton's index 0 points form a single $S^{1}$. 

Every arboreal surface $W$ can be Weinstein homotoped so that $G=G(W)$ is obtained by connecting \textit{standard blocks} with form that of \cref{types}, resulting in a presentation which we will call \textit{standard}. We will call a standard block Type 1 if it has only one 1-bone, and Type 2 otherwise. Suppose $G$ is standard and includes, in any order, a total of $n$ Type 1 blocks and $d$ Type 2 blocks, where the $i^{\text{th}}$ Type 2 block has $k_{i}$ 1-bones. An Euler characteristic count shows that $W$ has genus $g=(h-b+1)/2=(\sum_{i}k_{i}+\#\{i| k_{i}\text{ odd}\})/2$, where $h$ and $b$ are the numbers of 1-bones and punctures, respectively. In particular, if $k_{i}=2$ for all $i$, then $g=d$ and $b=n+1$, so we obtain all topological types of Weinstein surfaces. 

%%%%%%%%%%%types%%%%%%%%%%%%%%
\begin{figure}
\centering
\includegraphics[scale=1.2]{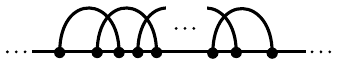}
\caption{Connecting blocks gives a standard form which can represent any arboreal graph. Type 1 blocks have one 1-bone, and each adds a puncture. The rest are Type 2, with $k\geq2$ index 1 bones: each increases genus by $\lceil k/2 \rceil$ and additionally adds a puncture when $k$ is odd.}
\label{types}
\end{figure}
%%%%%%%%%%%types%%%%%%%%%%%%%%

W will show how to turn any standard block into a circular 0-bone with some 1-bones attached. First we prove a local result that will allow us to merge circles.\\

\begin{lemma}\label{crosssimple}
The following move preserves homotopy limits mod 2.

%%%%%%%%%%%%%crosssimple%%%%%%%%%%%%%%%%%%
\begin{figure}[H]
\centering
\includegraphics[scale=1.1]{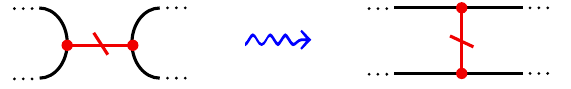}
\end{figure}
%%%%%%%%%%%%%%crosssimple%%%%%%%%%%%%%%%%%

\end{lemma}
\begin{proof}
Rotate the two vertices appropriately so as to apply the second move in \cref{move3}. After this application, rotate the top vertex to obtain the shape desired, and then shift appropriately.
\end{proof}

\begin{lemma}\label{circleslem}
The following move preserves homotopy limits mod 2.

%%%%%%%%%%%circles%%%%%%%%%%%%%%
\begin{figure}[H]
\centering

\tikzset{every picture/.style={line width=0.75pt}} %set default line width to 0.75pt        

\begin{tikzpicture}[x=0.75pt,y=0.75pt,yscale=-1,xscale=1]
%uncomment if require: \path (0,228); %set diagram left start at 0, and has height of 228

%Straight Lines [id:da04053261825895227] 
\draw [line width=1.5]    (20.38,132) -- (195.62,132) ;
%Curve Lines [id:da5885002830197252] 
\draw [color={rgb, 255:red, 240; green, 0; blue, 0 }  ,draw opacity=1 ][line width=1.5]    (38.2,132.6) .. controls (38.2,93.6) and (76.2,95.6) .. (76.2,132.6) ;
\draw [shift={(76.2,132.6)}, rotate = 90] [color={rgb, 255:red, 240; green, 0; blue, 0 }  ,draw opacity=1 ][fill={rgb, 255:red, 240; green, 0; blue, 0 }  ,fill opacity=1 ][line width=1.5]      (0, 0) circle [x radius= 2.61, y radius= 2.61]   ;
\draw [shift={(38.2,132.6)}, rotate = 270] [color={rgb, 255:red, 240; green, 0; blue, 0 }  ,draw opacity=1 ][fill={rgb, 255:red, 240; green, 0; blue, 0 }  ,fill opacity=1 ][line width=1.5]      (0, 0) circle [x radius= 2.61, y radius= 2.61]   ;
%Curve Lines [id:da0452747829082456] 
\draw [color={rgb, 255:red, 240; green, 0; blue, 0 }  ,draw opacity=1 ][line width=1.5]    (62.2,132.6) .. controls (62.2,93.6) and (100.2,95.6) .. (100.2,132.6) ;
\draw [shift={(100.2,132.6)}, rotate = 90] [color={rgb, 255:red, 240; green, 0; blue, 0 }  ,draw opacity=1 ][fill={rgb, 255:red, 240; green, 0; blue, 0 }  ,fill opacity=1 ][line width=1.5]      (0, 0) circle [x radius= 2.61, y radius= 2.61]   ;
\draw [shift={(62.2,132.6)}, rotate = 270] [color={rgb, 255:red, 240; green, 0; blue, 0 }  ,draw opacity=1 ][fill={rgb, 255:red, 240; green, 0; blue, 0 }  ,fill opacity=1 ][line width=1.5]      (0, 0) circle [x radius= 2.61, y radius= 2.61]   ;
%Curve Lines [id:da6446033054905815] 
\draw [color={rgb, 255:red, 240; green, 0; blue, 0 }  ,draw opacity=1 ][line width=1.5]    (88.2,132.6) .. controls (88.2,109.57) and (101.45,100.84) .. (112.3,105.16) .. controls (119.83,108.16) and (126.2,117.45) .. (126.2,132.6) ;
\draw [shift={(88.2,132.6)}, rotate = 270] [color={rgb, 255:red, 240; green, 0; blue, 0 }  ,draw opacity=1 ][fill={rgb, 255:red, 240; green, 0; blue, 0 }  ,fill opacity=1 ][line width=1.5]      (0, 0) circle [x radius= 2.61, y radius= 2.61]   ;
%Curve Lines [id:da260016341276625] 
\draw [color={rgb, 255:red, 240; green, 0; blue, 0 }  ,draw opacity=1 ][line width=1.5]    (113.2,132.6) .. controls (113.2,93.6) and (151.2,95.6) .. (151.2,132.6) ;
\draw [shift={(151.2,132.6)}, rotate = 90] [color={rgb, 255:red, 240; green, 0; blue, 0 }  ,draw opacity=1 ][fill={rgb, 255:red, 240; green, 0; blue, 0 }  ,fill opacity=1 ][line width=1.5]      (0, 0) circle [x radius= 2.61, y radius= 2.61]   ;
%Shape: Rectangle [id:dp10799159750136378] 
\draw  [draw opacity=0][fill={rgb, 255:red, 255; green, 255; blue, 255 }  ,fill opacity=1 ] (106,91) -- (132.4,91) -- (132.4,131) -- (106,131) -- cycle ;
%Curve Lines [id:da3209690800000953] 
\draw [color={rgb, 255:red, 240; green, 0; blue, 0 }  ,draw opacity=1 ][line width=1.5]    (136.21,133) .. controls (136.21,94) and (174.21,96) .. (174.21,133) ;
\draw [shift={(174.21,133)}, rotate = 90] [color={rgb, 255:red, 240; green, 0; blue, 0 }  ,draw opacity=1 ][fill={rgb, 255:red, 240; green, 0; blue, 0 }  ,fill opacity=1 ][line width=1.5]      (0, 0) circle [x radius= 2.61, y radius= 2.61]   ;
\draw [shift={(136.21,133)}, rotate = 270] [color={rgb, 255:red, 240; green, 0; blue, 0 }  ,draw opacity=1 ][fill={rgb, 255:red, 240; green, 0; blue, 0 }  ,fill opacity=1 ][line width=1.5]      (0, 0) circle [x radius= 2.61, y radius= 2.61]   ;
%Straight Lines [id:da9416423248787461] 
\draw [line width=1.5]    (111,132) -- (128.78,132) ;
%Straight Lines [id:da7985608341686848] 
\draw [color={rgb, 255:red, 240; green, 0; blue, 0 }  ,draw opacity=1 ][line width=1.5]    (52,99) -- (56,110.43) ;
%Straight Lines [id:da48107184600675157] 
\draw [color={rgb, 255:red, 240; green, 0; blue, 0 }  ,draw opacity=1 ][line width=1.5]    (77,98) -- (81,109.43) ;
%Straight Lines [id:da33932960390461553] 
\draw [color={rgb, 255:red, 240; green, 0; blue, 0 }  ,draw opacity=1 ][line width=1.5]    (152,99) -- (156,110.43) ;
%Shape: Arc [id:dp13868185305822311] 
\draw  [draw opacity=0][line width=1.5]  (351.88,75.9) .. controls (351.88,75.9) and (351.88,75.9) .. (351.88,75.9) .. controls (341.76,75.95) and (333.51,66.42) .. (333.45,54.61) .. controls (333.39,42.8) and (341.54,33.18) .. (351.66,33.13) .. controls (361.77,33.07) and (370.03,42.61) .. (370.09,54.42) .. controls (370.15,66.23) and (362,75.84) .. (351.88,75.9) -- (351.77,54.51) -- cycle ; \draw  [line width=1.5]  (351.88,75.9) .. controls (351.88,75.9) and (351.88,75.9) .. (351.88,75.9) .. controls (341.76,75.95) and (333.51,66.42) .. (333.45,54.61) .. controls (333.39,42.8) and (341.54,33.18) .. (351.66,33.13) .. controls (361.77,33.07) and (370.03,42.61) .. (370.09,54.42) .. controls (370.15,66.23) and (362,75.84) .. (351.88,75.9) -- cycle ;  
%Straight Lines [id:da5623697131397505] 
\draw [line width=1.5]    (370.38,55) -- (392.2,55) ;
\draw [shift={(370.38,55)}, rotate = 0] [color={rgb, 255:red, 0; green, 0; blue, 0 }  ][fill={rgb, 255:red, 0; green, 0; blue, 0 }  ][line width=1.5]      (0, 0) circle [x radius= 2.61, y radius= 2.61]   ;
%Straight Lines [id:da880176571170182] 
\draw [line width=1.5]    (304.2,56) -- (333.2,56) ;
\draw [shift={(333.2,56)}, rotate = 0] [color={rgb, 255:red, 0; green, 0; blue, 0 }  ][fill={rgb, 255:red, 0; green, 0; blue, 0 }  ][line width=1.5]      (0, 0) circle [x radius= 2.61, y radius= 2.61]   ;
%Straight Lines [id:da6999794454506117] 
\draw [color={rgb, 255:red, 0; green, 0; blue, 0 }  ,draw opacity=1 ][line width=1.5]    (316,50) -- (320,61.43) ;
%Straight Lines [id:da12065737574068636] 
\draw [line width=1.5]    (491.2,55) -- (520.2,55) ;
\draw [shift={(520.2,55)}, rotate = 0] [color={rgb, 255:red, 0; green, 0; blue, 0 }  ][fill={rgb, 255:red, 0; green, 0; blue, 0 }  ][line width=1.5]      (0, 0) circle [x radius= 2.61, y radius= 2.61]   ;
%Straight Lines [id:da5753964223138307] 
\draw [color={rgb, 255:red, 0; green, 0; blue, 0 }  ,draw opacity=1 ][line width=1.5]    (503,49) -- (507,60.43) ;
%Straight Lines [id:da36458919168580917] 
\draw [line width=1.5]    (587.38,55) -- (609.2,55) ;
\draw [shift={(587.38,55)}, rotate = 0] [color={rgb, 255:red, 0; green, 0; blue, 0 }  ][fill={rgb, 255:red, 0; green, 0; blue, 0 }  ][line width=1.5]      (0, 0) circle [x radius= 2.61, y radius= 2.61]   ;
%Shape: Arc [id:dp5440343618534564] 
\draw  [draw opacity=0][line width=1.5]  (355.88,192.9) .. controls (355.88,192.9) and (355.88,192.9) .. (355.88,192.9) .. controls (345.76,192.95) and (337.51,183.42) .. (337.45,171.61) .. controls (337.39,159.8) and (345.54,150.18) .. (355.66,150.13) .. controls (365.77,150.07) and (374.03,159.61) .. (374.09,171.42) .. controls (374.15,183.23) and (366,192.84) .. (355.88,192.9) -- (355.77,171.51) -- cycle ; \draw  [line width=1.5]  (355.88,192.9) .. controls (355.88,192.9) and (355.88,192.9) .. (355.88,192.9) .. controls (345.76,192.95) and (337.51,183.42) .. (337.45,171.61) .. controls (337.39,159.8) and (345.54,150.18) .. (355.66,150.13) .. controls (365.77,150.07) and (374.03,159.61) .. (374.09,171.42) .. controls (374.15,183.23) and (366,192.84) .. (355.88,192.9) -- cycle ;  
%Straight Lines [id:da13050169228664732] 
\draw [line width=1.5]    (374.38,172) -- (396.2,172) ;
\draw [shift={(374.38,172)}, rotate = 0] [color={rgb, 255:red, 0; green, 0; blue, 0 }  ][fill={rgb, 255:red, 0; green, 0; blue, 0 }  ][line width=1.5]      (0, 0) circle [x radius= 2.61, y radius= 2.61]   ;
%Straight Lines [id:da38728556267415726] 
\draw [line width=1.5]    (308.2,173) -- (337.2,173) ;
\draw [shift={(337.2,173)}, rotate = 0] [color={rgb, 255:red, 0; green, 0; blue, 0 }  ][fill={rgb, 255:red, 0; green, 0; blue, 0 }  ][line width=1.5]      (0, 0) circle [x radius= 2.61, y radius= 2.61]   ;
%Straight Lines [id:da7168489353484935] 
\draw [color={rgb, 255:red, 0; green, 0; blue, 0 }  ,draw opacity=1 ][line width=1.5]    (320,167) -- (324,178.43) ;
%Curve Lines [id:da3536971953556778] 
\draw [color={rgb, 255:red, 0; green, 183; blue, 0 }  ,draw opacity=1 ][line width=1.5] [line join = round][line cap = round]   (355.88,192.9) .. controls (355.65,177.91) and (330,145.67) .. (338,138.67) .. controls (346,131.67) and (356.2,136.8) .. (356.2,150.8) ;
\draw [shift={(356.2,150.8)}, rotate = 90] [color={rgb, 255:red, 0; green, 183; blue, 0 }  ,draw opacity=1 ][fill={rgb, 255:red, 0; green, 183; blue, 0 }  ,fill opacity=1 ][line width=1.5] [line join = round][line cap = round]     (0, 0) circle [x radius= 2.61, y radius= 2.61]   ;
\draw [shift={(355.88,192.9)}, rotate = 269.12] [color={rgb, 255:red, 0; green, 183; blue, 0 }  ,draw opacity=1 ][fill={rgb, 255:red, 0; green, 183; blue, 0 }  ,fill opacity=1 ][line width=1.5] [line join = round][line cap = round]     (0, 0) circle [x radius= 2.61, y radius= 2.61]   ;
%Rounded Rect [id:dp07813438651280613] 
\draw  [line width=1.5]  (520,37.83) .. controls (520,32.17) and (524.59,27.59) .. (530.24,27.59) -- (577.14,27.59) .. controls (582.8,27.59) and (587.38,32.17) .. (587.38,37.83) -- (587.38,68.56) .. controls (587.38,74.21) and (582.8,78.8) .. (577.14,78.8) -- (530.24,78.8) .. controls (524.59,78.8) and (520,74.21) .. (520,68.56) -- cycle ;
%Curve Lines [id:da047501398821335306] 
\draw [color={rgb, 255:red, 240; green, 0; blue, 0 }  ,draw opacity=1 ][line width=1.5] [line join = round][line cap = round]   (530.24,78.8) .. controls (530.09,67.96) and (533.58,55.03) .. (544,47.8) .. controls (554.42,40.57) and (572.38,37.83) .. (587.38,37.83) ;
\draw [shift={(587.38,37.83)}, rotate = 0] [color={rgb, 255:red, 240; green, 0; blue, 0 }  ,draw opacity=1 ][fill={rgb, 255:red, 240; green, 0; blue, 0 }  ,fill opacity=1 ][line width=1.5] [line join = round][line cap = round]     (0, 0) circle [x radius= 2.61, y radius= 2.61]   ;
\draw [shift={(530.24,78.8)}, rotate = 269.2] [color={rgb, 255:red, 240; green, 0; blue, 0 }  ,draw opacity=1 ][fill={rgb, 255:red, 240; green, 0; blue, 0 }  ,fill opacity=1 ][line width=1.5] [line join = round][line cap = round]     (0, 0) circle [x radius= 2.61, y radius= 2.61]   ;
%Straight Lines [id:da02629692841745379] 
\draw [color={rgb, 255:red, 240; green, 0; blue, 0 }  ,draw opacity=1 ][line width=1.5]    (557.43,28.31) -- (557.43,78.8) ;
\draw [shift={(557.43,78.8)}, rotate = 90] [color={rgb, 255:red, 240; green, 0; blue, 0 }  ,draw opacity=1 ][fill={rgb, 255:red, 240; green, 0; blue, 0 }  ,fill opacity=1 ][line width=1.5]      (0, 0) circle [x radius= 2.61, y radius= 2.61]   ;
\draw [shift={(557.43,28.31)}, rotate = 90] [color={rgb, 255:red, 240; green, 0; blue, 0 }  ,draw opacity=1 ][fill={rgb, 255:red, 240; green, 0; blue, 0 }  ,fill opacity=1 ][line width=1.5]      (0, 0) circle [x radius= 2.61, y radius= 2.61]   ;
%Straight Lines [id:da7200248882641153] 
\draw [color={rgb, 255:red, 240; green, 0; blue, 0 }  ,draw opacity=1 ][line width=1.5]    (552.69,53.19) -- (563.17,58.9) ;
%Straight Lines [id:da01120352190264029] 
\draw [color={rgb, 255:red, 240; green, 0; blue, 0 }  ,draw opacity=1 ][line width=1.5]    (533.69,51.19) -- (544.17,56.9) ;
%Straight Lines [id:da5369225686902543] 
\draw [line width=1.5]    (461.2,172) -- (490.2,172) ;
\draw [shift={(490.2,172)}, rotate = 0] [color={rgb, 255:red, 0; green, 0; blue, 0 }  ][fill={rgb, 255:red, 0; green, 0; blue, 0 }  ][line width=1.5]      (0, 0) circle [x radius= 2.61, y radius= 2.61]   ;
%Straight Lines [id:da39786763303938866] 
\draw [color={rgb, 255:red, 0; green, 0; blue, 0 }  ,draw opacity=1 ][line width=1.5]    (473,166) -- (477,177.43) ;
%Rounded Rect [id:dp8939925547227529] 
\draw  [line width=1.5]  (490,154.87) .. controls (490,149.19) and (494.6,144.59) .. (500.28,144.59) -- (617.72,144.59) .. controls (623.4,144.59) and (628,149.19) .. (628,154.87) -- (628,185.72) .. controls (628,191.4) and (623.4,196) .. (617.72,196) -- (500.28,196) .. controls (494.6,196) and (490,191.4) .. (490,185.72) -- cycle ;
%Straight Lines [id:da5658467620108382] 
\draw [color={rgb, 255:red, 240; green, 0; blue, 0 }  ,draw opacity=1 ][line width=1.5]    (595.43,145.31) -- (595.43,195.8) ;
\draw [shift={(595.43,195.8)}, rotate = 90] [color={rgb, 255:red, 240; green, 0; blue, 0 }  ,draw opacity=1 ][fill={rgb, 255:red, 240; green, 0; blue, 0 }  ,fill opacity=1 ][line width=1.5]      (0, 0) circle [x radius= 2.61, y radius= 2.61]   ;
\draw [shift={(595.43,145.31)}, rotate = 90] [color={rgb, 255:red, 240; green, 0; blue, 0 }  ,draw opacity=1 ][fill={rgb, 255:red, 240; green, 0; blue, 0 }  ,fill opacity=1 ][line width=1.5]      (0, 0) circle [x radius= 2.61, y radius= 2.61]   ;
%Straight Lines [id:da800086584554316] 
\draw [color={rgb, 255:red, 240; green, 0; blue, 0 }  ,draw opacity=1 ][line width=1.5]    (590.69,169.19) -- (601.17,174.9) ;
%Straight Lines [id:da21393093935075846] 
\draw [line width=1.5]    (537.1,144.17) -- (569.3,144.44) ;
%Curve Lines [id:da7473670817666189] 
\draw [color={rgb, 255:red, 240; green, 0; blue, 0 }  ,draw opacity=1 ][line width=1.5]    (505.73,144.7) .. controls (505.73,115.73) and (531.6,117.22) .. (531.6,144.7) ;
\draw [shift={(531.6,144.7)}, rotate = 90] [color={rgb, 255:red, 240; green, 0; blue, 0 }  ,draw opacity=1 ][fill={rgb, 255:red, 240; green, 0; blue, 0 }  ,fill opacity=1 ][line width=1.5]      (0, 0) circle [x radius= 2.61, y radius= 2.61]   ;
\draw [shift={(505.73,144.7)}, rotate = 270] [color={rgb, 255:red, 240; green, 0; blue, 0 }  ,draw opacity=1 ][fill={rgb, 255:red, 240; green, 0; blue, 0 }  ,fill opacity=1 ][line width=1.5]      (0, 0) circle [x radius= 2.61, y radius= 2.61]   ;
%Curve Lines [id:da6428169038033467] 
\draw [color={rgb, 255:red, 240; green, 0; blue, 0 }  ,draw opacity=1 ][line width=1.5]    (523.43,144.7) .. controls (523.43,127.6) and (532.45,121.11) .. (539.84,124.32) .. controls (544.96,126.55) and (549.3,133.45) .. (549.3,144.7) ;
\draw [shift={(523.43,144.7)}, rotate = 270] [color={rgb, 255:red, 240; green, 0; blue, 0 }  ,draw opacity=1 ][fill={rgb, 255:red, 240; green, 0; blue, 0 }  ,fill opacity=1 ][line width=1.5]      (0, 0) circle [x radius= 2.61, y radius= 2.61]   ;
%Curve Lines [id:da9122479987552631] 
\draw [color={rgb, 255:red, 240; green, 0; blue, 0 }  ,draw opacity=1 ][line width=1.5]    (540.45,144.7) .. controls (540.45,115.73) and (566.32,117.22) .. (566.32,144.7) ;
\draw [shift={(566.32,144.7)}, rotate = 90] [color={rgb, 255:red, 240; green, 0; blue, 0 }  ,draw opacity=1 ][fill={rgb, 255:red, 240; green, 0; blue, 0 }  ,fill opacity=1 ][line width=1.5]      (0, 0) circle [x radius= 2.61, y radius= 2.61]   ;
%Shape: Rectangle [id:dp25324800801087277] 
\draw  [draw opacity=0][fill={rgb, 255:red, 255; green, 255; blue, 255 }  ,fill opacity=1 ] (535.23,114.59) -- (553.2,114.59) -- (553.2,144.31) -- (535.23,144.31) -- cycle ;
%Curve Lines [id:da4834430925848495] 
\draw [color={rgb, 255:red, 240; green, 0; blue, 0 }  ,draw opacity=1 ][line width=1.5]    (556.12,145) .. controls (556.12,116.03) and (583.2,117.31) .. (583.2,144.8) ;
\draw [shift={(583.2,144.8)}, rotate = 90] [color={rgb, 255:red, 240; green, 0; blue, 0 }  ,draw opacity=1 ][fill={rgb, 255:red, 240; green, 0; blue, 0 }  ,fill opacity=1 ][line width=1.5]      (0, 0) circle [x radius= 2.61, y radius= 2.61]   ;
\draw [shift={(556.12,145)}, rotate = 270] [color={rgb, 255:red, 240; green, 0; blue, 0 }  ,draw opacity=1 ][fill={rgb, 255:red, 240; green, 0; blue, 0 }  ,fill opacity=1 ][line width=1.5]      (0, 0) circle [x radius= 2.61, y radius= 2.61]   ;
%Curve Lines [id:da20262901059412985] 
\draw [color={rgb, 255:red, 0; green, 183; blue, 0 }  ,draw opacity=1 ][line width=1.5] [line join = round][line cap = round]   (507.88,195.9) .. controls (507.65,180.91) and (482.2,132.8) .. (491.2,125.8) .. controls (500.2,118.8) and (515.2,130.8) .. (515.2,144.8) ;
\draw [shift={(515.2,144.8)}, rotate = 90] [color={rgb, 255:red, 0; green, 183; blue, 0 }  ,draw opacity=1 ][fill={rgb, 255:red, 0; green, 183; blue, 0 }  ,fill opacity=1 ][line width=1.5] [line join = round][line cap = round]     (0, 0) circle [x radius= 2.61, y radius= 2.61]   ;
\draw [shift={(507.88,195.9)}, rotate = 269.12] [color={rgb, 255:red, 0; green, 183; blue, 0 }  ,draw opacity=1 ][fill={rgb, 255:red, 0; green, 183; blue, 0 }  ,fill opacity=1 ][line width=1.5] [line join = round][line cap = round]     (0, 0) circle [x radius= 2.61, y radius= 2.61]   ;
%Curve Lines [id:da8307772103338934] 
\draw [color={rgb, 255:red, 0; green, 183; blue, 0 }  ,draw opacity=1 ][line width=1.5] [line join = round][line cap = round]   (612,196) .. controls (611.77,181.01) and (613,131) .. (603,125) .. controls (593,119) and (575.2,130.31) .. (575.2,144.31) ;
\draw [shift={(575.2,144.31)}, rotate = 90] [color={rgb, 255:red, 0; green, 183; blue, 0 }  ,draw opacity=1 ][fill={rgb, 255:red, 0; green, 183; blue, 0 }  ,fill opacity=1 ][line width=1.5] [line join = round][line cap = round]     (0, 0) circle [x radius= 2.61, y radius= 2.61]   ;
\draw [shift={(612,196)}, rotate = 269.12] [color={rgb, 255:red, 0; green, 183; blue, 0 }  ,draw opacity=1 ][fill={rgb, 255:red, 0; green, 183; blue, 0 }  ,fill opacity=1 ][line width=1.5] [line join = round][line cap = round]     (0, 0) circle [x radius= 2.61, y radius= 2.61]   ;
%Straight Lines [id:da33063650693951485] 
\draw [line width=1.5]    (539.02,144.7) -- (549.3,144.7) -- (552.3,144.7) ;
%Straight Lines [id:da742649003080738] 
\draw [line width=1.5]    (628.38,171) -- (650.2,171) ;
\draw [shift={(628.38,171)}, rotate = 0] [color={rgb, 255:red, 0; green, 0; blue, 0 }  ][fill={rgb, 255:red, 0; green, 0; blue, 0 }  ][line width=1.5]      (0, 0) circle [x radius= 2.61, y radius= 2.61]   ;
%Straight Lines [id:da42747752473372946] 
\draw [color={rgb, 255:red, 240; green, 0; blue, 0 }  ,draw opacity=1 ][line width=1.5]    (566.69,119.9) -- (571.17,130.6) ;
%Straight Lines [id:da3223147669093449] 
\draw [color={rgb, 255:red, 240; green, 0; blue, 0 }  ,draw opacity=1 ][line width=1.5]    (514.69,119.19) -- (519.17,129.9) ;

% Text Node
\draw (110,113.4) node [anchor=north west][inner sep=0.75pt]  [color={rgb, 255:red, 240; green, 0; blue, 0 }  ,opacity=1 ]  {$\cdots $};
% Text Node
\draw (196,128.4) node [anchor=north west][inner sep=0.75pt]    {$\cdots $};
% Text Node
\draw (-3,127.4) node [anchor=north west][inner sep=0.75pt]    {$\cdots $};
% Text Node
\draw (53,61.4) node [anchor=north west][inner sep=0.75pt]  [font=\Large]  {$\textcolor[rgb]{0.94,0,0}{\overbrace{\ \ \ \ \ \ \ \ \ \ \ \ \ \ \ \ \ \ \ \ \ \ }^{k \ }}$};
% Text Node
\draw (231,98.4) node [anchor=north west][inner sep=0.75pt]  [font=\LARGE,color={rgb, 255:red, 0; green, 0; blue, 255 }  ,opacity=1 ]  {$\rightsquigarrow $};
% Text Node
\draw (393,50.4) node [anchor=north west][inner sep=0.75pt]    {$\cdots $};
% Text Node
\draw (280,51.4) node [anchor=north west][inner sep=0.75pt]    {$\cdots $};
% Text Node
\draw (467,50.4) node [anchor=north west][inner sep=0.75pt]    {$\cdots $};
% Text Node
\draw (610,50.4) node [anchor=north west][inner sep=0.75pt]    {$\cdots $};
% Text Node
\draw (397,167.4) node [anchor=north west][inner sep=0.75pt]    {$\cdots $};
% Text Node
\draw (284,168.4) node [anchor=north west][inner sep=0.75pt]    {$\cdots $};
% Text Node
\draw (437,167.4) node [anchor=north west][inner sep=0.75pt]    {$\cdots $};
% Text Node
\draw (534.56,128.03) node [anchor=north west][inner sep=0.75pt]  [color={rgb, 255:red, 240; green, 0; blue, 0 }  ,opacity=1 ]  {$\cdots $};
% Text Node
\draw (651,166.4) node [anchor=north west][inner sep=0.75pt]    {$\cdots $};
% Text Node
\draw (123,175) node [anchor=north west][inner sep=0.75pt]   [align=left] {$ $};
% Text Node
\draw (332,9.4) node [anchor=north west][inner sep=0.75pt]    {$k=1$};
% Text Node
\draw (531,4.4) node [anchor=north west][inner sep=0.75pt]    {$k=3$};
% Text Node
\draw (336,204.4) node [anchor=north west][inner sep=0.75pt]    {$k=2$};
% Text Node
\draw (541,204.4) node [anchor=north west][inner sep=0.75pt]    {$k\geq 4$};
% Text Node
\draw (518,89.96) node [anchor=north west][inner sep=0.75pt]  [font=\Large,color={rgb, 255:red, 0; green, 0; blue, 0 }  ,opacity=1 ]  {$\textcolor[rgb]{0.94,0,0}{\overbrace{\ \ \ \ \ \ \ \ \ }^{k-4\ }}$};

\end{tikzpicture}
\end{figure}
%%%%%%%%%%%circles%%%%%%%%%%%%%%
\end{lemma}

Notice that the markings are compatible with the convention in \cref{invariantdef}.

\begin{proof}
Two rotations accomplish the $k=1$ case. The case $k=2$ is shown below, where we've already rotated the outer vertices. The step shown has one more rotation, and the result is then obtained by applying to the rightmost vertex the reverse of the second case in \cref{move3}.

%%%%%%%%%%%%%k2circles%%%%%%%%%%%%%%%%%%
\begin{figure}[H]
\centering
\includegraphics[scale=.8]{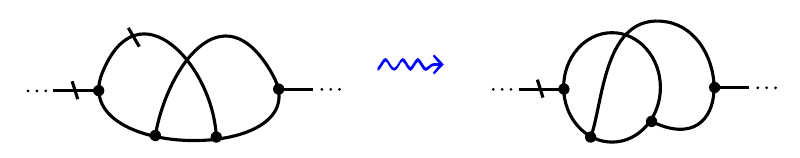}
\end{figure}
%%%%%%%%%%%%%%k2circles%%%%%%%%%%%%%%%%%

The case $k\geq 3$ is now shown, where we've already rotated the first and third outermost vertices on each side.
%%%%%%%%%%%%%kmorecircles%%%%%%%%%%%%%%%%%%
\begin{figure}[H]
\centering
\includegraphics[scale=.8]{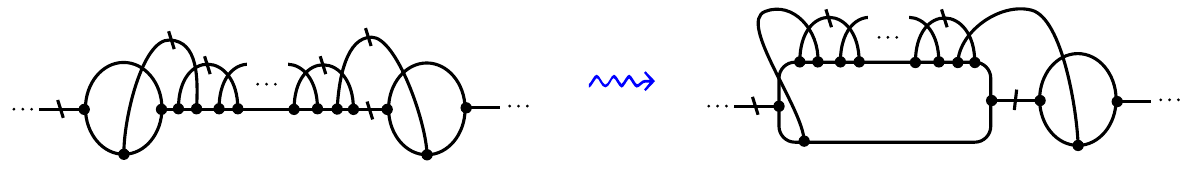}
\end{figure}
%%%%%%%%%%%%%%kmorecircles%%%%%%%%%%%%%%%%%
The step shown applies $2(k-3)$ times the reverse of first case in \cref{move3}, moving all vertices on the connecting 1-bone to the left circle. Now apply \cref{crosssimple}.
\end{proof}

\begin{remark}\label{alltogether}
Blocks can be connected together and \cref{crosssimple} applied to yield a single circular 0-bone. Note that in this process, any horizontal 1-bone connecting two circles becomes a marked vertical 1-bone. Univalent vertices on the ends can be removed by \cref{move1}.
\end{remark}

For example, as remarked earlier we can take $g$ Type 2 blocks, each with $k=2$, and $n$ Type 1 blocks. Call the resulting arboreal graph $G_{g,n}$, which is shown in \cref{wgn}. It is the arboreal graph associated to an arboreal surface $W_{g,n}$ which topologically is the oriented open surface with genus $g$ and $n+1$ punctures.\\

%%%%%%%%%%%%%wgn%%%%%%%%%%%%%%%%%
\begin{figure}
\centering

\tikzset{every picture/.style={line width=0.75pt}} %set default line width to 0.75pt        

\begin{tikzpicture}[x=0.75pt,y=0.75pt,yscale=-1,xscale=1]
%uncomment if require: \path (0,92); %set diagram left start at 0, and has height of 92

%Straight Lines [id:da5305331015591915] 
\draw [line width=1.5]    (21.01,79) -- (439.2,79) ;
\draw [shift={(439.2,79)}, rotate = 0] [color={rgb, 255:red, 0; green, 0; blue, 0 }  ][fill={rgb, 255:red, 0; green, 0; blue, 0 }  ][line width=1.5]      (0, 0) circle [x radius= 2.61, y radius= 2.61]   ;
\draw [shift={(21.01,79)}, rotate = 0] [color={rgb, 255:red, 0; green, 0; blue, 0 }  ][fill={rgb, 255:red, 0; green, 0; blue, 0 }  ][line width=1.5]      (0, 0) circle [x radius= 2.61, y radius= 2.61]   ;
%Curve Lines [id:da7114366241393925] 
\draw [line width=1.5]    (43.2,78.6) .. controls (43.2,39.6) and (81.2,41.6) .. (81.2,78.6) ;
\draw [shift={(81.2,78.6)}, rotate = 90] [color={rgb, 255:red, 0; green, 0; blue, 0 }  ][fill={rgb, 255:red, 0; green, 0; blue, 0 }  ][line width=1.5]      (0, 0) circle [x radius= 2.61, y radius= 2.61]   ;
\draw [shift={(43.2,78.6)}, rotate = 270] [color={rgb, 255:red, 0; green, 0; blue, 0 }  ][fill={rgb, 255:red, 0; green, 0; blue, 0 }  ][line width=1.5]      (0, 0) circle [x radius= 2.61, y radius= 2.61]   ;
%Curve Lines [id:da8661358609028664] 
\draw [line width=1.5]    (58.2,78.6) .. controls (58.2,39.6) and (96.2,41.6) .. (96.2,78.6) ;
\draw [shift={(96.2,78.6)}, rotate = 90] [color={rgb, 255:red, 0; green, 0; blue, 0 }  ][fill={rgb, 255:red, 0; green, 0; blue, 0 }  ][line width=1.5]      (0, 0) circle [x radius= 2.61, y radius= 2.61]   ;
\draw [shift={(58.2,78.6)}, rotate = 270] [color={rgb, 255:red, 0; green, 0; blue, 0 }  ][fill={rgb, 255:red, 0; green, 0; blue, 0 }  ][line width=1.5]      (0, 0) circle [x radius= 2.61, y radius= 2.61]   ;
%Curve Lines [id:da3868402108970791] 
\draw [line width=1.5]    (156.2,78.6) .. controls (156.2,39.6) and (194.2,41.6) .. (194.2,78.6) ;
\draw [shift={(194.2,78.6)}, rotate = 90] [color={rgb, 255:red, 0; green, 0; blue, 0 }  ][fill={rgb, 255:red, 0; green, 0; blue, 0 }  ][line width=1.5]      (0, 0) circle [x radius= 2.61, y radius= 2.61]   ;
\draw [shift={(156.2,78.6)}, rotate = 270] [color={rgb, 255:red, 0; green, 0; blue, 0 }  ][fill={rgb, 255:red, 0; green, 0; blue, 0 }  ][line width=1.5]      (0, 0) circle [x radius= 2.61, y radius= 2.61]   ;
%Curve Lines [id:da8280998056181875] 
\draw [line width=1.5]    (171.2,78.6) .. controls (171.2,39.6) and (209.2,41.6) .. (209.2,78.6) ;
\draw [shift={(209.2,78.6)}, rotate = 90] [color={rgb, 255:red, 0; green, 0; blue, 0 }  ][fill={rgb, 255:red, 0; green, 0; blue, 0 }  ][line width=1.5]      (0, 0) circle [x radius= 2.61, y radius= 2.61]   ;
\draw [shift={(171.2,78.6)}, rotate = 270] [color={rgb, 255:red, 0; green, 0; blue, 0 }  ][fill={rgb, 255:red, 0; green, 0; blue, 0 }  ][line width=1.5]      (0, 0) circle [x radius= 2.61, y radius= 2.61]   ;
%Curve Lines [id:da42247854994392087] 
\draw [line width=1.5]    (380.2,78.6) .. controls (380.2,39.6) and (418.2,41.6) .. (418.2,78.6) ;
\draw [shift={(418.2,78.6)}, rotate = 90] [color={rgb, 255:red, 0; green, 0; blue, 0 }  ][fill={rgb, 255:red, 0; green, 0; blue, 0 }  ][line width=1.5]      (0, 0) circle [x radius= 2.61, y radius= 2.61]   ;
\draw [shift={(380.2,78.6)}, rotate = 270] [color={rgb, 255:red, 0; green, 0; blue, 0 }  ][fill={rgb, 255:red, 0; green, 0; blue, 0 }  ][line width=1.5]      (0, 0) circle [x radius= 2.61, y radius= 2.61]   ;
%Curve Lines [id:da2691123113296098] 
\draw [line width=1.5]    (280.61,79) .. controls (280.61,40) and (318.61,42) .. (318.61,79) ;
\draw [shift={(318.61,79)}, rotate = 90] [color={rgb, 255:red, 0; green, 0; blue, 0 }  ][fill={rgb, 255:red, 0; green, 0; blue, 0 }  ][line width=1.5]      (0, 0) circle [x radius= 2.61, y radius= 2.61]   ;
\draw [shift={(280.61,79)}, rotate = 270] [color={rgb, 255:red, 0; green, 0; blue, 0 }  ][fill={rgb, 255:red, 0; green, 0; blue, 0 }  ][line width=1.5]      (0, 0) circle [x radius= 2.61, y radius= 2.61]   ;

% Text Node
\draw (112,52.4) node [anchor=north west][inner sep=0.75pt]    {$\cdots $};
% Text Node
\draw (334,51.4) node [anchor=north west][inner sep=0.75pt]    {$\cdots $};
% Text Node
\draw (53,3.4) node [anchor=north west][inner sep=0.75pt]  [font=\Large]  {$\overbrace{\ \ \ \ \ \ \ \ \ \ \ \ \ \ \ \ \ \ \ \ \ \ \ \ \ \ \ }^{g\ \text{pairs}}$};
% Text Node
\draw (289,9.4) node [anchor=north west][inner sep=0.75pt]  [font=\Large]  {$\overbrace{\ \ \ \ \ \ \ \ \ \ \ \ \ \ \ \ \ \ \ \ \ \ }^{n}$};

\end{tikzpicture}
\caption{The arboreal graph $G_{g,n}$, which is associated to an arboreal surface with genus $g$ and $n+1$ punctures.}
\label{wgn}
\end{figure}
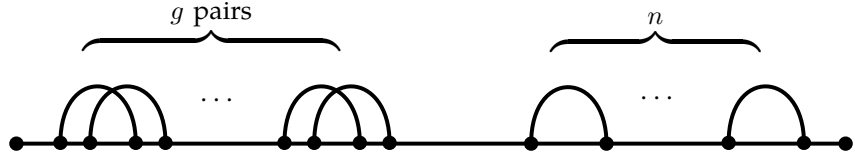
%%%%%%%%%%%%%%wgn%%%%%%%%%%%%%%%%%

\begin{proposition}We have $\mc L(W_{0,0})\simeq 0$. For $(g,n)\neq(0,0)$ we have $\mc L(W_{g,n}) \simeq  \holim D(G_{g,n}')$, where $G_{g,n}'$ is the following arboreal graph, with $\llbracket\cdot\rrbracket\in\{0,1\}$ the Iverson bracket.

%%%%%%%%%%%%%wgnprime%%%%%%%%%%%%%%%%%%
\begin{figure}[H]
\centering

\tikzset{every picture/.style={line width=0.75pt}} %set default line width to 0.75pt        

\begin{tikzpicture}[x=0.75pt,y=0.75pt,yscale=-1,xscale=1]
%uncomment if require: \path (0,154); %set diagram left start at 0, and has height of 154

%Rounded Rect [id:dp9555173963230276] 
\draw  [line width=1.5]  (15,81.16) .. controls (15,73.34) and (21.34,67) .. (29.16,67) -- (344.94,67) .. controls (352.76,67) and (359.1,73.34) .. (359.1,81.16) -- (359.1,123.64) .. controls (359.1,131.46) and (352.76,137.8) .. (344.94,137.8) -- (29.16,137.8) .. controls (21.34,137.8) and (15,131.46) .. (15,123.64) -- cycle ;
%Curve Lines [id:da3475197639514056] 
\draw [color={rgb, 255:red, 0; green, 183; blue, 0 }  ,draw opacity=1 ][line width=1.5] [line join = round][line cap = round]   (38.33,137.79) .. controls (38.1,122.8) and (4.1,56.8) .. (20.1,46.8) .. controls (36.1,36.8) and (44.1,53.8) .. (44.1,67.8) ;
\draw [shift={(44.1,67.8)}, rotate = 90] [color={rgb, 255:red, 0; green, 183; blue, 0 }  ,draw opacity=1 ][fill={rgb, 255:red, 0; green, 183; blue, 0 }  ,fill opacity=1 ][line width=1.5] [line join = round][line cap = round]     (0, 0) circle [x radius= 2.61, y radius= 2.61]   ;
\draw [shift={(38.33,137.79)}, rotate = 269.12] [color={rgb, 255:red, 0; green, 183; blue, 0 }  ,draw opacity=1 ][fill={rgb, 255:red, 0; green, 183; blue, 0 }  ,fill opacity=1 ][line width=1.5] [line join = round][line cap = round]     (0, 0) circle [x radius= 2.61, y radius= 2.61]   ;
%Straight Lines [id:da8717205220101305] 
\draw [color={rgb, 255:red, 240; green, 0; blue, 0 }  ,draw opacity=1 ][line width=1.5]    (61.5,68) -- (61.5,137.8) ;
\draw [shift={(61.5,137.8)}, rotate = 90] [color={rgb, 255:red, 240; green, 0; blue, 0 }  ,draw opacity=1 ][fill={rgb, 255:red, 240; green, 0; blue, 0 }  ,fill opacity=1 ][line width=1.5]      (0, 0) circle [x radius= 2.61, y radius= 2.61]   ;
\draw [shift={(61.5,68)}, rotate = 90] [color={rgb, 255:red, 240; green, 0; blue, 0 }  ,draw opacity=1 ][fill={rgb, 255:red, 240; green, 0; blue, 0 }  ,fill opacity=1 ][line width=1.5]      (0, 0) circle [x radius= 2.61, y radius= 2.61]   ;
%Straight Lines [id:da9673060822279435] 
\draw [color={rgb, 255:red, 240; green, 0; blue, 0 }  ,draw opacity=1 ][line width=1.5]    (54,97.91) -- (70.1,105.8) ;
%Curve Lines [id:da8303205201928796] 
\draw [color={rgb, 255:red, 0; green, 183; blue, 0 }  ,draw opacity=1 ][line width=1.5] [line join = round][line cap = round]   (188.33,137.79) .. controls (188.1,122.8) and (154.1,56.8) .. (170.1,46.8) .. controls (186.1,36.8) and (194.1,53.8) .. (194.1,67.8) ;
\draw [shift={(194.1,67.8)}, rotate = 90] [color={rgb, 255:red, 0; green, 183; blue, 0 }  ,draw opacity=1 ][fill={rgb, 255:red, 0; green, 183; blue, 0 }  ,fill opacity=1 ][line width=1.5] [line join = round][line cap = round]     (0, 0) circle [x radius= 2.61, y radius= 2.61]   ;
\draw [shift={(188.33,137.79)}, rotate = 269.12] [color={rgb, 255:red, 0; green, 183; blue, 0 }  ,draw opacity=1 ][fill={rgb, 255:red, 0; green, 183; blue, 0 }  ,fill opacity=1 ][line width=1.5] [line join = round][line cap = round]     (0, 0) circle [x radius= 2.61, y radius= 2.61]   ;
%Straight Lines [id:da14816258334718813] 
\draw [color={rgb, 255:red, 240; green, 0; blue, 0 }  ,draw opacity=1 ][line width=1.5]    (157.5,68) -- (157.5,137.8) ;
\draw [shift={(157.5,137.8)}, rotate = 90] [color={rgb, 255:red, 240; green, 0; blue, 0 }  ,draw opacity=1 ][fill={rgb, 255:red, 240; green, 0; blue, 0 }  ,fill opacity=1 ][line width=1.5]      (0, 0) circle [x radius= 2.61, y radius= 2.61]   ;
\draw [shift={(157.5,68)}, rotate = 90] [color={rgb, 255:red, 240; green, 0; blue, 0 }  ,draw opacity=1 ][fill={rgb, 255:red, 240; green, 0; blue, 0 }  ,fill opacity=1 ][line width=1.5]      (0, 0) circle [x radius= 2.61, y radius= 2.61]   ;
%Straight Lines [id:da22778386476785417] 
\draw [color={rgb, 255:red, 240; green, 0; blue, 0 }  ,draw opacity=1 ][line width=1.5]    (150,97.91) -- (166.1,105.8) ;
%Curve Lines [id:da6280772125445371] 
\draw [color={rgb, 255:red, 0; green, 183; blue, 0 }  ,draw opacity=1 ][line width=1.5] [line join = round][line cap = round]   (93.33,137.79) .. controls (93.1,122.8) and (59.1,56.8) .. (75.1,46.8) .. controls (91.1,36.8) and (99.1,53.8) .. (99.1,67.8) ;
\draw [shift={(99.1,67.8)}, rotate = 90] [color={rgb, 255:red, 0; green, 183; blue, 0 }  ,draw opacity=1 ][fill={rgb, 255:red, 0; green, 183; blue, 0 }  ,fill opacity=1 ][line width=1.5] [line join = round][line cap = round]     (0, 0) circle [x radius= 2.61, y radius= 2.61]   ;
\draw [shift={(93.33,137.79)}, rotate = 269.12] [color={rgb, 255:red, 0; green, 183; blue, 0 }  ,draw opacity=1 ][fill={rgb, 255:red, 0; green, 183; blue, 0 }  ,fill opacity=1 ][line width=1.5] [line join = round][line cap = round]     (0, 0) circle [x radius= 2.61, y radius= 2.61]   ;
%Straight Lines [id:da7989973817315498] 
\draw [color={rgb, 255:red, 240; green, 0; blue, 0 }  ,draw opacity=1 ][line width=1.5]    (116.5,68) -- (116.5,137.8) ;
\draw [shift={(116.5,137.8)}, rotate = 90] [color={rgb, 255:red, 240; green, 0; blue, 0 }  ,draw opacity=1 ][fill={rgb, 255:red, 240; green, 0; blue, 0 }  ,fill opacity=1 ][line width=1.5]      (0, 0) circle [x radius= 2.61, y radius= 2.61]   ;
\draw [shift={(116.5,68)}, rotate = 90] [color={rgb, 255:red, 240; green, 0; blue, 0 }  ,draw opacity=1 ][fill={rgb, 255:red, 240; green, 0; blue, 0 }  ,fill opacity=1 ][line width=1.5]      (0, 0) circle [x radius= 2.61, y radius= 2.61]   ;
%Straight Lines [id:da6405184156777699] 
\draw [color={rgb, 255:red, 240; green, 0; blue, 0 }  ,draw opacity=1 ][line width=1.5]    (109,97.91) -- (125.1,105.8) ;
%Straight Lines [id:da5725922536124526] 
\draw [color={rgb, 255:red, 240; green, 0; blue, 0 }  ,draw opacity=1 ][line width=1.5]    (333.5,68) -- (333.5,137.8) ;
\draw [shift={(333.5,137.8)}, rotate = 90] [color={rgb, 255:red, 240; green, 0; blue, 0 }  ,draw opacity=1 ][fill={rgb, 255:red, 240; green, 0; blue, 0 }  ,fill opacity=1 ][line width=1.5]      (0, 0) circle [x radius= 2.61, y radius= 2.61]   ;
\draw [shift={(333.5,68)}, rotate = 90] [color={rgb, 255:red, 240; green, 0; blue, 0 }  ,draw opacity=1 ][fill={rgb, 255:red, 240; green, 0; blue, 0 }  ,fill opacity=1 ][line width=1.5]      (0, 0) circle [x radius= 2.61, y radius= 2.61]   ;
%Straight Lines [id:da7579883258559611] 
\draw [color={rgb, 255:red, 240; green, 0; blue, 0 }  ,draw opacity=1 ][line width=1.5]    (326,97.91) -- (342.1,105.8) ;
%Straight Lines [id:da9145908415458772] 
\draw [color={rgb, 255:red, 240; green, 0; blue, 0 }  ,draw opacity=1 ][line width=1.5]    (267.5,68) -- (267.5,137.8) ;
\draw [shift={(267.5,137.8)}, rotate = 90] [color={rgb, 255:red, 240; green, 0; blue, 0 }  ,draw opacity=1 ][fill={rgb, 255:red, 240; green, 0; blue, 0 }  ,fill opacity=1 ][line width=1.5]      (0, 0) circle [x radius= 2.61, y radius= 2.61]   ;
\draw [shift={(267.5,68)}, rotate = 90] [color={rgb, 255:red, 240; green, 0; blue, 0 }  ,draw opacity=1 ][fill={rgb, 255:red, 240; green, 0; blue, 0 }  ,fill opacity=1 ][line width=1.5]      (0, 0) circle [x radius= 2.61, y radius= 2.61]   ;
%Straight Lines [id:da30719678239580595] 
\draw [color={rgb, 255:red, 240; green, 0; blue, 0 }  ,draw opacity=1 ][line width=1.5]    (260,97.91) -- (276.1,105.8) ;
%Straight Lines [id:da3109850929028837] 
\draw [color={rgb, 255:red, 240; green, 0; blue, 0 }  ,draw opacity=1 ][line width=1.5]    (289.5,68) -- (289.5,137.8) ;
\draw [shift={(289.5,137.8)}, rotate = 90] [color={rgb, 255:red, 240; green, 0; blue, 0 }  ,draw opacity=1 ][fill={rgb, 255:red, 240; green, 0; blue, 0 }  ,fill opacity=1 ][line width=1.5]      (0, 0) circle [x radius= 2.61, y radius= 2.61]   ;
\draw [shift={(289.5,68)}, rotate = 90] [color={rgb, 255:red, 240; green, 0; blue, 0 }  ,draw opacity=1 ][fill={rgb, 255:red, 240; green, 0; blue, 0 }  ,fill opacity=1 ][line width=1.5]      (0, 0) circle [x radius= 2.61, y radius= 2.61]   ;
%Straight Lines [id:da9691871363605882] 
\draw [color={rgb, 255:red, 240; green, 0; blue, 0 }  ,draw opacity=1 ][line width=1.5]    (282,97.91) -- (298.1,105.8) ;

% Text Node
\draw (126,92.4) node [anchor=north west][inner sep=0.75pt]    {$\cdots $};
% Text Node
\draw (302,92.4) node [anchor=north west][inner sep=0.75pt]    {$\cdots $};
% Text Node
\draw (14,0.4) node [anchor=north west][inner sep=0.75pt]  [font=\Large,color={rgb, 255:red, 0; green, 183; blue, 0 }  ,opacity=1 ]  {$\overbrace{\ \ \ \ \ \ \ \ \ \ \ \ \ \ \ \ \ \ \ \ \ \ \ \ \ \ \ \ \ \ \ \ \ \ \ \ \ }^{g}$};
% Text Node
\draw (262,22.4) node [anchor=north west][inner sep=0.75pt]  [font=\Large,color={rgb, 255:red, 240; green, 0; blue, 0 }  ,opacity=1 ]  {$\overbrace{\ \ \ \ \ \ \ \ \ \ \ \ \ \ \ \ }^{n-\llbracket g=0\rrbracket}$};

\end{tikzpicture}
\label{wgnprime}
\end{figure}
%%%%%%%%%%%%%%wgnprime%%%%%%%%%%%%%%%%%

\end{proposition}

We have indicated the shifts explicitly, but they are redundant given that our diagram is defined by using the $G\mapsto D(G)$ operation in \cref{invariantdef}. Note that $D(G_{g,n}')$ has $g$ unmarked and $g+n-1$ marked 1-bones.

\begin{proof}
We start with $G_{g,n}$ shown in \cref{wgn}. For $g=n=0$ we have an open disk and $G_{0,0}$ is a single contractible 0-bone, so the result follows by \cref{simplify}. The other cases follow from applying \cref{circleslem} and following \cref{alltogether}.
\end{proof}

\begin{remark}\label{cylinder} The case of $W_{0,1}$ is somewhat degenerate. The graph $G'_{0,1}$ should be thought of as two vertices connected by two edges. Thus it is not arboreal, but we can still put a diagram on $Q=Q(G_{0,1})$. By removing the univalent vertices of $G_{0,1}$, it's clear that $D(G'_{0,1})$ should assign $\id_{\mc A}$ to all arrows in $Q$. Then $\mc L(W_{0,1})$ is the (ordinary) equalizer of the dg functors $\pi_{1},\pi_{2} :P(\mc A)\lr \mc A$. Thus the objects are homotopy autoequivalences.
\end{remark}

If $(g,n)\notin\{(0,0),(0,1)\}$, then $G'_{g,n}$ is a legitimate arboreal graph. There are many cospans in $D(G'_{g,n})$ which involve the fibrations $\rh_{i}:\mc A^{\to}\lr \mc A$. Objects in the homotopy limit then involve many homotopy equivalences, but we can improve this situation using \cref{simplify}. In particular, a spanning tree of $G'_{g,n}$ is given by the 0-bone with one edge subtracted. So we can take all but one cospan of the fibrations $\rh_{i}$ to yield on-the-nose equalities. 

Equivalently, we can first assume that all $\rh_{i}$ cospans are equalities, in which case we can represent every black edge in $G'_{g,n}$ with a single object of $\mc A$, together arranged in a circle. We connect these objects with arrows in the directions indicated by the functors in the cospans, or equivalently by the markings of the 1-bones\footnote{Given our convention on arboreal graphs, this can also be stated as: if a joint is attached from the top (resp. bottom), the corresponding arrow points to the right (resp. left). Compare with Figure 14.}. This yields a quiver whose underlying graph is the circle with 2(2g+n-1) vertices. Choose one of these vertices to ``open up'', replacing the equality it represents with a homotopy equivalence. 

In the following we will not explicitly write down morphism spaces or composition or differential rules, as these follow the same pattern we've been using throughout.\\

\begin{example} For completeness, notice that the right quiver for $W_{0,1}$ is one vertex with a single arrow making a loop. Then necessarily this arrow must be the homotopy equivalence, and we recover the result in \cref{cylinder}.\\
\end{example}

\begin{example} The case $(g,n)=(0,2)$ yields the quiver with two vertices and two arrows in opposite directions. One vertex must be chosen to not be an equality of objects of $\mc A$, but either choice yields the same category. Either way, objects in $\mc L(W_{0,2})$ are representations $A_{0}\overset{a_{0}}\to A_{1}\overset{a_{1}}\to A_{2}$, along with homotopy equivalences $A_{0}\simeq A_{2}$ and $C(a_{0})\simeq C(a_{1})[1]$. \\
\end{example}

\begin{example} The case $(g,n)=(1,0)$ yields the quiver with two vertices and two arrows in the same direction. Now the choice of vertex makes a difference, and we get two equivalent presentations for $\mc L(W_{1,0})$. Objects can be taken to be representations $A_{0}\overset{a_{0}}\to A_{1}\overset{a_{1}}\leftarrow A_{2}$ or $A_{0}\overset{a_{0}}\leftarrow A_{1}\overset{a_{1}}\to A_{2}$, in both cases with homotopy equivalences $A_{0}\simeq A_{2}$ and $C(a_{0})\simeq C(a_{1})$.\\
\end{example}

%%%%%%%%%%%%%33circles%%%%%%%%%%%%%%%%%%
\begin{figure}
\centering
\includegraphics[scale=.8]{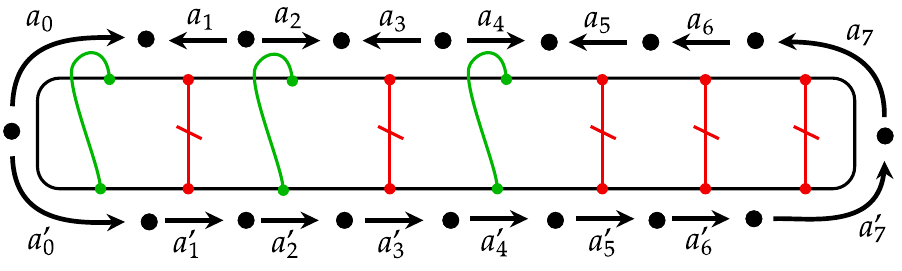}
\label{33circles}
\caption{An object of $\mc L(W_{g,n})$ can be represented in the following way. Open up the shown quiver at a fixed vertex to make a line, represent this quiver in objects of $\mc A$ and supply homotopy equivalences between the ends of this representation and between the cones of the maps (with shifts as appropriate).}
\end{figure}
%%%%%%%%%%%%%%33circles%%%%%%%%%%%%%%%%%

\begin{example} With all small cases dealt with, we show a somewhat larger case in order to illustrate the pattern. A good size is given by $(g,n)=(3,3)$. Here we draw the associated quiver over the graph $G'_{3,3}$, with the arrows in the direction determined the way the corresponding 1-bone is attached. Figure 14 means that an object in $\mc L(W_{3,3})$ is a representation of the black quiver, with \textit{one specified vertex} opened up into a homotopy equivalence, along with homotopy equivalences $C(a_{i})\simeq C(a_{i})[n_{i}]$, where $n_{i}$ is $0$ if the 1-bone connecting $a_{i}$ and $a_{i}'$ is unmarked, and 1 if marked. For example, instead of the maps $a_{0}$ and $a_{0}'$ having a common codomain $A_{0}$, this vertex could be replaced by a homotopy equivalence $A_{0}\simeq A_{0}'$. As before, different choices of this specified vertex yield representations of different linear quivers, with the data of a homotopy equivalence  between the two ends.
\end{example}

	\bibliography{references}

\begin{thebibliography}{Nad17}

\bibitem[AGEN]{A}
Daniel Alvarez-Gavela, Yakov Eliashberg, and David Nadler.
\newblock Positive arborealization of polarized Weinstein manifolds.
\newblock \texttt{arXiv:2011.08962}.

\bibitem[B73]{B}
Kenneth S. Brown.
\newblock Abstract Homotopy Theory and Generalized Sheaf Cohomology,
\newblock {\em Trans. Am. Math. Soc.},186: 419--458, 1973

\bibitem[CE12]{CE}
Kai Cieliebak and Yakov Eliashberg. 
\newblock From Stein to Weinstein and back
\newblock {\em Am. Math. Soc. Coll. Pub.}, 59, American Mathematical Society, Providence, RI, 2012. Symplectic geometry of affine complex manifolds.

\bibitem[DHKS04]{DHKS}
W. G. Dwyer, P. S. Hirschhorn, D. M. Kan, and J. H. Smith.
\newblock Homotopy limit functors on model categories and homotopical categories
\newblock {\em Math. Surveys and Monographs}, 113, American Mathematical Society, Providence, RI, 2024

\bibitem[DK18]{DK}
Tobias Dyckerhoff and Mikhail Kapranov.
\newblock Triangulated surfaces in triangulated categories.
\newblock {\em J. Eur. Math. Soc.}, 20(6):1473--1524, 2018

\bibitem[DS95]{DS}
William G. Dwyer and Jan Spalinski
\newblock Homotopy theories and model categories. 
\newblock {\em Handbook of algebraic topology}.
\newblock {\em Elsevier Science B.V}, 73--126, 1995

\bibitem[Fuk93]{F}
Kenji Fukaya.
\newblock Morse homotopy, $A_\infty$-category, and Floer homologies.
\newblock {\em Proceedings of GARC Workshop on Geometry and Topology '93.}, 1993

\bibitem[GM03]{GM}
Sergei I. Gelfand and Yuri I. Manin
\newblock {Methods of Homological Algebra}
\newblock{\em Springer Monographs in Math.}, Springer-Verlag Berlin, Heidelberg, 2003

\bibitem[H03]{Hir}
Philip S. Hirschhorn.
\newblock Model categories and their localizations.
\newblock {\em American Mathematical Society}, 2003

\bibitem[H99]{H}
Mark Hovey.
\newblock {\em Model Categories}.
\newblock {\em Mathematical Surveys and Monographs}, (63), 1999

\bibitem[Kar]{K}
Do\u{g}ancan Karaba\c{s}.
\newblock Microlocal Sheaves on Pinwheels.
 \newblock  \texttt{arXiv:1810.09021}.
 
 \bibitem[Kel]{Keller}
Bernhard Keller.
\newblock On differential graded categories.
\newblock \texttt{arXiv:math/0601185}.

\bibitem[Kon]{Kon}
Maxim Kontsevich.
\newblock Symplectic geometry of homological algebra.
\newblock \texttt{ihes.fr/$\sim\!\!$ maxim/publicationsfrancais.html}, 2009

\bibitem[KS90]{KS2}
Masaki Kashiwara and Pierre Schapira.
\newblock {\em Sheaves on manifolds}.
\newblock {\em Grundlehren Math. Wiss.}, (292), 1990

\bibitem[Lur09]{L1}
Jacob Lurie.
\newblock {\em Higher Topos Theory}.
\newblock {\em Princeton University Press.}, 2009

\bibitem[Lur17]{L2}
Jacob Lurie.
\newblock Higher Algebra.
\newblock \texttt{math.ias.edu/$\sim\!\!$ lurie/papers/HA.pdf}, 2017

\bibitem[Nad16]{N2}
David Nadler. 
\newblock Wrapped microlocal sheaves on pairs of pants. 
\newblock \texttt{arXiv:1604.00114}.

\bibitem[Nad17]{N1}
David Nadler.
\newblock Arboreal singularities.
\newblock {\em Geom. Topol.}, 21(2):1231--1274, 2017

\bibitem[Nad]{N2}
David Nadler.
\newblock Non-characteristic expansion of {L}egendrian singularities.
  \newblock \texttt{arXiv:1507.01513}.

\bibitem[NS]{NS}
David Nadler and Vivek Shende.
\newblock Sheaf quantization in Weinstein symplectic manifolds.
\newblock \texttt{arXiv:2007.10154}.

\bibitem[R14]{R}
Emily Riehl.
\newblock Categorical homotopy theory.
\newblock {\em Cambridge University Press}, 2014

\bibitem[RV22]{RV}
Emily Riehl and Dominic Verity.
\newblock Elements of $\infty$-Category Theory.
\newblock {\em Cambridge University Press}, 2022

\bibitem[Sei08]{Sei}
Paul Seidel.
\newblock Fukaya Categories and Picard-Lefschetz Theory.
\newblock{European Math. Soc}, Zurich Lectures in Advanced Mathematics, 2008

\bibitem[She21]{Sh}
Vivek Shende.
\newblock Microlocal Category for Weinstein Manifolds via the h-Principle.
\newblock {\em Publ. Res. Inst. Math. Sci.} 57(3/4):1041--1048, 2021.

\bibitem[Sta18]{S}
Laura Starkston.
\newblock Arboreal singularities in Weinstein skeleta.
\newblock {\em Selecta Math. (N.S.)}, 24(5):4105--4140, 2018

\bibitem[Tab05]{T}
Gon\c{c}alo Tabuada.
\newblock Une structure de cat\'egorie de mod\`eles de Quillen sur la cat\'egorie des dg-cat\'egories.
\newblock{\em Comptes Rendus. Math\'ematique.} 340(2):15--19, 2005


\end{thebibliography}
	\bibliographystyle{alpha}

\end{document}